\documentclass[a4paper,reqno]{amsart}
\usepackage{graphicx} 
\usepackage{subcaption}
\usepackage{multirow}
\usepackage{xcolor}
\usepackage{amsmath}
\usepackage{amsfonts}
\usepackage{amsthm}
\usepackage{float}
\usepackage{geometry}
\usepackage[normalem]{ulem}
\usepackage{algorithm} 
\usepackage{algpseudocode}
\usepackage{tikz}  
\usetikzlibrary{matrix,arrows,decorations.pathmorphing}  

\usepackage{booktabs}
\usepackage{array}
\usepackage{algorithm}
\usepackage{algpseudocode}

\usepackage{url,hyperref}
\hypersetup{
    citecolor=cyan,
    colorlinks=true,
    linkcolor=blue,
    filecolor=magenta,      
    urlcolor=cyan,
    pdftitle={Overleaf Example},
    pdfpagemode=FullScreen,
    }
\newtheorem{theorem}{Theorem}[section]

\newtheorem{proposition}[theorem]{Proposition}
\newtheorem{lemma}[theorem]{Lemma}

\newtheorem{formula}[theorem]{Formula}
\newtheorem{remark}{Remark}[section]

\numberwithin{equation}{section}

\usepackage[foot]{amsaddr}

\newcommand{\N}{\mathbb{N}}

\newcommand{\R}{\mathbb{R}}
\newcommand{\C}{\mathbb{C}}
\newcommand{\I}{{\rm i}}

\newcommand{\dd}{{\rm d}}

\newcommand{\M}{\mathcal{M}}
\newcommand{\hil}{\mathcal{H}}
\newcommand{\quadr}{\mathcal{Q}}
\newcommand{\siegel}{\mathbb{H}}

\newcommand{\trans}{\top}

\newcommand{\hquad}{\hspace{0.5em}}

\newcommand{\kw}[1]{\noindent\textbf{Keywords: }#1
}
\newcommand{\msc}[1]{\noindent\textbf{Mathematics Subject Classification (2020): }#1}

\usepackage{cleveref}
\usepackage{biblatex}
\title[Residual-Based Time Discretization on Nonlinear Approximation Manifolds]{Residual-Based Time Discretization on \\ Nonlinear Approximation Manifolds: \\ Analysis and Gaussian Applications}

\author{Eddy de Le\'on$^{\ast}$}
\thanks{$^{\ast}$Department of Mathematics,
Technical University of Munich,
85748 Garching, Germany, Email: eddy.de.leon@tum.de}

\author{Caroline Lasser$^{\dagger}$}
\thanks{$^{\dagger}$Department of Mathematics,
Technical University of Munich,
85748 Garching, Germany, Email: classer@tum.de}

\begin{document}

\begin{abstract}
    We study time-discrete parametric approximations of evolution equations in Hilbert spaces based on residual minimization. The solution is represented by a parametrized ansatz belonging to a low-dimensional nonlinear manifold, and time stepping is performed by minimizing suitably defined residuals at each step. Two natural residual formulations are considered: discretization followed by parametrization of the evolution equation, and discretization of the Dirac--Frenkel variational principle governing the parameter dynamics.
A unified error analysis is developed for both approaches within the family of $\zeta$-methods. The resulting bounds separate the effects of time discretization from those of residual minimization and yield first- and second-order convergence under Lipschitz, one-sided Lipschitz, and dissipativity assumptions. For the variational formulation, additional stability conditions involving the conditioning of the parametrization map arise naturally.
The framework is applied to Gaussian approximation manifolds, for which residual norms and gradients admit explicit closed-form expressions when polynomial operators are involved. This enables efficient implementation without spatial discretization. Numerical experiments for time-dependent Schr\"odinger equations illustrate the theoretical convergence rates and the influence of residual accuracy on conservation properties.
\end{abstract}

\maketitle

\bigskip
\msc{Primary 65M12; Secondary 65M15, 65L20, 41A30, 65D15}

\bigskip
\kw{parametric evolution equations in Hilbert spaces; nonlinear manifold approximation; residual-based time stepping; convergence analysis; Gaussian manifold}

\section{Introduction}

The numerical solution of time-dependent partial differential equations in high dimensions remains a central challenge in scientific computing. In particular, evolution equations of Schr\"odinger type arise ubiquitously in quantum dynamics, molecular physics, and wave propagation, where the solution typically exhibits oscillatory behavior, strong localization, and multi-scale structure. Standard grid-based discretizations quickly become infeasible as the
spatial dimension increases, whereas reduced-order approaches must carefully balance accuracy, stability, and computational cost.

\medskip
We consider time-dependent evolution equations posed in the Hilbert space 
$\hil = L^2(\Omega,\C)$ with a general right-hand side operator 
$f:[0,T]\times \Omega \times \hil \to \C$ of the form
\begin{equation} \label{eq:PDE_original}
\left\{
\begin{array}{ll}
\partial_t \psi(t,x) = f(t,x,\psi), & (t,x)\in [0,T]\times \Omega, \\[1ex]
\psi(0,x) = \psi_0(x), & x\in \Omega .
\end{array}
\right.
\end{equation}
The initial datum $\psi_0\in \hil$ may, for instance, be normalized with $\|\psi_0\|=1$. 
The right-hand side $f(t,\cdot,\psi)\in \hil$ may represent a time-dependent Schr\"odinger operator, 
\[
f(t,x,\psi) = \frac{1}{\I}\left(-\frac12 \Delta_x \psi(t,x) + V(t,x)\psi(t,x)\right),
\quad (t,x)\in [0,T]\times \Omega,
\]
or, more generally, a linear partial differential operator of the form
\[
f(t,x,\psi) = \sum_{\alpha} c_{\alpha}(t,x)\,
\frac{\partial^{|\alpha|}}{\partial x^\alpha}\psi(t,x),
\quad (t,x)\in [0,T]\times \Omega.
\]
Nonlinear evolution operators are likewise admissible, provided that the resulting
initial value problem is well posed in $\hil$, and in particular allows for some type 
of (one-sided) Lipschitz control.

\medskip
A powerful class of methods approximates the solution $\psi(t,\cdot)$ within a low-dimensional nonlinear manifold $\mathcal M$ embedded in $\hil$. Rather than resolving the full spatial domain, one represents the solution in the form
\[
\psi(t,\cdot) \approx u(\theta(t),\cdot),
\]
where the parameter vector $\theta(t)\in\Theta$ evolves in time, for a complex parameter space $\Theta$. Such parametrized ansatz spaces include Gaussian wave packets \cite{Heller_1976,Richings_etal2015}, Hagedorn bases \cite{Hagedorn1998,CL1}, neural networks  \cite{Raissi_etal2019,Gutierrez_etal2022,Bruna_etal2024,Datar_etal2026} also with Gaussian activation functions \cite{Anderson_etal2024,BCM,Dupuy_etal2025}, and more general manifolds. 
These representations can capture localization, phase information,
and deformation using comparatively few degrees of freedom. 

\medskip
A central mathematical framework underlying many parametrized evolution schemes is the Dirac--Frenkel variational principle \cite{Lu}, which projects the infinite-dimensional dynamics onto the tangent space of the approximation manifold. This yields a system of ordinary differential equations for the parameters $\theta(t)$ obtained through a linear least-squares condition. While the continuous-time formulation enjoys attractive geometric properties, including conservation of norm and energy in Hamiltonian systems \cite{LS}, its numerical implementation suffers from the so-called matrix singularity problem \cite{KAY1989} or tangent space collapse \cite{ZCVP}. 
An alternative viewpoint reverses the order of discretization and parametrization. One first discretizes the original evolution equation in time and subsequently parametrizes within the manifold each time step by minimizing a residual in the Hilbert space norm. This ``discretize-then-parametrize'' philosophy leads naturally to a sequence of nonlinear optimization problems for the parameters \cite{kvaal2023,schrader2026}. Both approaches are actively used in practice and give rise to different numerical schemes with distinct stability and convergence properties. Despite their widespread use, a unified mathematical understanding of these parametric time-stepping strategies remains incomplete,
and error bounds that account simultaneously for time discretization and manifold approximation seem to be rare. Notable exceptions are \cite{ZCVP}, which motivated the present work, the analysis of regularized schemes \cite{FLL,LN}, and \cite{BCM} for a 
continuous-time analysis of dynamical sampling \cite{BCM}. 

\medskip
Here, we aim at a unified framework for the analysis of time-discrete parametrized approximations of evolution equations in Hilbert spaces. We consider nonlinear approximation manifolds of the form
\[
\mathcal M = \{ u(\theta,\cdot)\in\hil \ : \ \theta \in \Theta \subset \mathbb C^p, p\geq 1 \},
\]
and formulate numerical schemes through the minimization of residual functions at each time step. Two natural residual structures arise:
(i) discretization--then--parametrization residuals, obtained by first applying a time discretization scheme within the family of $\zeta$-methods and subsequently optimizing the parameters on the manifold; 
(ii) parametrization--then--discretization residuals, derived from time discretizations of the Dirac--Frenkel variational principle.
Both approaches lead to iterative schemes of the form
\begin{align*}
&\theta_{k+1} = \arg\min_{\theta\in\Theta} \| r_k(\theta,\cdot) \|_{\mathcal H},\\
&u(\theta_{k+1},\cdot) \approx \psi(t_k,\cdot),\quad k\ge 0,
\end{align*}
where $r_k(\theta,\cdot)$ denotes the corresponding residual at the $k$th time step.

\medskip
Our first main contribution is a unified error analysis for both residual formulations without assumptions on the underlying approximation manifold, see \Cref{prop:discr-opt} and \Cref{prop:opt-discr}. We show that the global approximation error decomposes into a residual contribution arising from the manifold optimization and a classical time discretization contribution governed by the regularity of the exact solution and the choice of $\zeta$-scheme. This yields explicit convergence rates of order $h$ or $h^2$ under natural Lipschitz, one-sided Lipschitz, and dissipativity assumptions. For the parametrization--then--discretization approach, additional conditioning terms involving the singular values of the parametrizing gradient appear naturally and clarify stability requirements. A second contribution concerns the explicit evaluation of residual norms and their gradients for Gaussian approximation manifolds. When the evolution operator acts polynomially on Gaussian functions, all inner products required for residual minimization admit closed-form expressions in terms of Gaussian moments. We present compact formulas for these integrals and their parameter derivatives, which pave the way to efficient computation even in higher dimensions. When closed forms are unavailable, we briefly discuss quadrature-based approximations and their accuracy. This analytic tractability highlights an important advantage of Gaussian manifolds: spatial discretization can be avoided to a large extent, and the computational effort is concentrated on low-dimensional optimization problems.

\medskip
To illustrate the theoretical results, we present numerical experiments for one-dimensional Schr\"odinger equations with harmonic, double-well, and hyperbolic cosine potentials. These examples probe regimes of exact representability, dynamical interference, and oscillatory wave packet motion. The experiments confirm the predicted convergence orders and demonstrate how residual minimization accuracy directly influences accuracy and conservation behavior. Although our numerical investigations focus on Schr\"odinger dynamics, the framework is not restricted to conservative
evolution equations. To illustrate this broader applicability, Appendix~\ref{sec:FP_eq} briefly discusses a Fokker--Planck example.

\medskip
More generally, the framework developed here connects variational approximation theory, time discretization of evolution equations, and nonlinear optimization within a single analytical structure. While Gaussian manifolds and time-dependent Schr\"odinger equations serve as a concrete and computationally attractive example, the error analysis applies to general nonlinear parametrizations and evolution equations. 

\subsection*{Structure of the paper}
\Cref{sec:pda} introduces the abstract parametrized approximation framework and the residual minimization formulation. \Cref{sec:residual_fn} derives explicit residuals corresponding to discretize-then and parametrize-then strategies within the family of $\zeta$-methods. \Cref{sec:error} develops the unified error analysis. \Cref{sec:gaussian_formulas} presents closed-form Gaussian integral formulas, quadrature approximations for residual evaluation, and details of the initialization scheme. \Cref{sec:Num_Experiments} illustrates the dynamics of the chosen examples. \Cref{sec:Num_Experiments_II} reports numerical experiments illustrating the theoretical results. Conclusions are given in \Cref{sec:concl}.


\section{Parametrized dynamical approximation} \label{sec:pda}

We work in an ambient Hilbert space $\hil = L^2(\Omega,\C)$ with $\Omega\subseteq\R^d$ and denote the inner product and induced norm by
$\langle\cdot,\cdot\rangle_\hil$ and $\Vert\cdot\Vert_\hil$, respectively. For solving the evolutionary equation \eqref{eq:PDE_original}, we introduce a time grid $t_k = k h$ of uniform step-size $h>0$ and aim at an optimal parameter sequence $(\theta_k)_{k\ge0}\subset\Theta$ from the parameter set $\Theta\subset\C^p$ with $p>1$. We aim at the approximation of the PDE solution, $\psi(t_k,x)\approx u(\theta_k,x)$ for all time steps $k\ge0$.
Given a residual function
\begin{align*}
    r_k: \Theta\times\Omega \to \C, \quad
          (\theta,x) \mapsto r_k(\theta,x),
\end{align*}
parametrized by a `known' parameter $\theta_k$ at time step $t_k$, we define the cost function
\begin{equation} 
    F_k: \Theta\subset\C^p \to \R, \quad
    \theta \mapsto \Vert r_k(\theta,\cdot) \Vert^2_{\hil}.
\end{equation}
This is the function to be minimized in the parameter space $\Theta\subset\C^p$ to determine the next point in the discrete trajectory of approximate solutions. Namely, given an initial condition $\psi_0\in\hil$, we determine the approximation $u(\theta_0,\cdot)\in\M$ and then compute the subsequent values $u(\theta_{k+1},\cdot)\in\M$ iteratively via the optimization sequence
\begin{equation*}
    \theta_{k+1} = \arg \min_{\theta\in\Theta} F_k(\theta),\quad k=0,1,\ldots
\end{equation*}
We will consider and analyze two specific choices for the residual function $r_k(\theta,\cdot)$ later in \Cref{sec:residual_fn}.

\subsection{Abstract discretization process} \label{sec:discretization_process}
Given a choice of step-size $h>0$, the required number of steps $N$ in the interval $[0,T]$ is given by $N=\lceil T/h \rceil$. For the time grid $t_0,\ldots,t_N$, we perform the following process to determine the discrete solutions 
$u(\theta_0,\cdot),\ldots,u(\theta_N,\cdot)$:
\begin{itemize}
    \item[(i)] Given the initial condition $\psi_0\in\hil$, obtain $u_0\in\M$ via 
    $$ \theta_{0} = \arg\min_{\theta\in\Theta} \Vert \psi_0(\cdot) - u(\theta,\cdot) \Vert_\hil $$
    \item[(ii)] Construct the cost function $F_k(\theta)=\Vert r_k(\theta,\cdot)\Vert_\hil^2$ for $k=0$.
    \item[(iii)] Minimize $F_k$ and assign the optimum to $\theta_{k+1}$ for $k=0$.
    \item[(iv)] Repeat steps (ii) and (iii) for $k=1,2,\ldots$ to obtain $\theta_2,\theta_3,\ldots$
    \item[(v)] Compute $u(\theta_k,\cdot)$ at the values $\theta_k$ determined in steps (i)--(iv) to provide the discrete approximations $\psi(t_k,\cdot)\approx u(\theta_k,\cdot)$ for all $k=0,1,\ldots,N$
\end{itemize}

\bigskip
See also \Cref{alg:residual} for a more algorithmic formulation of the process. Our main focus here will be on the comparative analysis of several natural choices of residual optimization, which result from different orderings of time discretization and parameter optimization, as well as a different choice of the time-integration within the class of $\zeta$-methods. 

\bigskip
\begin{remark}
We note that the optimizations in steps (i) and (iii) are non-trivial processes in general, which we treat in 
a black box fashion. 
Any computed parameter value $\theta_{k+1}$ is useful in the present
context, since the resulting approximation error is estimated in
terms of the associated residual norm.
\end{remark}

\begin{algorithm}[ht]
\caption{Residual-based parametric $\zeta$-method}
\label{alg:residual}
\begin{algorithmic}[1]

\Require Initial parameter $\theta_0\in\Theta$, final time $T>0$, step-size $h>0$
\State $N \gets \lceil T/h\rceil$

\For{$k=0,\ldots,N-1$}

    \State Construct the residual function
    $r_k(\theta,\cdot)$
    according to either the discretize-then formulation~\eqref{eq:residual_Disc-Opt}
    or the parametrize-then formulation~\eqref{eq:residual_Opt-Disc}

    \State Define the cost function
    $F_k(\theta)=\|r_k(\theta,\cdot)\|_{\mathcal H}^2$

    \State Compute $\theta_{k+1} = \arg\min_{\theta\in\Theta}F_k(\theta)$

    \State Set $u_{k+1}=u(\theta_{k+1},\cdot)$

\EndFor

\Ensure Approximate trajectory $\{u(\theta_k,\cdot)\}_{k=0}^{N}$

\end{algorithmic}
\end{algorithm}

\subsection{Initialization}\label{sec:initial}
The first step in the entire process is determining
the appropriate parameter $\theta_0\in\Theta$ that approximates the initial condition $\psi_0\in\hil$. To do this, we define the cost function
\begin{equation} \label{eq:cost_Fn0}
    \begin{split}
        F(\theta) &= \Vert \psi_0 - u(\theta,\cdot) \Vert^2_\hil, \\
      &= \langle \psi_0,\psi_0 \rangle_\hil - 2 \Re \langle\psi_0, u(\theta,\cdot)\rangle_\hil + \langle u(\theta,\cdot),u(\theta,\cdot)\rangle_\hil.
    \end{split}
\end{equation}
For example, we might consider a univariate mixture
\[
u(\theta,x) = \sum_{l=1}^L b(\theta^l,x)
\]
with holomorphically parametrized complex Gaussians
\begin{equation} \label{eq:gaussian_basis}
b(\theta^l,x) = \exp(-a^l x^2 + \kappa^l x + \gamma^l),\quad x\in\R,   
\end{equation}

with $a\in \siegel = \{z\in\C: \Re(z) >0\}$, $\kappa\in\C$, $\gamma\in\C$ and total parameter $\theta=(a,\kappa,\gamma)$, see also \Cref{sec:gaussian_formulas}. Then, the gradient of the cost function satisfies
\begin{align*}
    \nabla_{\theta} F(\theta) = - \langle \psi_0,\nabla_\theta u(\theta,\cdot) \rangle_\hil + \langle u(\theta,\cdot),\nabla_\theta u(\theta,\cdot) \rangle_\hil .
\end{align*}
In this specific situation, we employ a heuristic five-step strategy to determine a minimizer 
$\theta_0 = \arg\min_{\theta\in\Theta}F(\theta)$ with $\theta_0=(a_0^l,\kappa_0^l,\gamma_0^l)_{l=1,\ldots,L}$,  
see \Cref{sec:init_Gauss} for further details.

\subsection{Optimization Process} \label{sec:Optimization_process}
At each time step, an optimization problem is solved. 
The parameter $\theta_{k+1}$ is defined as a minimizer of $F_k$, namely,
\begin{equation*}
    \theta_{k+1} = \arg\min_{\theta\in\Theta} F_k(\theta) ,
\end{equation*}
where we recall $F_k(\theta)=\Vert r_k(\theta,\cdot) \Vert^2_{\hil}$.
Since a global minimizer may not exist or may be computationally
inaccessible, in particular with a constraint like that for the width matrix 
$A\in\siegel^d$ of a Gaussian basis function, 
we only expect to obtain a parameter value $\theta_{k+1}$ for which
the residual is sufficiently small in a small neighborhood of the starting point $\theta_k$ for every step $k$. We emphasize that local minimizers with small enough residual will be our target here, since the overall approximation scheme has inherent time discretization, initialization, and quadrature errors.  
For the numerical experiments presented in \Cref{sec:Num_Experiments}, we perform unconstrained optimization in $\C^p$ using the Broyden–Fletcher–Goldfarb–Shanno (BFGS) algorithm. This algorithm searches for the local minimum until one of the following stopping criteria is met: $\Vert \nabla F_k(\theta) \Vert_{\infty} < \varepsilon_1$, $|F_k(\theta^{(n)})-F_k(\theta^{(n-1)})|<\varepsilon_2$ or $\Vert \theta^{(n)}-\theta^{(n-1)}\Vert_2 < \varepsilon_3$ for $n=1,2,\ldots$, given predefined accuracies $\varepsilon_1,\varepsilon_2,\varepsilon_3>0$. If none of these criteria is met, we terminate the computations if a certain maximum number of steps $n_{\max}$ is reached.
This method is available in various optimization packages, and the solver's efficiency is significantly improved (especially in higher dimensions) if the explicit form of the gradient $\nabla_\theta F_k(\theta)$ is provided. We discuss such an explicit form in \Cref{sec:sum}.

The numerical experiments shown in \Cref{sec:Num_Experiments,sec:Num_Experiments_II} were carried out using unconstrained optimization. However, since these examples deal with Gaussian bases, one may impose the positivity of the real part of the width parameter $a$ of each Gaussian, i.e., we choose a small $\epsilon$ and impose the linear constraint $a>\epsilon$. In our experiments, it was not necessary.


\section{Residual and Cost Function} \label{sec:residual_fn}

We consider a numerical approximation to the PDE \eqref{eq:PDE_original} via a $\theta$-parametrization and an iterative process in the parameter space $\Theta\subset \C^p$, where $p$ is the complex dimension. The set of time-discrete approximations  $u(\theta_0,\cdot),u(\theta_1,\cdot),\ldots$ are constructed via the iterative process 
\begin{equation} \label{eq:iteration}
    \theta_{k+1} = \arg \min_{\theta\in\Theta} \Vert r_k(\theta,\cdot) \Vert_{\hil}, \quad k=0,1,\ldots
\end{equation}
for various forms of the residual function $r_k(\theta,\cdot)$, whose explicit form will depend on the chosen numerical schemes.

\subsection{Discretization then parametrization}
The discretize-then approach falls into the class of Rothe methods for evolution equations, see for example the classic monograph \cite{Kacur1985}. In the context of parametric approximation for Schr\"odinger evolution, it has been proposed in \cite{kvaal2023,Schrader2024,schrader2026}; see also \cite{LN} for more general stiff evolution equations. Here, following \cite{ZCVP}, we employ the usual $\zeta$-scheme for time-discretization, which generalizes the Euler methods and the trapezoidal rule. 
Namely, given a continuous function $\varphi:[0,T]\to \R$, the integral over
small intervals $[t_a,t_b]$ with length $h=t_b-t_a$ are approximated via the weighted rule $\int_{t_a}^{t_b} \varphi(t) \dd t \approx h ( \zeta \varphi(t_a) + \hat{\zeta} \varphi(t_b)) $ with $\zeta\in [0,1]$ and $\hat{\zeta}=1-\zeta$.
The construction of the corresponding residual proceeds in two steps.\\
\textbf{Step (i)}. Integrate both sides of the PDE \eqref{eq:PDE_original} with respect to $t$ on small intervals of the form $[t_{k+1},t_k]$ with length $h=t_{k+1}-t_k$. Namely,
\begin{align*}
    \psi(t_{k+1}) - \psi(t_k) &= \int_{t_k}^{t_{k+1}} f(t,\cdot,\psi) \dd t, \\
    \psi_{k+1} - \psi_k &\approx h \left[ \zeta f(t_k,\cdot,\psi_k) + \hat{\zeta} f(t_{k+1},\cdot,\psi_{k+1}) \right] .
\end{align*}
This leads to an abstract optimization problem in the Hilbert space $\hil$,
\[
    \psi_{k+1} = \arg\min_{\psi\in\hil} \Vert \psi - \psi_k - h [\zeta f(t_k,\cdot,\psi_k) + \hat{\zeta} f(t_{k+1},\cdot,\psi) ] \Vert_{\hil} .
\]
\textbf{Step (ii)}. Consider the $\theta$-parametrization $\psi(t,\cdot)\approx u(\theta(t),\cdot)$ for all $t\in [0,T]$ and replace these values in the expression shown above. 
This leads to an optimization problem in the parameter space
$\Theta\subset\mathbb C^p$, 
and we define the \textit{Discretization then Parametrization} approach as 
the iterative method 
\[
\theta_{k+1} = \arg \min_{\theta\in\Theta} \Vert r_k(\theta,\cdot) \Vert_{\hil}, \quad k=0,1,\ldots
\]
with the residual function
\begin{equation} \label{eq:residual_Disc-Opt}
    r_k(\theta,\cdot) = u(\theta,\cdot) - u(\theta_k,\cdot) - h \left[ \zeta f(t_k,\cdot,u(\theta_k,\cdot)) + \hat{\zeta}f(t,\cdot,u(\theta,\cdot)) \right],
\end{equation}
see also \cite[\S4.1]{ZCVP}. Given the parameter $\theta_k\in\Theta$ at time $t_k$, we have in mind that an optimized parameter $\theta_{k+1}$ provides an approximation $u(\theta_{k+1},\cdot)\approx\psi(t_{k+1},\cdot)$. The value $\zeta=1$ (resp. $\zeta=0$) corresponds to the forward (resp. backward) Euler method.

\subsection{Parametrization then discretization}
The second approach parametrizes first and then discretizes in time. 
The formal construction follows three steps. First, we consider the general Dirac-Frenkel approximation at a point $u\in\M$. Second, we parametrize $u$ with $\theta\in\Theta$. Finally, we discretize in time considering a $\zeta$-scheme. \\
\textbf{Step (i)}. Let $\M$ be the manifold approximation to the space of solutions of the PDE \eqref{eq:PDE_original}. Then, following the Dirac-Frenkel principle, we project $f$ to the tangent space at $u\in\M$,
\[
    \partial_t u = P_{T_u \M} f(t,\cdot,u),
\]
where the projection is determined as the solution of the linear least squares problem
\[
    P_{T_u \mathcal{M}} f(t,\cdot,u) = \arg\min_{w\in T_u\M} \Vert f(t,\cdot,u) - w \Vert_{\hil} .
\]
\textbf{Step (ii)}. Consider the parametrized manifold $\M=\{ u(\theta,\cdot) \ : \ \theta\in\Theta \}$. This implies that $\partial_t u = \nabla_{\theta} u \cdot \dot{\theta}$ and that the tangent space at $u\in\M$ is of the form $T_u \M = \{ \nabla_{\theta} u \cdot \eta \ : \ \eta\in\C^p \}$, leading to
\begin{equation} \label{eq:DF}
    \dot{\theta} = \arg \min_{\eta\in\C^p} \Vert \tilde{r}(\theta,\eta,\cdot) \Vert_{\hil} ,
\end{equation}
where $\tilde{r}(\theta,\eta,\cdot)= \nabla_{\theta} u \cdot \eta - f(t,\cdot,u) $ is the Dirac--Frenkel residual function. \\
\textbf{Step (iii)}. Consider time discretization via the approximation $\dot{\theta}|_{t_k} \approx (\theta_{k+1}-\theta_k)/h$. Thus, the tangent space at $u_k$ is constrained to $\eta=(\theta-\theta_k)/h$, i.e., $T_{u_k} \M = \{ \nabla_{\theta} u_k \cdot (\theta-\theta_k)/h \ : \ \theta\in\Theta \}$, leading to
\[
    \theta_{k+1} = \arg\min_{\theta \in \Theta} \Vert \nabla_{\theta_k} u_k \cdot (\theta-\theta_k) - h f(t,\cdot,u_k) \Vert_{\hil} .
\]
Similarly, the discretization $\dot{\theta}|_{t_{k}} \approx (\theta_k-\theta_{k-1})/h$ (after a relabelling $k\leftarrow k+1$) leads to a search inside of the moving tangential space $V_u = \{ \nabla_{\theta} u \cdot (\theta-\theta_k)/h \ : \ \theta\in\Theta \}$, where $u=u(\theta,\cdot)$. Thus,
\[
    \theta_{k+1} = \arg\min_{\theta \in \Theta} \Vert \nabla_{\theta} u \cdot (\theta-\theta_k) - h f(t,\cdot,u) \Vert_{\hil} .
\]
The first expression corresponds to the forward Euler method,
whereas the second corresponds to the backward Euler method.
Consider an extension to $\zeta\in [0,1]$ such that $\zeta=1$ ($\zeta=0$) will reduce to forward (backward) Euler. 
This condition is satisfied by the residual function
\begin{equation} \label{eq:residual_Opt-Disc}
    r_k(\theta,\cdot) = \left[ \zeta \nabla_{\theta_{k}} u(\theta_k,\cdot) + \hat{\zeta} \nabla_{\theta} u(\theta,\cdot) \right] \cdot (\theta-\theta_k) - h \left[ \zeta f(t,\cdot,u(\theta_k,\cdot)) + \hat{\zeta} f(t,\cdot,u(\theta,\cdot)) \right] .
\end{equation}

Note that in comparison to the Dirac–Frenkel residual $\tilde{r}(\theta,\dot{\theta},\cdot)$, there is a multiplicative factor $h$, namely, $r_k(\theta_{k+1},\cdot) \approx h \tilde{r}(\theta|_{t_k},\dot{\theta} |_{t_k}, \cdot)$. More precisely,
\[ 
\lim_{h\to 0} r_k(\theta_{k+1},\cdot)/h = \tilde{r}(\theta|_{t_k},\dot{\theta} |_{t_k}, \cdot) .
\]

\begin{remark}
The main difference of the two residuals \eqref{eq:residual_Disc-Opt} and \eqref{eq:residual_Opt-Disc} is the weighted linearization $[\zeta \nabla_{\theta} u(\theta_k,\cdot) + \hat{\zeta}  \nabla_{\theta} u(\theta,\cdot)] (\theta - \theta_k)$ of the difference $u(\theta,\cdot) - u(\theta_k,\cdot)$.
\end{remark}

\subsubsection{Euler methods} 
For the two extreme cases $\zeta\in\{0,1\}$, the residual \eqref{eq:residual_Opt-Disc} has been considered before. 
For $\zeta=1$, it takes the form
\begin{equation*} 
    r_k(\theta_{k+1},\cdot) = \nabla_\theta u(\theta_k,\cdot) (\theta_{k+1} - \theta_k) - h f(t_k,\cdot,u(\theta_k,\cdot)) .
\end{equation*}
The corresponding one-step method $\theta_{k+1} = \arg\min_{\theta\in\Theta}\|r_k(\theta,\cdot)\|_\hil$ is the forward Euler method for the Dirac--Frenkel principle \eqref{eq:DF}, that is, $\theta_{k+1} = \theta_k + h\dot\theta_k$ with 
\[
\dot\theta_k = \arg \min_{\eta\in\C^p} \Vert \nabla_\theta u(\theta_k,\cdot) \eta  - f(t_k,\cdot,u(\theta_k,\cdot)) \Vert_\hil.
\]
For each time step, it requires to solve a linear least squares problem. On the other hand, for $\zeta=0$, we obtain 
\begin{equation*} 
    r_k(\theta_{k+1},\cdot) = \nabla_\theta u(\theta_{k+1},\cdot) (\theta_{k+1} - \theta_k) - h f(t_{k+1},\cdot,u(\theta_{k+1},\cdot)),
\end{equation*}
see also \cite[\S3.1.4]{ZCVP} and for a regularized implicit Euler discretization \cite[\S4.6]{FLL}.
Here, in the fully implicit approach, a nonlinear optimization is performed for each time-step. Note, that the implicit Euler method applied to the normal equations of the Dirac--Frenkel principle \eqref{eq:DF},
\[
\langle \nabla_\theta u(\theta,\cdot),\nabla_\theta u(\theta,\cdot)\rangle_\hil \ \dot\theta = 
\langle \nabla_\theta u(\theta,\cdot),f(t,\cdot,u(\theta,\cdot))\rangle_\hil, 
\]
results in the related, but different nonlinear system 
\[
\theta_{k+1} = \theta_k + h \langle \nabla_\theta u(\theta_{k+1},\cdot),\nabla_\theta u(\theta_{k+1},\cdot)\rangle_\hil^{-1}\langle \nabla_\theta u(\theta_{k+1},\cdot),f(t_{k+1},\cdot,u(\theta_{k+1},\cdot))\rangle_\hil.
\]

\begin{table}[th]
\centering
\caption{Comparison of the two considered residual-based formulations.}
\label{tab:comparison}
\renewcommand{\arraystretch}{1.15}
\begin{tabular}{p{3.2cm}p{5.2cm}p{5.2cm}}
\toprule
 & \textbf{Discretize--then--Parametrize}
 & \textbf{Parametrize--then--Discretize} \\
\midrule
Starting point &
Evolution equation in $\mathcal H$ &
Dirac--Frenkel principle on $\mathcal M$\\
Order of operations & 
Time discretization followed by parametrization &
Parametrization followed by time discretization \\
Residual structure &
$u(\theta)-u(\theta_k)$ compared to a $\zeta$-step of the PDE &
Linearized increment $\nabla_\theta u\,(\theta-\theta_k)$
compared to a $\zeta$-step of the PDE\\
Requirements &
(one-sided) Lipschitz or dissipativity assumptions &
Same assumptions + singular value bounds for $\nabla_\theta u$\\
Main result &
\Cref{prop:discr-opt} &
\Cref{prop:opt-discr} \\
Error sources &
Time discretization, residual minimization &
Same sources + poor conditioning\\
\bottomrule
\end{tabular}
\end{table}

\subsection{Explicit form of the cost function}\label{sec:sum}
In this subsection and for the numerical experiments in \Cref{sec:Num_Experiments,sec:Num_Experiments_II}, we consider a specific type of parametric approximation $\psi(t,x)\approx u(\theta(t),x)$, namely, 
\begin{equation}\label{eq:sum}
    u(\theta(t),x) = \sum_{l=1}^L b(\theta^l(t),x), \qquad x\in\Omega,
\end{equation}
where $\theta(t)=(\theta^1(t),\ldots,\theta^L(t))$ with $\theta^l(t)\in\Theta$ and $b(\theta^l(t),\cdot)\in\hil = L^2(\Omega,\C)$ are the individual basis functions. 
We assume that the function $f$ defining the PDE \eqref{eq:PDE_original} acts on such an ansatz $u$ as
\begin{equation} \label{eq:f_condition_L}
    f(t,x,u) = \sum_{l=1}^L \phi(\theta^l(t),t,x) b(\theta^l(t),x),
\end{equation}
with a function $\phi$ such that $\phi(\theta^l(t),t,\cdot)b(\theta^l(t),\cdot) \in \hil$, for all $l=1,\ldots,L$ and for all $t\in [0,T]$. 
The following remark illustrates the above assumptions and demonstrates their natural occurrence in a quantum application.

\begin{remark}
A motivating scenario for these assumptions is the case of a Schr\"odinger operator 
$f(t,x,\cdot) = -\I [ -\frac12\Delta_x + V(t,x) ]$ and a parametric approximation $u(\theta,\cdot)$, that is a sum of complex Gaussians
\begin{equation*}
        b(\theta^l,x) = \exp(- x^\trans A^lx + (\kappa^l)^\trans x + \gamma^l ),\quad x\in\R^d,
\end{equation*}
in holomorphic parametrization with 
\[
\theta^l=(A^l,\kappa^l,\gamma^l)\ \text{with}\ A^l\in \siegel^d, \kappa^l\in\C^d, \gamma^l\in\C.
\]
Here, $\siegel^d$ denotes the Siegel half space of degree $d$, that is, the set of $d\times d$ complex symmetric matrices with positive definite real part. Then, the parameter space is $\Theta = \siegel^d \times \C^{d+1} \subset \C^p$ with $p=\frac{1}{2}(d+1)(d+2)$, since $\dim(\siegel^d)=\frac{1}{2}d(d+1)$. In this case, the function $\phi$ of \eqref{eq:f_condition_L} is a sum of multivariate polynomials of degree $\le 2$ plus the potential function $V(t,x)$, 
\[
\phi(\theta^l,t,x) = -\I \left[ {\rm tr} A^l - \frac12 (-2A^l x + \kappa^l)^\trans (-2A^l x + \kappa^l) + V(t,x) \right].
\]
\end{remark}

\subsubsection{Discretize-then}
Assuming the basis form \eqref{eq:sum} and the $f$-condition \eqref{eq:f_condition_L}, the residual function \eqref{eq:residual_Disc-Opt} for the discretize-then approach is the finite sum
\begin{align}\label{eq:residual_sum}
    r_k(\theta,\cdot) = \sum_{l=1}^L r^l_k(\cdot)
\end{align}
with summands $r^l_k(\cdot) = Q_1(\theta^l_k,\theta^l,\cdot) b(\theta^l_k,\cdot) + Q_2(\theta^l_k,\theta^l,\cdot) b(\theta^l,\cdot)$, 
that depend on the basis functions $b(\theta^l_k,\cdot)$ and $b(\theta^l,\cdot)$ as well as the functions
\begin{align*}
    Q_1(\theta^l_k,\theta^l,\cdot) &= -1 - h \zeta \phi(\theta^l_k,\cdot) , \\
    Q_2(\theta^l_k,\theta^l,\cdot) &= 1 - h \hat\zeta \phi(\theta^l,\cdot) .
\end{align*}

\subsubsection{Parametrize-then}
To express the parametrize-then residual \eqref{eq:residual_Opt-Disc} in an analogous form, we need an additional assumption. We have to require the existence of a vector $v:\Omega\to\R^p$ such that 
\begin{equation}\label{eq:grad_b}
\nabla_{\theta} b(\theta,x) = \vec{v}(x) b(\theta,x), 
\end{equation}
where the vector $\vec v(x) = \vec v(\theta,x)$ might depend on $\theta$ as well.
For instance, $\vec{v}(x) = (-x^2,x,1)$ if the basis function $b(\theta,x)$ is the one-dimensional Gaussian \eqref{eq:gaussian_basis}. In higher dimensions $d\ge 1$, $\vec{v}(x)\cdot\theta$ would correspond to the argument of the exponential, i.e., $\vec{v}(x)\cdot\theta=-x^\trans A x + \kappa^\trans \cdot x + \gamma$. Then, we also have \eqref{eq:residual_sum} with summands defined by
\begin{align*}
    Q_1(\theta^l_k,\theta^l,\cdot) &= \zeta \left[ \vec{v}(x)\cdot (\theta^l-\theta^l_k) - h \phi(\theta^l_k,\cdot) \right], \\
    Q_2(\theta^l_k,\theta^l,\cdot) &= \hat{\zeta} \left[ \vec{v}(x)\cdot (\theta^l-\theta^l_k) - h \phi(\theta^l,\cdot) \right].
\end{align*}

\subsubsection{Unified form}
Hence, for a parametrization $u(\theta,\cdot) = \sum_{l=1}^L b(\theta^l,\cdot)$ with properties \eqref{eq:f_condition_L} and potentially \eqref{eq:grad_b}, the cost function for both residual approaches is explicitly given by 
\begin{align*}
    F_k(\theta) = \Vert r_k(\theta,\cdot)\Vert_\hil^2 = \sum_{l',l=1}^L \langle r_k^{l'}(\cdot) , r_k^l(\cdot) \rangle_\hil,
\end{align*}
with 
\begin{align*}
    \langle r_k^{l'}(\cdot) , r_k^l(\cdot) \rangle_\hil = 
   \sum_{\alpha,\beta=1}^2 \langle Q_\alpha(\theta_k^{l'},\theta^{l'},\cdot) b^{l'}_\alpha, 
   Q_\beta(\theta_k^{l},\theta^l,\cdot) b_\beta^{l} \rangle_\hil,
\end{align*}
where $b^l_1 = b(\theta_k^l,\cdot)$ and $b^l_2=b(\theta^l,\cdot)$. These inner products are then computed either analytically (via closed form formulas) or numerically, depending on the choice of basis functions and the functions $\phi$ and $\vec v$. For the components of the gradient 
\begin{equation*}
    \nabla_{\theta} F_k(\theta) = [ \nabla_{\theta^1} F_k(\theta), \ldots, \nabla_{\theta^L} F_k(\theta) ],
\end{equation*}
the above summations shrink, and we have
\begin{align*}
\nabla_{\theta^r} F_k(\theta) = \sum_{l'=1}^L & \left( \nabla_{\theta^r} \langle Q_1(\theta^{l'}_k,\theta^{l'},\cdot) b(\theta^{l'}_k,\cdot), Q_2(\theta^r_k,\theta^r,\cdot) b(\theta^r,\cdot) \rangle_\hil \right. \\
   & + \left. \nabla_{\theta^r} \langle Q_2(\theta^{l'}_k,\theta^{l'},\cdot) b(\theta^{l'},\cdot), Q_2(\theta^{r}_k,\theta^{r},\cdot) b(\theta^r,\cdot) \rangle_\hil \right), \quad {\rm for} \hquad r=1,\ldots,L.
\end{align*}

\subsubsection{Gaussian basis functions}
If a Gaussian basis is used and $\hil=L^2(\R^d,\C)$, then the cost function $F_k(\theta)$ and its gradient $\nabla F_k(\theta)$ can be computed in two different ways, depending on whether the function $\phi$ in \eqref{eq:f_condition_L} is polynomial or not.

If $\phi$ is a polynomial, then the integrals defining $F_k(\theta)$ are computed explicitly via polynomial Gaussian integrals. More precisely, given any polynomial $\phi$, the integral $F_k(\theta)$ and its gradient  $\nabla_\theta F_k(\theta)$ are computed in terms of sums of monomial Gaussian integrals, see \Cref{sec:gaussian_formulas}. An important advantage of Gaussian approximations is that the computation of these integrals in closed form can be generalized to higher dimensions. 
With analytic formulas, the cost of residual evaluation scales quadratically with the number $L$ of basis functions. 
The dominant computational expense is then the nonlinear optimization.

If $\phi$ is not a polynomial, then we compute each summand $ \langle Q_\alpha b^{l'}_\alpha, 
   Q_\beta b_\beta^{l} \rangle_{L^2} $ via quadratures (we omit the $\theta$-dependence for readability), i.e., 
    \begin{align*}
        \langle Q_\alpha b^{l'}_\alpha, 
   Q_\beta b_\beta^{l} \rangle_{L^2} &\approx \sum_{x_j\in {\rm Grid}} \overline{Q_\alpha(x_j) b^{l'}_\alpha (x_j)} Q_\beta (x_j) b_\beta^{l} (x_j) w_j, \\
        \nabla_{\theta^{l}} \langle Q_\alpha b^{l'}_\alpha, 
   Q_\beta b_\beta^{l} \rangle_{L^2} &\approx \sum_{x_j\in {\rm Grid}} \overline{Q_\alpha(x_j) b^{l'}_\alpha (x_j)} \nabla_{\theta^{l}} \left[ Q_\beta (x_j) b_\beta^{l} (x_j) \right] w_j,
    \end{align*}
    for weights $w_j\ge0$ and grid points $x_j\in\R^d$. We note that for the gradient formula a holomorphic parametrization of the Gaussians has to be used. For instance, we may consider Gauss--Hermite quadrature if $d=1$, see \Cref{sec:GH}. In higher dimension $d\ge 1$, tensor product quadrature comes with a computational complexity that increases exponentially with the dimension $d$; see also \cite{PC} for a sparse Gaussian method. 

\section{Unified Error Analysis}\label{sec:error}

For the error analysis, we return to the setting of a general Hilbert space $\hil$. For readability, we omit the sub-indices and simply write $\langle\cdot,\cdot\rangle$ for the inner product and $\Vert \cdot \Vert$ for the induced norm. 
For our analysis, the knowledge on whether $\psi(t,\cdot)$ is approximated by one or several basis functions or by which type of basis functions is \emph{not} needed.  For the general time marching scheme
\begin{align}\label{eq:theta_k+1}
\theta_{k+1} &= \arg \min_{\theta\in\Theta} \Vert r_k(\theta,\cdot) \Vert, \\\nonumber
u_{k+1} &= u(\theta_{k+1},\cdot),\quad k=0,1,\ldots,
\end{align}
we consider the two types of residual functions in the form 
\begin{equation}\label{eq:residual}
    r_k(\theta,\cdot) = \left\{
    \begin{array}{l}
        u(\theta,\cdot) - u(\theta_k,\cdot) - h \left[ \zeta f(t_k,\cdot,u(\theta_k,\cdot)) + \hat{\zeta}f(t,\cdot,u(\theta,\cdot)) \right],\\*[1ex]
        \left[ \zeta \nabla_{\theta} u(\theta_k,\cdot) + \hat{\zeta}  \nabla_{\theta} u(\theta,\cdot) \right] (\theta - \theta_k) - h \left[ \zeta f(t_k,\cdot,u(\theta_k,\cdot)) + \hat{\zeta} f(t,\cdot,u(\theta,\cdot)) \right],
    \end{array}\right. 
\end{equation}
with parameters $\theta,\theta_k\in\Theta$ and $\zeta\in[0,1]$ and $\hat{\zeta}=1-\zeta$. The two residuals -- one of them gradient-free, the other one containing the gradient of $u$ -- correspond to the discretize then parametrize approach~\eqref{eq:residual_Disc-Opt} and the opposite one \eqref{eq:residual_Opt-Disc}, respectively. The residual $r_k(\theta,\cdot)$ enters the error estimates in \Cref{prop:discr-opt} and \Cref{prop:opt-discr} via the maximal norm value 
\[
\rho = \max\left\{\Vert r_k(\theta_{k+1},\cdot) \Vert: 0\le t_{k+1}\le T\right\}.
\]
The second main error contribution is related to the time-regularity of the solution $\psi(t,\cdot)$ and the choice of $\zeta$. We will use constants $c_1,c_2>0$ defined by 
\begin{equation}\label{def:const}
\begin{array}{ll}
c_1 = \left( |\tfrac12-\zeta| + \tfrac14\right) \max_{t\in[0,T]}\|\psi''(t,\cdot)\|, & \text{if}\ \zeta\in[0,1]\ \text{and}\ \psi\in C^2([0,T],\mathcal H),\\*[2ex]
c_2 = \frac{1}{12} \max_{t\in[0,T]}\|\psi'''(t,\cdot)\|, &
\text{if}\ \zeta = \frac12\ \text{and}\ \psi\in C^3([0,T],\mathcal H).
\end{array}
\end{equation}
Given $\rho$ and $c_1,c_2$, we will be able to bound the error of the discretize-then approximation schemes, see 
\Cref{sec:discr-opt}. For the parametrize-then approach, see \Cref{sec:opt-discr}, we will require additional lower bounds on the smallest singular value of the parametrizing gradient.

\subsection{Error analysis for the discretization-then approach}\label{sec:discr-opt}

The error analysis here extends the results of \cite[Proposition 5 \& 6]{ZCVP}, that consider the explicit and implicit Euler method for Lipschitz continuous right-hand side. \cite{LN} also uses the discretize-then setting and performs the error analysis for regularized 
Radau IIA and Gau\ss\ implicit Runge--Kutta methods. Here, we analyze the full family of $\zeta$-methods, $\zeta\in[0,1]$, and show the expected gain in accuracy for $\zeta=\frac12$ provided that the original dynamics are regular enough. Also in the one-sided Lipschitz and the dissipative regime, it is the implicit Euler method and the ones with $\zeta\le\frac12$ that provide the expected estimates. 
A common feature of currently available error estimates for
parametric approximation methods is the presence of residual terms.

\begin{theorem}[Discretize-then]\label{prop:discr-opt}
Consider either one of the following three cases:
\begin{description}
    \item[(Lip)] If there exists a constant $\lambda> 0$ such that for all $t\in\R$ and $u,v\in\hil$
    \begin{align*}
        \Vert f(t,\cdot,u) - f(t,\cdot,v) \Vert &\leq \lambda \Vert u-v \Vert, 
    \end{align*}
    then choose $\zeta\in[0,1]$ and $s\in\{1,2\}$. Choose $C>1$ and consider a sufficiently small $h>0$ satisfying
$\lambda\hat\zeta h<\frac{1}{C}<1$. Set $C_0 = e^{C\lambda T}$ and $M= C_0/\lambda$. Define $E_0 = C_0 \Vert\psi(t_0,\cdot)-u(\theta_0,\cdot)\Vert$.
    \item[(os-Lip)] If there exists a constant $\ell\in\R$ such that for all $t\in \R$ and $u,v\in\hil$
    \begin{align*}
        \Re \langle f(t,\cdot,u)-f(t,\cdot,v) , u-v \rangle & \leq \ell \Vert u-v \Vert^2,
    \end{align*}
    then choose $\zeta=0$ (implicit Euler) and $s=1$. Consider a sufficiently small $h>0$ satisfying the (possible) restriction $\ell h\le \frac12$. Set 
    \[
(C_0,M) = \left\{ 
\begin{array}{ll}
(1,T), &  \text{if}\quad \ell=0,\\*[1ex]
(1,1/|\ell|), & \text{if}\quad \ell<0,\\*[1ex]
e^{2\ell T}(1,1/\ell), & \text{if}\quad \ell>0. 
\end{array}
\right.
    \]    
    Define $E_0 = C_0 \Vert\psi(t_0,\cdot)-u(\theta_0,\cdot)\Vert$. 
    \item[(diss)] If the one-sided Lipschitz condition (os-Lip) holds with $\ell=0$, then choose $\zeta\le\frac12$ and $s\in\{1,2\}$. Set $M=T$ and define 
    \[
    E_0^2 = \Vert \psi(t_0,\cdot)-u(\theta_0,\cdot) \Vert^2 + \zeta^2 h^2 \Vert f(t_0,\cdot,\psi(t_0,\cdot)) - f(t_0,\cdot,u(\theta_0,\cdot))\Vert^2.
    \]
\end{description}

\noindent
Then, for either one of the assumed scenarios (Lip), (os-Lip), or (diss), the error of the discretize-then-parametrize $\zeta$-method satisfies
\[
\boxed{
\Vert \psi(t_k,\cdot)-u(\theta_k,\cdot) \Vert \leq  E_0 + M\left(\frac{\rho}{h} + c_s h^s\right)}
\]
for all $t_{k+1} = (k+1)h\le T$, where the constants $c_1,c_2>0$ are defined in \eqref{def:const}.
\end{theorem}

\begin{proof}
    For notational simplicity, we suppress the time and space dependence of the function $f$, denote $r_k = r(\theta_{k+1},\cdot)$ and $u_k = u(\theta_k,\cdot)$. 
    We subtract the relations for the local error $e_k$ of the standard $\zeta$-method and the one for the residual $r_k$ from each other,
    \begin{align}\label{eq:local_error}
    \psi(t_{k+1}) - \psi(t_k) &= h\left(\zeta f(\psi(t_k)) +\hat\zeta f(\psi(t_{k+1}))\right) + e_k,\\ \nonumber
    u_{k+1} - u_k &= h\left(\zeta f(u_k) +\hat\zeta f(u_{k+1})\right) + r_k,
    \end{align}
    and obtain that the global error 
    \[
    \varepsilon_{k+1} = \psi(t_{k+1})-u_{k+1}
    \]
    satisfies the error recursion
    \begin{equation}\label{eq:error}
    \varepsilon_{k+1} = \varepsilon_{k} +\zeta h\left( f(\psi(t_k))-f(u_k)\right) + \hat\zeta h\left( f(\psi(t_{k+1}))-f(u_{k+1})\right) + e_k - r_k.
    \end{equation}

    \paragraph{(Lip)} We start with the Lipschitz continuous case. By the triangle inequality and the Lipschitz bound on $f$, we deduce from \eqref{eq:error} that
    \[
    \|\varepsilon_{k+1}\| \le (1+\lambda \zeta h) \|\varepsilon_{k}\| + \lambda \hat\zeta h\|\varepsilon_{k+1}\| + \|e_k\| + \|r_k\|,
    \]
    and, due to the step-size restriction $h\lambda\hat\zeta<1$, 
    \[
    \|\varepsilon_{k+1}\| \le \frac{1+\lambda \zeta h}{1-\lambda \hat\zeta h}\ \|\varepsilon_{k}\| + 
        \frac{1}{1-\lambda \hat\zeta h} \left( \|e_k\| + \|r_k\|\right) .
    \]
    Iterating this estimate, we obtain
\begin{align*}
    \Vert \varepsilon_k \Vert \leq & \left( \frac{1+\lambda\zeta h}{1-\lambda\hat{\zeta} h} \right)^k \Vert \varepsilon_0 \Vert + 
    \frac{1}{1-\lambda\hat{\zeta}h } \sum_{j=0}^{k-1} \left( \frac{1+\lambda\zeta h}{1-\lambda\hat{\zeta} h} \right)^{j} \left( \Vert e_{k-1-j} \Vert + \Vert r_{k-1-j}\Vert \right).
\end{align*}
Moreover, again due to the step-size restriction, we have $\frac{1}{1-\lambda \hat{\zeta} h} < C$. Thus,
\begin{equation*}
    \frac{1+\lambda\zeta h}{1-\lambda\hat{\zeta} h} = 1 + \frac{\lambda h }{1-\lambda\hat{\zeta} h} \leq 1 + C\lambda h \le e^{C \lambda h},
\end{equation*}
and, using the finite sum formula, we obtain
\begin{align*}
    \sum_{j=0}^{k-1} \left( 1 + C \lambda h \right)^j = \frac{(1 + C \lambda h)^k -1 }{C\lambda h} \leq \frac{1}{C\lambda h} e^{C\lambda h k}.
\end{align*}
Together with the local error estimate of Lemma \ref{lem:local_error}, this proves the claimed error bound.

\paragraph{(os-Lip)}
We use the error recursion \eqref{eq:error} for the implicit Euler method, $\zeta=0$, $\hat\zeta=1$, that is,
\begin{equation*}
    \varepsilon_{k+1} = \varepsilon_k + h \left( f(\psi(t_{k+1})) - f(u_{k+1}) \right) + e_k - r_k.
\end{equation*}
We compute the real part of the inner product with $\varepsilon_{k+1}$ on both sides of the equation,
\begin{align*}
    \Vert \varepsilon_{k+1}\Vert^2 = & \Re \langle\varepsilon_k,\varepsilon_{k+1}\rangle + \Re \langle e_k,\varepsilon_{k+1}\rangle - \Re \langle r_k,\varepsilon_{k+1}\rangle \\
    &+ h \Re\langle f(\psi(t_{k+1})) - f(u_{k+1}) ,\psi(t_{k+1}) - u_{k+1}\rangle.
\end{align*}
Then, we use the one-sided Lipschitz property and the Cauchy-Schwarz inequality to obtain
\begin{align*}
    \Vert \varepsilon_{k+1} \Vert \leq \Vert\varepsilon_k\Vert + \Vert e_k\Vert + \Vert r_k \Vert + \ell h \Vert \varepsilon_{k+1} \Vert.
\end{align*}
Due to the (possible) step-size restriction, we have $1-\ell h >0$ and  
\begin{align*}
    \Vert \varepsilon_{k+1}\Vert \leq \frac{1}{1-\ell h} \left( \Vert\varepsilon_k\Vert + \Vert e_k\Vert + \Vert r_k \Vert \right).
\end{align*}
Iterating the argument, we arrive at
\begin{equation*}
    \Vert \varepsilon_k \Vert \leq \left(\frac{1}{1-\ell h}\right)^k \Vert \varepsilon_0 \Vert + \frac{1}{1-\ell h}\sum^{k-1}_{j=0} \left(\frac{1}{1-\ell h}\right)^{j} 
    \left( \Vert e_{k-j-1}\Vert + \Vert r_{k-j-1}\Vert \right) .
\end{equation*}
For the dissipative case $\ell=0$, this implies
\[
\Vert \varepsilon_k \Vert \leq  \Vert \varepsilon_0 \Vert + k\left( \rho_k + c_1 h^2\right)\le \Vert \varepsilon_0 \Vert + T\left( \frac{\rho_k}{h} + c_1 h\right).
\]
For $\ell\neq0$, we calculate the geometric sum as
\[
\gamma := \frac{1}{1-\ell h}\sum^{k-1}_{j=0} \left(\frac{1}{1-\ell h}\right)^{j} = \frac{1}{\ell h}\left(\left(\frac{1}{1-\ell h}\right)^k-1\right).
\]
In the non-expansive case $\ell<0$, we have $(1-\ell h)^{-k}<1$ and $\gamma\le \frac{1}{|\ell| h}$. This implies $C_0 = 1$ and $M = 
\frac{1}{|\ell|}$. For the expansive case $\ell>0$, the step-size restriction allows for  
$1-\ell h \ge e^{-2\ell h}$ and thus for $C_0=e^{2\ell T}$, $M = e^{2\ell T}/\ell$. 

\paragraph{(diss)}
    We rewrite the error recursion \eqref{eq:error} of as
    \[
    \varepsilon_{k+1} - \hat\zeta h\left(f(\psi(t_{k+1}))-f(u_{k+1})\right) = 
     \varepsilon_{k} + \zeta h\left(f(\psi(t_{k}))-f(u_{k})\right) + e_k - r_k
    \]
    and use the triangle inequality to obtain
    \[
    \Vert\varepsilon_{k+1} - \hat\zeta h\left(f(\psi(t_{k+1}))-f(u_{k+1})\right)\Vert \le 
     \Vert\varepsilon_{k} + \zeta h\left(f(\psi(t_{k}))-f(u_{k})\right)\Vert + \Vert e_k\Vert + \Vert r_k\Vert.
    \]
   We introduce an energy
   \[
   E_k := \sqrt{\Vert\varepsilon_k\Vert^2 + \zeta^2 h^2 \Vert f(\psi(t_{k}))-f(u_{k}) \Vert^2}
   \]
   and use the one-sided Lipschitz condition to estimate
    \begin{align*}
    &\Vert\varepsilon_{k} + \zeta h\left(f(\psi(t_{k}))-f(u_{k})\right)\Vert^2 
    \le  E_k^2,\\
    &\Vert\varepsilon_{k+1} - \hat\zeta h\left(f(\psi(t_{k+1}))-f(u_{k+1})\right)\Vert^2  
     \ge E_{k+1}^2,
    \end{align*}
where the bound from below is due to $\hat\zeta^2 = \zeta^2 + (1-2\zeta)\ge\zeta^2$ for $\zeta\le \frac12$. Therefore, we can relate the energies as
\[
E_{k+1} \le E_k + \Vert e_k\Vert + \Vert r_k\Vert.
\]
Iterating the estimate, we obtain
\[
E_k \le E_0 + \sum_{j=0}^{k-1} \left( \Vert e_j\Vert + \Vert r_j\Vert\right)
\]
and $\Vert\varepsilon_k\Vert \le E_0 + k\left( \rho_k + c_s h^{s+1}\right)$, using that $\Vert\varepsilon_k\Vert \le E_k$.
\end{proof}

\bigskip
\begin{remark}
    \Cref{prop:discr-opt} provides a unified convergence analysis for residual-based
discretize-then-parametrize schemes within the full family of
$\zeta$-methods. The estimate separates the residual contribution $\rho/h$ and the classical time
discretization contribution $c_s h^s$ as the sources of error and 
makes explicit how optimization and time discretization jointly determine the overall accuracy.
\end{remark}

\subsection{Error analysis for the parametrize-then approach}\label{sec:opt-discr}

Error and stability analysis for the parametrize-then approach has been carried out in \cite[\S3]{ZCVP} for time-continuous dynamics. The analysis of fully discretized, regularized Runge--Kutta schemes in \cite{FLL} comprises the explicit and implicit Euler method. There, the regularization induces a step-size restriction that substitutes the upper bound \eqref{eq:bound_f} and the lower bound on the parametrizing gradient's singular values \eqref{eq:bound_sv}. These two bounds are needed here to control the distance of consecutive parameter values. They are the key ingredients for the new error constant $c_s'$ that distinguishes the error bound of \Cref{prop:opt-discr} from the Lipschitz case in \Cref{prop:discr-opt}.

\begin{theorem}[Parametrize-then]\label{prop:opt-discr}
Choose $\zeta\in[0,1]$ and $s\in\{1,2\}$. 
Assume that there exists constants $\lambda, m, \sigma_0>0$ such that for all $t\in\R$, $u,v\in\hil$, and $\theta,\eta\in\C^p$,
    \begin{align}\nonumber
        &\Vert f(t,\cdot,u) - f(t,\cdot,v) \Vert \leq \lambda \Vert u-v \Vert,\\\label{eq:bound_f}
        &\Vert f(t,\cdot,u)\Vert \le m,\\\label{eq:bound_sv}
        &\sigma_{\min}(\zeta \nabla_\theta u(\theta) + \hat\zeta\nabla_\theta u(\eta)) \ge \sigma_0.
    \end{align}
Choose $C>1$ and consider a sufficiently small $h>0$ satisfying
$\lambda\hat\zeta h<\frac{1}{C}<1$. Set $C_0 = e^{C\lambda T}$ and $M= C_0/\lambda$. Define $E_0 = C_0 \Vert\psi(t_0,\cdot)-u(\theta_0,\cdot)\Vert$.
Then, the error of the parametrize-then-discretize $\zeta$-method satisfies
\[
\boxed{
\Vert \psi(t_k,\cdot)-u(\theta_k,\cdot) \Vert \leq  E_0 + M\left(\frac{\rho}{h} + (c_s+c_s') h^s\right)}
\]
for all $t_{k+1} = (k+1)h\le T$, where the constants $c_1,c_2>0$ are defined in \eqref{def:const} and 
\[
c_s' = 2\gamma_{s+1}  \left( \frac{2 m}{\sigma_{0}-h\hat\zeta \lambda \gamma_1}\right)^{s+1},\quad s=1,2,
\]
with $\gamma_{n} = \frac{1}{n!} \,p^{n/2}\, \sup_{|\alpha|=n} \Vert\partial^\alpha_\theta u\Vert$ for $n\ge 1$.
\end{theorem}

\begin{proof}
As in the previous proof, we denote $u_k = u(\theta_k,\cdot)$. To distinguish between the norm in the Hilbert space $\hil$ and the Euclidean norm for the parameter space $\C^p$, we write for the latter one $\|\cdot\|_p$. We Taylor expand around $\theta_k$ and $\theta_{k+1}$, 
\begin{align*}
    u_{k+1}-u_k &= \nabla_\theta u(\theta_k)(\theta_{k+1}-\theta_{k}) + q_k, \\
    u_{k}-u_{k+1} &=  \nabla_\theta u(\theta_{k+1})(\theta_{k}-\theta_{k+1}) + \hat q_k,
\end{align*}
with $\Vert q_k\Vert,\Vert \hat q_k\Vert \le  \gamma_2 \Vert \theta_{k+1}-\theta_k\Vert^2_p$.
We add these two equations in combination with $\zeta$-weighting and introduce the residual $r_k = r(\theta_{k+1},\cdot)$,
\begin{align*}
u_{k+1}-u_k &= \zeta \left(u_{k+1}-u_k\right) + \hat\zeta(u_{k+1}-u_k)\\
&= \left(\zeta \nabla_\theta u(\theta_k) + \hat\zeta\nabla_\theta u(\theta_{k+1})\right)  (\theta_{k+1}-\theta_k) + 
\zeta q_k - \hat\zeta q_k \\
&=h\left(\zeta f(u_k) + \hat\zeta f(u_{k+1})\right) + r_k + 
\zeta q_k - \hat\zeta q_k.
\end{align*}
Subtracting this equation from the one for the local error \eqref{eq:local_error}, we obtain
\[
\varepsilon_{k+1} = \varepsilon_k + h\left(
\zeta\left(f(\psi(t_{k}))-f(u_{k})\right) + 
\hat\zeta\left(f(\psi(t_{k+1}))-f(u_{k+1})\right)
\right) + e_k -r_k + p_k
\]
with $p_k = \hat\zeta\hat q_k - \zeta q_k$. For $\zeta=\frac12$, we expand $u$ to third order, so that the
second-order terms cancel under midpoint weighting. Thus, we have
\[
\Vert p_k\Vert \le 2\gamma_{s+1}  \Vert\theta_{k+1}-\theta_k\Vert^{s+1}_p
\]
with $s=2$ for $\zeta=\frac12$ and with $s=1$ for $\zeta\neq\frac12$. To estimate the update length $\Vert\theta_{k+1}-\theta_k\Vert_p$, 
we revisit the scheme \eqref{eq:theta_k+1} and write 
\begin{align*}
&\theta_{k+1} = \theta_k + h\dot\theta_k\ \text{with}\\
&\dot\theta_k = \arg\min_\eta \Vert \left(\zeta \nabla_\theta u(\theta_k) + \hat\zeta\nabla_\theta u(\theta_k+h\eta)\right)\eta 
- \left(\zeta f(u(\theta_k)) + \hat\zeta f(u(\theta_k+h\eta))\right)\Vert.
\end{align*}
This is a nonlinear least squares problem as considered in Lemma~\ref{lem:least} with 
\[
A(\eta) = \zeta \nabla_\theta u(\theta_k) + \hat\zeta\nabla_\theta u(\theta_k+\eta),\quad 
b(\eta) = \zeta f(u(\theta_k)) + \hat\zeta f(u(\theta_k+\eta)). 
\]
The Lemma's coercivity estimate requires a lower bound $\sigma_0>0$ on the singular values of $A(h\eta)$, a Lipschitz bound on $b$, 
and the value of $b(0) = f(u_k)$. The Lipschitz estimate is
\begin{align*}
\Vert b(\eta)-b(y)\Vert &\le \hat\zeta \Vert f(u(\theta_k+\eta))- f(u(\theta_k+y))\Vert
\le \hat\zeta \lambda \gamma_1 \Vert\eta-y\Vert_p
\end{align*}
for all $\eta,y$, and the upper bound for $b(0)$ is $\Vert b(0)\Vert \le m$. Thus, we obtain
\[
\Vert\dot\theta_k\Vert_p \le \frac{2 m}{\sigma_{0}-h\hat\zeta \lambda \gamma_1}
\]
and
\[
\Vert p_k\Vert \le 2\gamma_{s+1} 
\left( \frac{2 m}{\sigma_{0}-h\hat\zeta \lambda \gamma_1}\right)^{s+1} h^{s+1}.
\]
Hence, both the parametrizing error $p_k$ and the local error $e_k$ are of the same order $h^{s+1}$. Once this is established, the Lipschitz argument of Theorem~\ref{prop:discr-opt} provides the claimed global error bound.   
\end{proof}

\begin{remark}
\Cref{prop:opt-discr} extends the preceding error analysis to the parametrize-then-discretize formulation. In comparison with \Cref{prop:discr-opt}, an additional
error contribution appears through the constant $c_s'$, reflecting the
fact that the manifold parametrization enters the dynamics through its
tangent space. The estimate therefore provides a
quantitative explanation for the well-known sensitivity of
Dirac--Frenkel-type methods to tangent-space collapse. 
\end{remark}

\subsection{Norm and energy error}

For the time-dependent Schr\"odinger equation $\I\partial_t \psi(t) = \hat{H}\psi(t)$, $\psi(0) = \psi_0$ with a self-adjoint Hamiltonian operator $\hat{H}$ and initial condition $\psi_0\in\hil$, there are two conserved quantities. The norm and the energy satisfy
\[
\Vert \psi(t,\cdot)\Vert = \Vert\psi_0\Vert,\qquad 
\langle \psi(t,\cdot),H\psi(t,\cdot)\rangle = 
\langle \psi_0,H\psi_0\rangle
\]
for all $t\in\R$. The time-continuous Dirac--Frenkel approximation inherits this conservation property, see e.g. \cite[Chapter II.1.5]{CL1} or  \cite{LS}. Here, however, with the residual-minimization schemes considered above, these conservation errors are controlled by the norm of the global error $\varepsilon_k = \psi(t_k,\cdot) - u(\theta_k,\cdot)$.

\begin{proposition}[Norm and energy]\label{prop:norm_energy}
Let us consider a time-dependent Schr\"odinger equation $\I \partial_t \psi = \hat{H} \psi$ with normalized initial condition $\psi_0\in\hil$, $\Vert\psi_0\Vert = 1$. Then, the norm 
and energy of the approximate solution $u_k:=u(\theta_k,\cdot)$ at time $t_k$ obtained by \eqref{eq:theta_k+1} obey
\begin{align*}
    &\left| \Vert u_k \Vert - 1 \right| \leq \Vert \varepsilon_k \Vert, \\*[1ex]
    &\left| \frac{\langle u_k,\hat{H} u_k\rangle}{\langle \psi_0,\hat{H}\psi_0\rangle} -1 \right| \leq 
    K\, \Vert\varepsilon_k\Vert
\end{align*}
for all $t_k\le T$ with constant $K = \max_{t_k\le T}\left(\Vert \hat{H}\psi(t_k,\cdot)\Vert + \Vert \hat{H}u_k\Vert\right)/|\langle \psi_0,\hat{H}\psi_0\rangle|$.
\end{proposition}
\begin{proof}
Using the triangular inequality and considering that $\Vert\psi(t_k,\cdot)\Vert = 1$, we observe that
\begin{align*}
    &\Vert u_k  \Vert  \leq  \Vert u_k - \psi(t_k,\cdot) \Vert + \Vert \psi(t_k,\cdot) \Vert = \Vert \varepsilon_k\Vert + 1,\\
    &1  \le \Vert u_k\Vert + \Vert u_k - \psi(t_k,\cdot) \Vert = \Vert u_k\Vert + \Vert \varepsilon_k\Vert.
\end{align*} 
For the energy, we use the self-adjointness of $H$ and the Cauchy-Schwarz inequality to obtain
\begin{align*}
   \left| \langle u_k,\hat{H}u_k\rangle - \langle \psi_0,\hat{H}\psi_0\rangle \right| &=
   \left| \langle u_k,\hat{H}u_k\rangle - \langle \psi(t_k,\cdot),\hat{H}\psi(t_k,\cdot)\rangle \right|\\
   &= 
   \left| \langle u_k-\psi(t_k,\cdot),\hat{H}u_k\rangle + \langle \psi(t_k,\cdot),\hat{H}\left(u_k-\psi(t_k,\cdot)\right)\rangle \right|\\
   &\le \Vert\varepsilon_k\Vert \left(\Vert \hat{H}\psi(t_k,\cdot)\Vert + \Vert \hat{H}u_k\Vert\right). 
\end{align*}
\end{proof}


\section{Gaussian integrals and initialization} \label{sec:gaussian_formulas}

For a Gaussian basis and $\hil=L^2(\R^d,\C)$, the cost function $F_k(\theta)=\Vert r_k(\theta,\cdot) \Vert^2_\hil$ for parameter $\theta\in\Theta$ can be expressed in terms of sums of Gaussian integrals. They are at the core of successful computational methods like the variational multi-configuration Gaussian method (vMCG) \cite{Richings_etal2015}. Here, for a simple and conceptual exposition, we focus on the univariate case $d=1$. Let us recall that we use complex parameters $\theta=(a,\kappa,\gamma)\in\Theta =\siegel\times \C^2 \subset \C^3$ and denote the single Gaussian function by
\begin{equation*}\label{eq:gaussian_def_1d}
    g(\theta,x) := e^{-ax^2+\kappa x +\gamma},\quad x\in\R.
\end{equation*}
Notice that the identity $g(\theta_1,x) g(\theta_2,x) = g(\theta_1+\theta_2,x)$ holds for any $\theta_1,\theta_2\in\Theta$ with this parametrization choice, which leads to compact formulas for the inner products.

\subsection{Monomial integrals}
For $n\in\N$, we then define the $n$-monomial Gaussian integral as
\begin{equation}\label{eq:Gn}
    G_n(\theta) := \int_{\R} x^n g(\theta,x) \dd x.
\end{equation}
The case $n=0$ is the usual Gaussian integral
\begin{equation}\label{eq:G0}
    G_0(\theta) = e^{\frac{\kappa^2}{4a}+\gamma} \sqrt{\frac{\pi}{a}},
\end{equation}
which holds for complex parameters $\theta$ as long as $a\in\siegel$, see \cite[Appendix A]{Fol89}.

With this formula, one can determine the explicit form of the integral $G_n(\theta)$ for any $n\in\N$, using complex scaled\footnote{Note that the coefficients of this polynomial are not complex. In fact, its coefficients are the absolute values of the coefficients of the usual Hermite polynomial.} (physicists') Hermite polynomials $\tilde{H}_n(x):=(-\I )^n H_n(\I x)$. They obey the three-term recurrence
\begin{align*}
&\tilde{H}_{n+1}(x)=2x \tilde{H}_n(x) + 2n \tilde{H}_{n-1}(x),\quad n>0,\\
&\tilde{H}_0(x)=1, \ \tilde{H}_1(x)=2x.
\end{align*}
 This scaling provides a convenient representation of the integrals without additional powers of $\I$.

\begin{formula} \label{prop:n_gaussian_integral}
Given $\theta=(a,\kappa,\gamma)\in\Theta$ and $n\in\N$, the closed form of the $n$-monomial Gaussian integral and its gradient is given by
\begin{align*}
    G_n(\theta) &= (2\sqrt{a})^{-n} \tilde{H}_n\left( \frac{\kappa}{2\sqrt{a}} \right) G_0(\theta),\\
    \nabla_{\theta} G_{n}(\theta) &= \begin{pmatrix}
       - G_{n+2}(\theta) &
       G_{n+1}(\theta) &
       G_{n}(\theta)
    \end{pmatrix}
\end{align*}
with a suitable choice of the branch\footnote{In this paper, we have chosen the branch with $|{\rm Arg}(\sqrt{a})|<\pi/4$, since we want these formulas to reduce to the usual Gaussian integrals when $a$ is real, i.e., we want a positive $\sqrt{a}$ as $\Im(a)\to 0$. } of the square root for $a\in\siegel$; see Appendix~\ref{app:Gauss} for a proof.
\end{formula}

\subsection{Polynomial integrals and quadrature}\label{sec:GH}
For a Gaussian integral with respect to a function $\varphi:\R\to\C$ we denote
\begin{equation*}
    I(\theta;\varphi) := \int_{\R} \varphi(x) g(\theta;x) \dd x.
\end{equation*}
For a vector field $\vec{\Phi}:\R\to\C^m$ with components $\Phi_1,\ldots,\Phi_m$, we set
\begin{equation*}
    I(\theta;\vec{\Phi}) := \begin{pmatrix}
       I(\theta;\Phi_1) & 
       \ldots & 
       I(\theta;\Phi_m)
    \end{pmatrix}.
\end{equation*}
With this compact notation, Gaussian inner products and their parameter gradients can be stated as follows. 

\begin{formula}
Let $Q_1(\xi,\eta;\cdot),Q_2(\xi,\eta;\cdot):\R\to \C$ be functions with arbitrary but fixed parameters $\xi,\eta\in\Theta\subset \C^p$.
 Then, for $\sigma\in\{\xi,\eta\}$,
\begin{align*}
   \left\langle Q_1(\xi,\eta;\cdot) g(\xi,\cdot),Q_2(\xi,\eta;\cdot)g(\eta,\cdot) \right\rangle_\hil &= I(\bar{\xi}+\eta;\varphi),\\
   \nabla_{\sigma} \left\langle Q_1(\xi,\eta;\cdot) g(\xi,\cdot),Q_2(\xi,\eta;\cdot)g(\eta,\cdot) \right\rangle_\hil &= 
   I(\bar{\xi}+\eta; \nabla_{\sigma} \varphi) + \delta_{\sigma \eta} I(\bar{\xi}+\eta; \vec{v} \varphi)
\end{align*}
with $\varphi:=\overline{Q_1(\xi,\eta;\cdot)} Q_2(\xi,\eta;\cdot)$ and $\vec{v}(x)=(-x^2,x,1)$.
\end{formula}

\begin{proof}
The first identity follows from $\overline{g(\xi,\cdot)} g(\eta,\cdot) = g(\bar\xi+\eta,\cdot)$. The subsequent gradient identity stems from two observations. First, for a function $\varphi$, that depends on a parameter $\theta$, the integral satisfies
\begin{align*}
    \nabla_\theta I(\theta;\varphi) = I(\theta;\nabla_\theta \varphi) + I(\theta;\varphi\vec{v}).
\end{align*} 
Second, only the $\eta$-derivative sees the parameter dependence of the Gaussian  $g(\bar\xi+\eta,\cdot)$ and creates the vector $\vec{v}$.
\end{proof}

For polynomial functions, we obtain a compact description of the Gaussian weighted inner products in terms of the monomial integrals $G_n(\theta)$, $n\ge0$, introduced in \eqref{eq:Gn}. For a multi-variate extension of this type of formula, see \cite{VanicekZhang2025}.

\begin{formula} \label{Prop:inner_prod_xi_eta}
Let $Q_1$ and $Q_2$ be polynomials with parameter dependent coefficients,  
\begin{equation*} 
    Q_1(\xi,\eta,x) = \sum_{i=0}^{m} \alpha_i(\xi,\eta) x^i, \quad
    Q_2(\xi,\eta,x) = \sum_{j=0}^{n} \beta_j(\xi,\eta) x^j.
\end{equation*}
We assume that $\overline{\alpha_i(\xi,\eta)} = \alpha_i(\bar\xi,\bar\eta)$ for all $i=0,\ldots,m$. Then, the Gaussian weighted inner product and its $\sigma$-gradient, $\sigma\in\{\xi,\eta\}$, are given by
\begin{align*}
   \left\langle Q_1(\xi,\eta;\cdot) g(\xi,\cdot),Q_2(\xi,\eta;\cdot)g(\eta,\cdot) \right\rangle_\hil 
   &= \sum_{s=0}^{n+m} c_s(\xi,\eta) G_s(\bar{\xi}+\eta),\\
   \nabla_\sigma \left\langle Q_1(\xi,\eta) g(\xi,\cdot),Q_2(\xi,\eta)g(\eta,\cdot) \right\rangle_\hil 
   &= \sum_{s=0}^{n+m} \left( \nabla_\sigma c_s(\xi,\eta) \cdot G_s(\bar{\xi}+\eta) + \delta_{\sigma \eta} c_s(\xi,\eta) \nabla_\eta G_s(\bar{\xi}+\eta) \right)
\end{align*}
where
\begin{align*}
c_s(\xi,\eta) = \sum_{\substack{i,j\\i+j=s}}^{m,n} \alpha_i(\bar\xi,\bar\eta) \beta_j(\xi,\eta),\quad
\nabla_\eta c_s(\xi,\eta) = \sum_{\substack{i,j\\i+j=s}}^{m,n} \alpha_i(\bar\xi,\bar\eta) \nabla_\eta \beta_j(\xi,\eta) . 
\end{align*}
\end{formula}

When closed-form expressions are unavailable, the integrals must be approximated numerically.
For univariate Gaussian integrals with smooth integrands, Gauss--Hermite quadrature 
\begin{equation} \label{eq:GH_quadrature}
   \mathcal{Q}_n(\varphi) = \sum_{j=1}^n w_j \varphi(x_j)
\end{equation}
seems to be the method of choice.  In the numerical experiments of \Cref{sec:int_quad}, we have considered integrand functions $\varphi$, that allow for a change of variables in the complex plane, and have used the approximation 
\[
I(\theta;\varphi) \approx \quadr_n (\theta;\varphi)
\]
with 
\begin{equation} \label{eq:GH_quad_theta}
    \quadr_n (\theta;\varphi) = \frac{e^{\kappa^2/4a+\gamma}}{\sqrt{a}} \sum_{j=1}^n w_j \varphi\left( \frac{x_j}{\sqrt{a}} + \frac{\kappa}{2a} \right).
\end{equation}

\begin{remark}
In the numerical experiments reported in \Cref{sec:int_quad}, we compare the exact Gaussian formulas with the results 
obtained by 10-point Gauss--Hermite quadrature. 
Their computational cost is comparable in the uni-variate case, but not in multiple dimensions.
For quantum molecular simulations in higher dimensions, however, exact formulas 
are the method of choice, see e.g. \cite{Richings_etal2015,VanicekZhang2025}.
\end{remark}

\subsection{Initializing a Gaussian mixture}\label{sec:init_Gauss}

The initial parameter vector
$\theta_0\in\Theta$ determines the initial approximation
$u(\theta_0,\cdot)\in\mathcal M$ of the given datum $\psi_0$.
As highlighted by the error estimates of Theorems~\ref{prop:discr-opt} 
and~\ref{prop:opt-discr}, the quantity 
\[
\|\psi_0-u(\theta_0,\cdot)\|_{\mathcal H}^2
\]
constitutes one of the relevant error sources. A sufficiently accurate
initialization is therefore essential for observing the asymptotic
behavior predicted by the theory. 

\paragraph{Gaussian notation.}
In this section, for a clear geometric interpretation, we will denote the Gaussians in terms of their width coefficient $a\in\siegel$, center $\nu\in\C$ and amplitude $C\in\C^*$ and to avoid confusion with the previous use of $\theta$, we denote this parametrization by $\xi$. Namely, we denote the Gaussian basis by
\[
b(\xi,x) = C e^{-a(x-\nu)^2},\quad x\in\R,
\]
with $\xi=(a,\nu,C)$.
The theta parameter from the previous subsection in terms of these data is $\theta=(a,2a\nu,-a\nu^2+\log C )$.

For the initial Gaussian approximations 
$u(\xi_0,\cdot) = \sum_{l=1}^L b(\xi_0^l,\cdot)$ with $\xi_0^l = (a_0^l,\nu_0^l,C_0^l)$
used in the numerical experiments of Sections~\ref{sec:Num_Experiments} and \ref{sec:Num_Experiments_II}, we have employed the heuristic construction 
described below. It assumes that the critical points of $\psi_0$ are known. 
The cost function under consideration in the following construction is
\[
F(\xi) = \|\psi_0-u(\xi,\cdot)\|_{\mathcal H}^2,
\]
with the parameter space $\Xi=(\siegel\times\C\times\C^*)^L$.

\subsubsection{Basic initialization} \label{sec:init_general}
The objective of the following procedure is not to solve the global optimization problem
$\xi_0=\arg\min_{\xi\in\Xi} F(\xi)$,
which may be ill-posed or characterized by many local minima, but rather to generate a
well-conditioned initial guess with subsequent local optimization. In addition to approximation quality, the construction seeks to avoid strong overlap between basis functions, poor conditioning and thus adversely affects on the parametrize-then formulation.

\begin{itemize}
    \item[(i)] \textbf{Select Gaussian number and width.} Set the number of basis functions $L\geq 1$ and set the width parameters $a_0^l=a$ for all $l=1,\ldots,L$ for some constant $a>0$. Of course, the width need not be uniform and could be chosen in temperate progression as in \cite{Rowan_etal2020}. The number of basis functions $L$ could be the number of critical points of $\psi_0$.
    \item[(ii)] \textbf{Place Gaussian centers.} Center the Gaussians at the critical points of $\psi_0(x)$, which we denote as $\nu_0^l$ for $l=1,\ldots,L$.
    \item[(iii)] \textbf{Determine amplitudes.} Choose $\gamma_0^l$ as the solutions of the set of equations $\psi_0(\nu_0^l) - u(\xi,\nu_0^l)=0$ with $l=1,\ldots,L$.
    \item[(iv)] \textbf{Local optimization.} Improve the approximation $\xi_0$ by minimizing the cost function $F(\xi)$, using the values of steps (i)--(iii) as the initial guesses, i.e., $\xi^{(0)}=(a_0^l,\nu_0^l,C_0^l)_{l=1,\ldots,L}$. 
\end{itemize}

For quantum mechanical problems, we assume that $\Vert\psi_0\Vert=1$ and the local optimization problem will be constrained to $\Vert u(\xi,\cdot)\Vert=1$. 

\subsubsection{Refinement strategy}\label{sec:refine}
If the attained residual remains above a prescribed tolerance,
repeat steps (i)--(iv) with a larger number of Gaussian basis
functions. The refinement is performed by replacing each Gaussian
component by an odd number $m\ge 1$ of localized sub-components whose centers are
distributed according to the roots of an appropriate Hermite
polynomial. Since derivatives of Gaussians are Gaussian-envelope
polynomials, these roots provide natural locations for the refinement, see also \Cref{sec:init_poly}.
Specifically, for a Gaussian component $b(\xi_0^l,\cdot)$ with
$\xi_0^l=(a_0^l,\nu_0^l,C_0^l)$, the refined centers are
initialized as
\[
\nu_0^{l'}
=
\nu_0^l+\frac{x_{l'}}{\sqrt{a_0^l}},
\qquad
l'=1,\ldots,m,
\]
where $x_{l'}$ are the roots of the physicists' Hermite polynomial
$H_m$. The remaining parameters are then obtained from the
optimization problem of step (iv), subject to fixed centers and width parameters $a_0^{l'}$ subject to the linear constraint $(1+\varepsilon)a^l_0 \leq a^{l'} \leq a_{\max}$ for some small $\varepsilon>0$ and an upper bound $a_{\max}>(1+\varepsilon)a_0^l$.
Without the constraints, the optimizer will yield a single Gaussian as the trivial solution, e.g., the centers of the sub-components will cluster around the same point or one sub-component will converge to the original Gaussian and the rest will take weight zero. 
For all the experiments presented in this paper, we have used initial constants $a=1.5$ for step (i) of the basic initialization and initial widths parameters $a_0^{l'}=1.5 a_0^l$ for the Gaussian subcomponents (refinement stage), which we also use as the upper bound, i.e., $a_{\max}=1.5 a_0^l$.

\begin{remark}
Although the present construction is only used during initialization, a similar approach could be used for dynamically adding or removing Gaussians during the time evolution. Such adaptive strategies are not considered here but constitute interesting future work.
\end{remark}

\subsubsection{Polynomials with a Gaussian envelope} \label{sec:init_poly}
The initialization described above relies on
the availability of characteristic locations of the target function,
most notably its critical points. For general initial data these points
may have to be determined numerically. For the class of
Gaussian-envelope polynomials, however, they can be computed
systematically. This class is relevant for the numerical examples considered
below and is preserved under differentiation. We consider initial data of the form
\[
\psi_0(x)=P_n(x)b(\xi^{\rm e},x), \qquad x\in\mathbb R,
\]
where $P_n$ is a polynomial with complex coefficients and $b(\xi^{\rm e},\cdot)$ a Gaussian with parameters\footnote{Here, we use superscripts for the Gaussian envelope to avoid confusion with the Gaussian components of the approximation $u(\xi,\cdot)$.} 
$\xi^{\rm e}=(a^{\rm e},\nu^{\rm e},C^{\rm e})\in \mathbb H\times\C\times \C^*$.
Let $P_m$, $m\ge n$, be defined recursively by
\[
P_{m+1}(x)
=
2a^{\rm e}(x-\nu^{\rm e})P_m(x)-P_m'(x),
\qquad m\ge n,
\]
with initial polynomial $P_n$ given by the representation of
$\psi_0$. Then repeated differentiation gives
\[
\psi_0^{(j)}(x)
=
(-1)^j P_{n+j}(x)b(\xi^{\rm e},x),
\qquad j\ge 0 .
\]
Since the Gaussian factor does not vanish, the zeros of
$\psi_0^{(j)}$ coincide with the zeros of $P_{n+j}$. 

For initial data of this form, the locations required in the
initialization procedure of \Cref{sec:init_general} are therefore obtained from
the roots of $P_{n+1}$. Likewise, the roots
needed in the refinement strategy of \Cref{sec:refine} arise from the same
recursion applied to individual Gaussian components. In the special case $P_n\equiv 1$, the recursion generates scaled
Hermite polynomials. Thus, the Hermite-root placement used in the
refinement step can be viewed as the pure-Gaussian instance of the
same general mechanism.

\subsubsection{Illustration of the initialization procedure}

We illustrate the initialization strategy for two of the initial conditions of the numerical experiments in 
\Cref{sec:Num_Experiments} and \Cref{sec:Num_Experiments_II}. The aim is to construct a compact Gaussian mixture that accurately
represents the initial condition while providing sufficient
expressiveness for the subsequent time evolution. Since the dynamics
may rapidly leave the single-Gaussian manifold, the initialization
is guided not only by the approximation quality at time $t=0$ but
also by the anticipated complexity of the evolution.

\begin{figure}[htb!]
    \centering
    \begin{subfigure}{0.45\textwidth}
    \includegraphics[width=\linewidth]{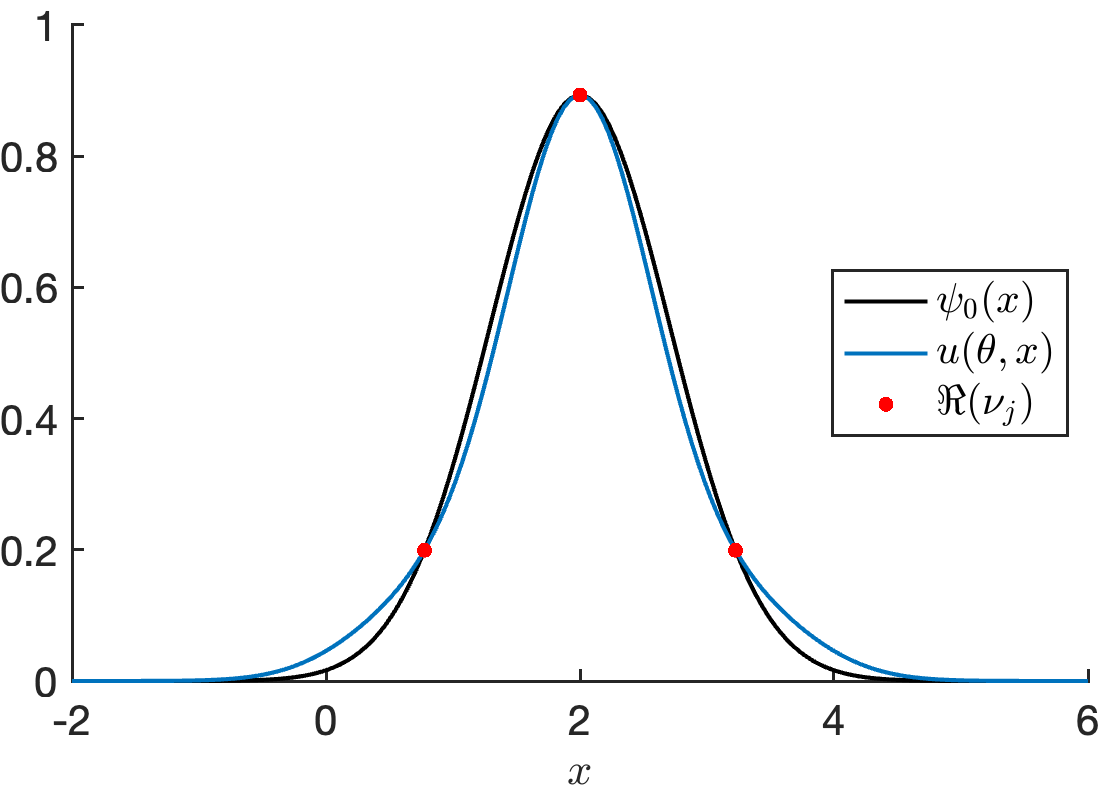}
    \caption{Example 1: steps (i)--(iii) + refinement} \label{fig:init_Gauss_a}
    \end{subfigure}
\hfill
    \begin{subfigure}{0.45\textwidth}
    \includegraphics[width=\linewidth]{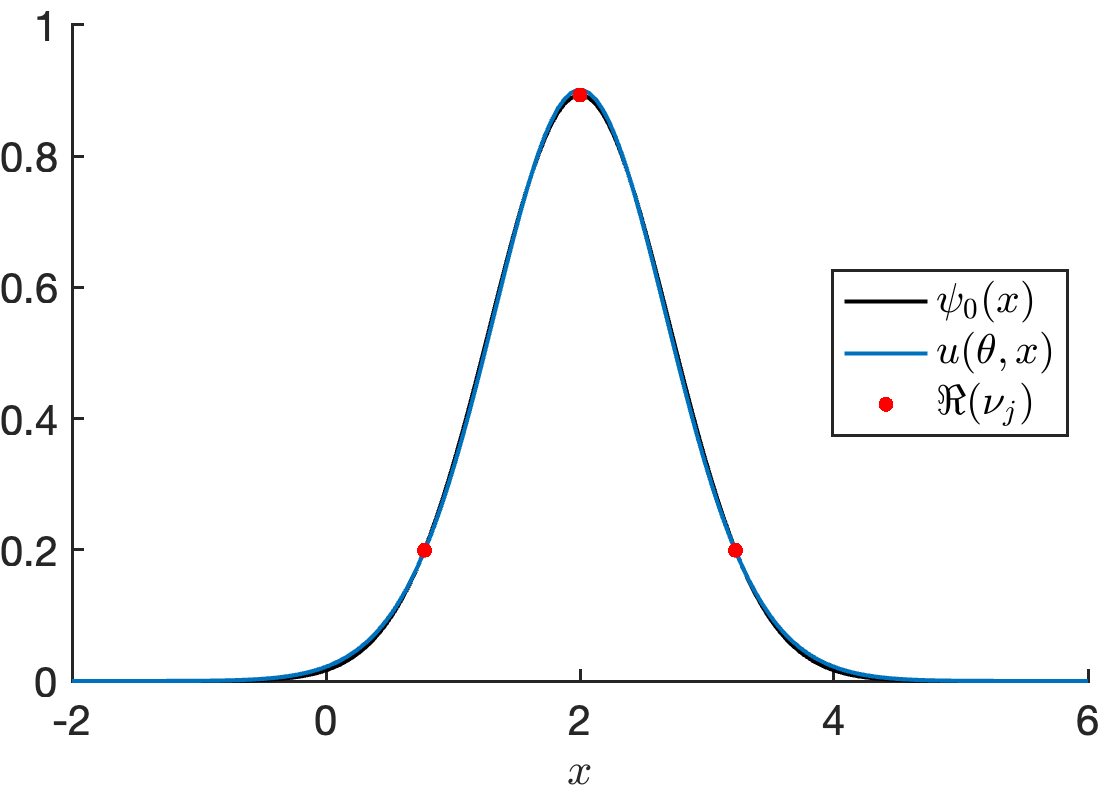}
    \caption{Example 1: step (iv)} \label{fig:init_Gauss_b}
    \end{subfigure} 
    
    \vspace{0.5em}
    
    \begin{subfigure}{0.45\textwidth}
    \includegraphics[width=\linewidth]{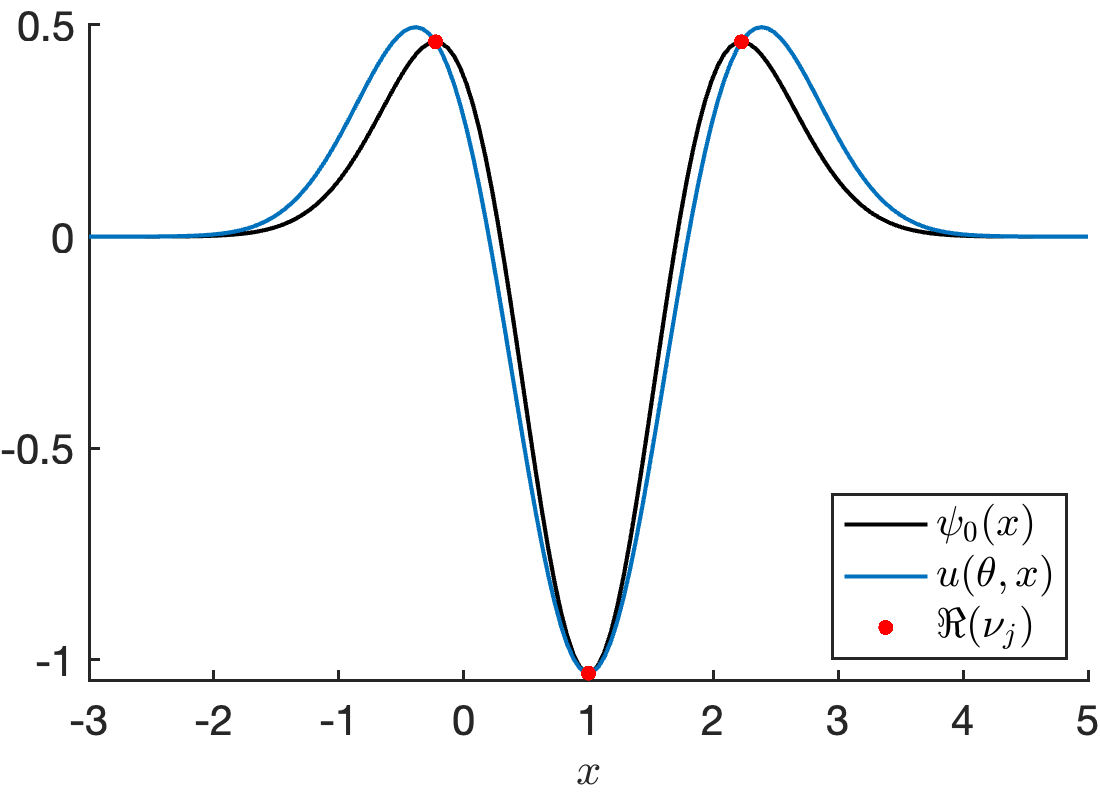}
    \caption{Example 2: steps (i)--(iii)} \label{fig:init_Gauss_c}
    \end{subfigure} 
\hfill
    \begin{subfigure}{0.45\textwidth}
    \includegraphics[width=\linewidth]{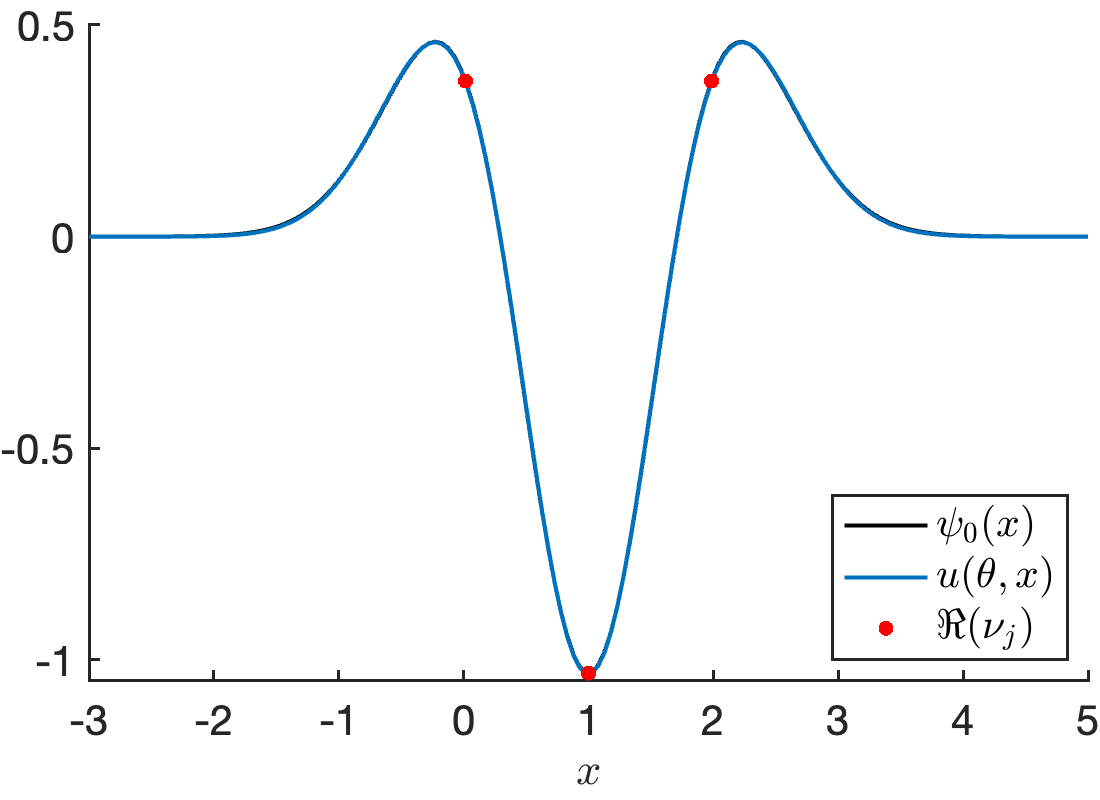}
    \caption{Example 2: step (iv)} \label{fig:init_Gauss_d}
    \end{subfigure} 

    \caption{Approximation of two types of wave packets by $L=3$ Gaussians. The top and bottom row show the approximation of the single Gaussian $\psi_0^{\rm Ex.1} = \mathcal N_0 e^{-(x-2)^2}$ and of the excited state $\psi_0^{\rm Ex.2} =\mathcal N_0 H_2(x-1)e^{-(x-1)^2}$, respectively, where $\mathcal N_0$ are normalization constants. The plots to the left show the approximation with steps (i)--(iii), and the plots to the right show the improvement with step (iv).} \label{fig:init_Gauss}
\end{figure}

\Cref{fig:init_Gauss} illustrates the initialization procedure.
Here, we show the approximation with three Gaussian components to emphasize the improvement with step (iv), since using more Gaussians already provides a very good approximation with only steps (i)--(iii).
The critical-point
construction provides suitable locations for the Gaussian components,
and the subsequent constrained optimization improves the approximation
without introducing degenerate Gaussian configurations. The resulting
mixtures serve as initial values for the numerical experiments in
Sections~\ref{sec:Num_Experiments} and~\ref{sec:Num_Experiments_II}.

 
\section{Qualitative Schrödinger dynamics}\label{sec:Num_Experiments}

We conduct numerical experiments for the uni-variate Schrödinger equations
\[
\I \partial_t\psi(t,x) = \left( -\frac12\partial_x^2 + V(t,x)\right) \psi(t,x),\quad \psi(0,x) = \psi_0(x)
\]
with three choices for the potential function $V(t,x)$ and various initial conditions $\psi_0\in \hil$ with $\hil = L^2(\R,\C)$. 
The purpose of the present section is to illustrate the qualitative behavior for several representative
Schrödinger dynamics. Quantitative investigations of residual minimization,
time discretization, quadrature errors, norm and energy conservation are presented in \Cref{sec:Num_Experiments_II}.

\paragraph{Experimental design.}
The error estimates of \Cref{prop:discr-opt,prop:opt-discr} separate the contributions of time discretization and residual minimization. In practical computations, however, additional sources of error are present. These include the initialization error
as well as quadrature errors arising when residual norms and gradients are evaluated numerically. Throughout the experiments, these additional errors are controlled sufficiently so that the asymptotic behavior predicted by the theory becomes observable. Furthermore, the number $L$ of Gaussian basis functions was chosen such that the approximation manifold is sufficiently expressive while maintaining a reasonably well-conditioned parametrization. Since the parametrize-then-discretize formulation depends explicitly on the conditioning of the parametrization gradient, this choice enables a meaningful comparison of the two residual formulations under conditions where neither method is \emph{a priori} disadvantaged.

\bigskip
We have considered the following three test cases with different initial conditions.

\paragraph{Test 1: Harmonic oscillator (exactly representable benchmark).}  The potential is the standard harmonic one,  
    \[
    V(t,x) = \frac12 x^2,\quad x\in\R,\quad t\in[0,10].
    \]
    We recall that the ground state is the centered Gaussian $\psi_{\rm GS}(x)=\mathcal{N}_0 e^{-\frac{1}{2} x^2}$ with energy $E_{\rm GS}=\langle\psi_{\rm GS},\hat{H}\psi_{\rm GS}\rangle_\hil=\frac{1}{2}$. We consider two different initial conditions, for which harmonic quantum motion can be explicitly resolved within a nonlinear parametrization. The first one lies within the range of a Gaussian parametrization, whereas the second one requires Hagedorn functions, see \cite[Theorem~2.5]{Hagedorn1998} or the review \cite[Chapter 4]{LL}.
\begin{description}
    \item[Shifted Gaussian.] The shifted Gaussian $ \psi_0(x)=\mathcal{N}_0 e^{-(x-2)^2}$ has its
    energy $E_0 = \langle\psi_0,\hat{H}\psi_0\rangle_\hil = 2.625$ significantly above the ground state energy and has a slimmer initial width than that of the harmonic ground state. \Cref{fig:harmonic} shows the evolution of this state. The numerical solution indicates that the width coefficient $\Re(a)$ oscillates in the interval $[\frac{1}{4},1]$ and its center oscillates (with a different period) in the interval $[-2,2]$.
    \item[Shifted excited state.] The initial condition $\psi_0(x)=\mathcal{N}_0 H_2(x-X_0) e^{-\frac{1}{2}(x-X_0)^2}$ is similar to the second excited state, except for the initial shift to the right given by $X_0=1$. 
    We evolve this state considering a three-Gaussian approximation, see \Cref{sec:init_poly} for the determination of the initial value $\theta_0$.
    \Cref{fig:harmonic} shows that the constant-time sections of the density $|\psi(t,x)|^2$ remains identical to that of the first excited state $\psi_{1S}$, except for a shift by $X\in [-X_0,X_0]$; namely, the probability density is of the form $|\psi(t,x)|^2=|\psi_{1S}(0,x-X(t))|^2$.
\end{description}

This example serves as a benchmark case in which the Gaussian manifold
is invariant under the exact dynamics. Consequently, the residual can
be minimized to machine precision and the observed errors are
dominated by time discretization.

\begin{figure}[htb!] 
    \centering
    \includegraphics[width=0.45\linewidth]{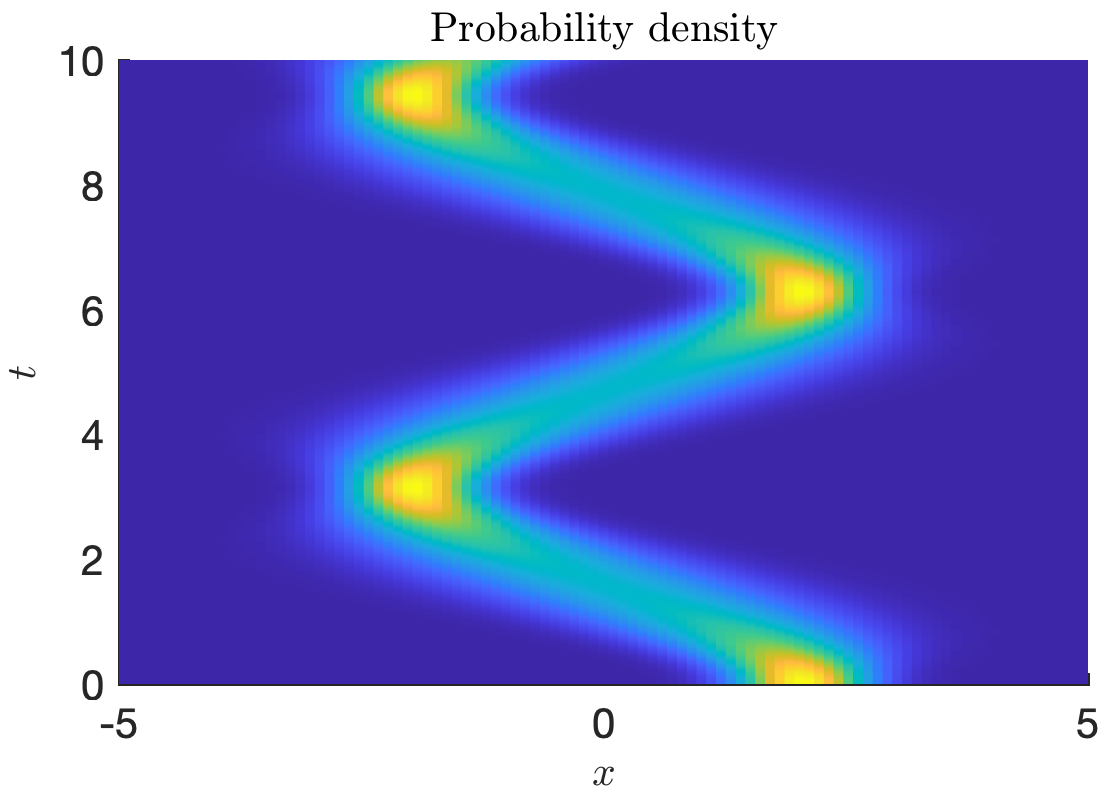}
    \includegraphics[width=0.45\linewidth]{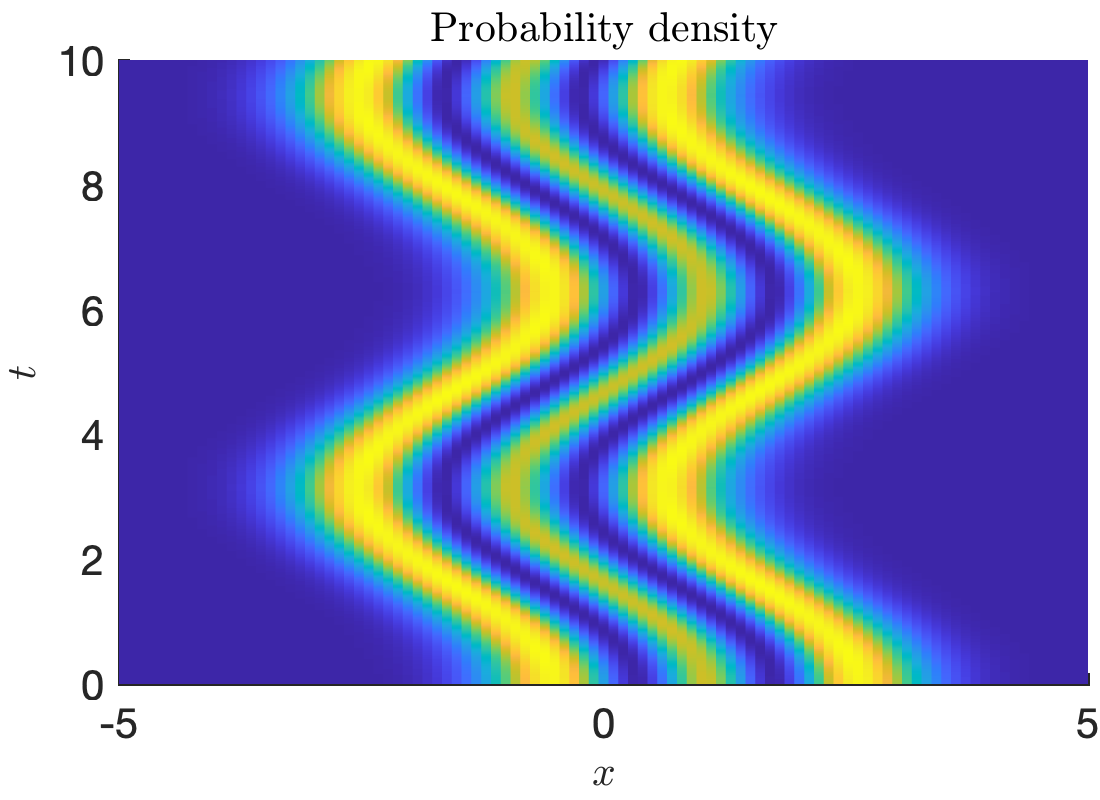} 
    \caption{Contour plots of the position density $|\psi(t,x)|^2$ for the harmonic oscillator. The left plot shows the dynamics for the Gaussian initial data, the right plot for the excited initial condition with a three-Gaussian approximation. } \label{fig:harmonic}
\end{figure}

\paragraph{Test 2: Double-well (interference and splitting).} The potential is a double-well
    \[
    V(t,x) = \alpha_4 x^4 + \alpha_2 x^2,\quad x\in\R,\quad t\in[0,10].
    \]
    Following \cite{Joubert2022,FLL}, we choose the values for the potential as $\alpha_2=-1/8$ and $\alpha_4=(\alpha_2)^2$. 
    The local minima $y_{\min}=-1/4$ are then located at $x_{\pm}=\pm 2$, and the local maximum $y_{\max}=0$ is located at $x_0=0$. 
    We note that double-well dynamics cannot be exactly resolved within nonlinear parametrization. 
\begin{description}
    \item[First excited state.]  We start by studying the evolution of the first excited state in order to visualize the stability of the numerical schemes presented in the paper. To approximate this state, we consider the two-Gaussian approximation $u_0(x) = \mathcal{N}_0 ( e^{-a_*(x-\nu_*)^2} - e^{-a_0(x+\nu_*)^2} ) $, where $a_*=0.4475$ and $\nu_*=1.8732$ were determined numerically by minimizing the energy function $E(\theta)=\langle u(\theta,\cdot),\hat{H} u(\theta,\cdot) \rangle$ and imposing an antisymmetry condition on the Gaussian components of $u(\theta,\cdot)$. The energy of this state is $E_{0} = 0.1668$.  The numerical result shown on the left-hand side of \Cref{fig:double-well} shows the stationary nature of this state, i.e., the probability density function is time-independent.
    \item[Gaussian superposition.] Then we continue with the superposition 
of Gaussians centered in the left and the right well, namely, $\psi_0 (x) = \mathcal{N}_0 ( e^{-(x-2)^2} - e^{-(x+2)^2} ) $. 
The energy of this state is $E_0=0.3181$, which is larger than the first excitation energy. The centers of the Gaussian components are shifted relative to the centers $\pm\nu_*$ of the first excited state, and their initial width is significantly slimmer. 
\end{description}

This example illustrates a regime in which the solution leaves the
single-Gaussian manifold and motivates the use of multiple Gaussian
components.

\begin{figure}[htb!]
    \centering
    \includegraphics[width=0.45\linewidth]{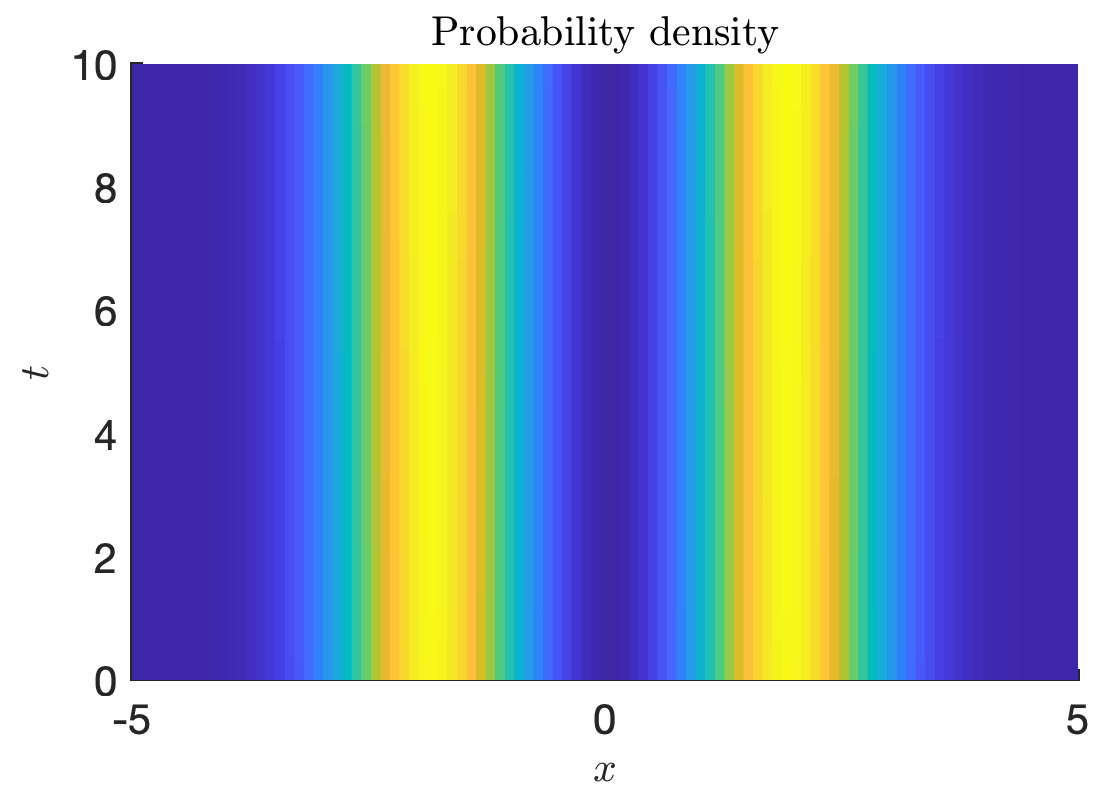} 
    \includegraphics[width=0.45\linewidth]{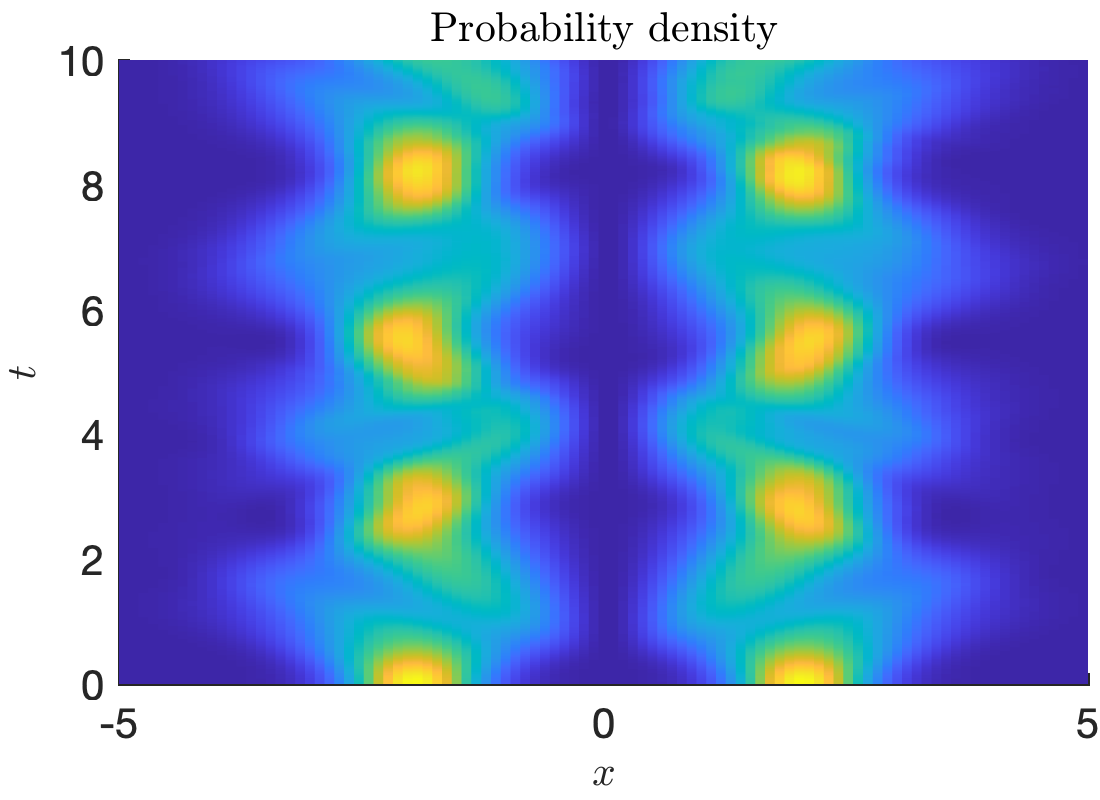} 
    \caption{Contour plots of the position density $|\psi(t,x)|^2$ for the double-well. The left plot shows the dynamics of the first excited state (two-Gaussian approximation), the right plot shows the superposition (six-Gaussian approximation). } \label{fig:double-well}
\end{figure}

\paragraph{Test 3: Hyperbolic cosine (strongly non-polynomial dynamics).} The potential is the hyperbolic cosine 
\[
V(t,x) = \cosh(x)-1, \quad x\in\R, \quad t\in [0,10].
\]
We choose this potential since its second-order Taylor expansion about $x=0$ is the harmonic oscillator. However, the dynamics will be different in general.
We study the evolution of a single Gaussian initial condition $\psi_0(x)=\mathcal{N}_0 e^{-\frac{1}{2}(x-1)^2}$ considering a multi-Gaussian approximation, i.e., we consider a refinement as described in \Cref{sec:refine}. For reference, the energy of this state is $E_0=1.2313$, which is larger than the energy $E_{GS}=0.5301$ of the single-Gaussian approximation of the ground state $\psi_{GS}(x)= \mathcal{N}_0 e^{-0.5591 x^2}$. 
The same initial condition in a harmonic oscillator potential would produce a solution whose density $|\psi(t,x)|^2$ constant-time sections do not change, except for a shift in the center of the Gaussians $X\in [-1,1]$, as in the example shown in the left-hand side of \Cref{fig:harmonic}. 
This is no longer the case in a hyperbolic cosine potential, as one can observe in \Cref{fig:cosh}. 
Moreover, one single Gaussian is not enough to describe the correct behavior of the wave function, but a refinement with at least three Gaussians is needed. The plot on the right panel shows the evolution for a $L=5$ approximation.

\bigskip
This example demonstrates that the framework remains effective even
when the dynamics are no longer generated by a polynomial Hamiltonian
and exact representability is lost.

\begin{figure}[htb!] 
    \centering 
    \begin{subfigure}{0.45\textwidth}
    \includegraphics[width=\linewidth]{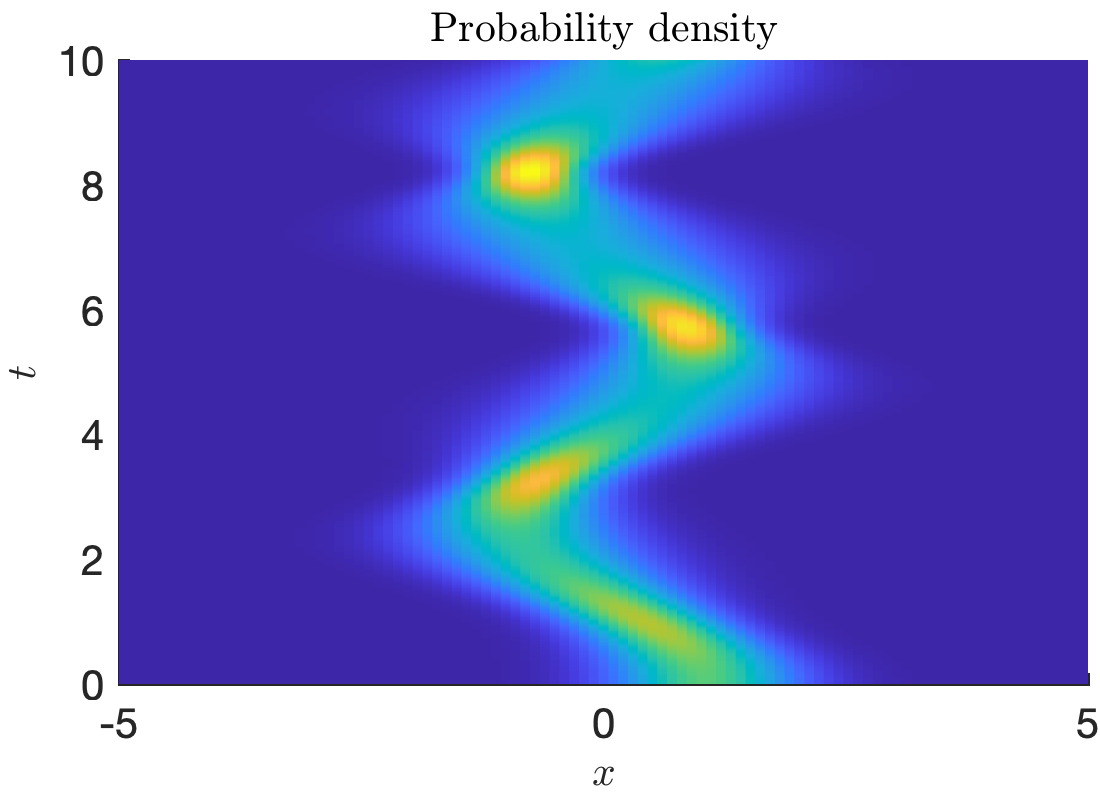}
    \caption{$L=1$}
    \end{subfigure}
\hfill
    \begin{subfigure}{0.45\textwidth}
    \includegraphics[width=\linewidth]{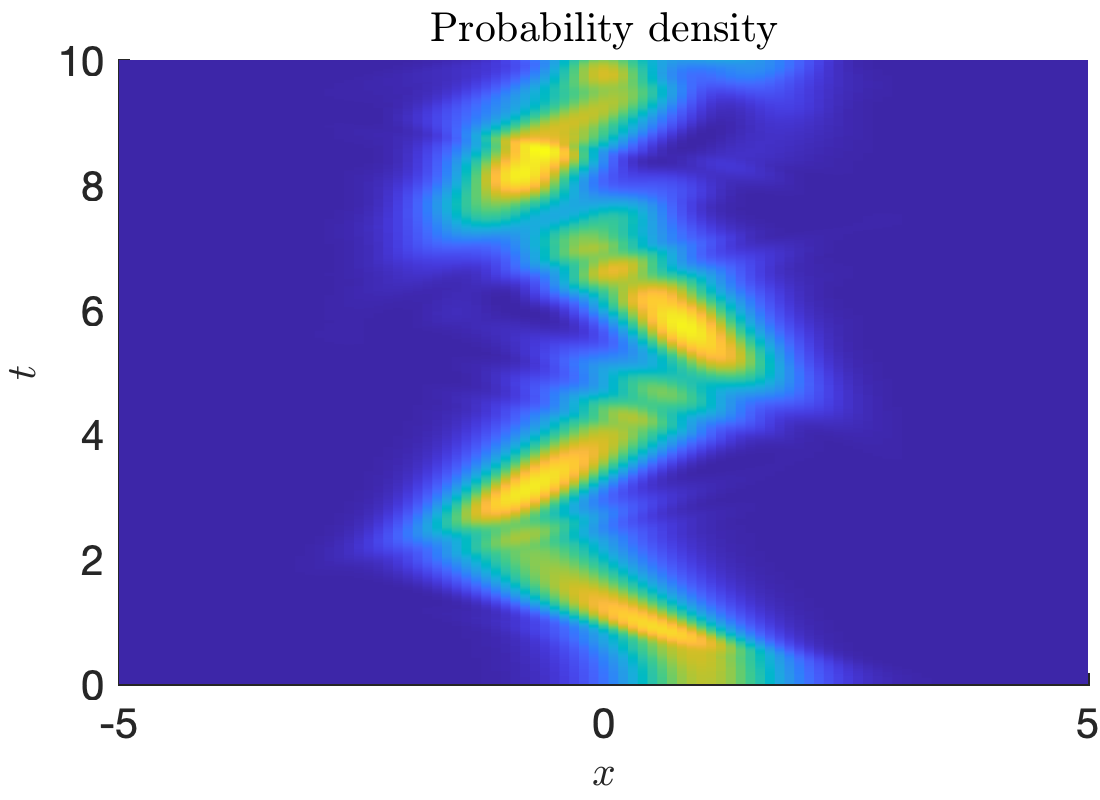}
    \caption{$L=5$}
    \end{subfigure}

    \caption{Evolution of a single Gaussian initial condition in a hyperbolic cosine potential using multi-Gaussian approximations. Step-size: $h=10^{-3}$.} \label{fig:cosh}
\end{figure}

\paragraph{Reproducibility: algorithmic parameters.}
All numerical experiments have been done using a unified Matlab code\footnote{\url{https://gitlab.lrz.de/00000000014C092A/discretize-then-and-parametrize-then}} admitting the following choices: (i) the discretization/parametrization approach, (ii) the weight $\zeta\in [0,1]$, (iii) the step-size $h>0$, (iv) the computation of the inner products using quadratures or closed-form formulas (if the potential is a polynomial).
The optimization part described in \Cref{sec:Optimization_process} has been done with Matlab's \texttt{fminunc} built-in function with a fixed maximum number of iterations \texttt{MaxIterations}=$200$ and a constant value $10^{-8}$ for the tolerances
\texttt{OptimalityTolerance}, \texttt{StepTolerance}, \texttt{FunctionTolerance}.
Unless explicitly stated, all the experiments shown in this section have been done using the following choices: (i) discretize-then approach, (ii) weight $\zeta=\frac12$, (iii) step-size $h=10^{-2}$, (iv) closed-form formulas for the polynomial potentials and quadratures for the non-polynomial one.


\section{Quantitative numerical validation}\label{sec:Num_Experiments_II}

The purpose of this section is to validate the predictions of Theorems~\ref{prop:discr-opt} and~\ref{prop:opt-discr}. 
as well as Proposition~\ref{prop:norm_energy}.
In particular, we investigate the dependence
of the global approximation error on the residual magnitude, the time-step size $h$, the parameter $\zeta$, and the accuracy of the
quadrature approximation used in the residual evaluation.  As quantitative measures we will use the value of the cost function at the optimized parameter $\theta_k$ for time step $t_k = kh$, $k\ge 0$, as well as the norm and the energy deviation,
\[
F_{k}(\theta_{k+1}) = \Vert r_{k}(\theta_{k+1},\cdot)\Vert^2,\quad \left|\Vert u_k\Vert -1\right|,\quad \left|E_k/E_0-1\right|,
\]
where $u_k = u(\theta_k,\cdot)$ and $E_k = \langle u_k,\hat{H}u_k\rangle_\hil$ is the energy associated with the Hamiltonian $\hat{H}=-\frac12\partial_x^2 + V(x)$. We consider maximal and averaged quantities, e.g.
\[
\max_{k\in K} F_{k}(\theta_{k+1}),\quad {\rm avg}_{k\in K} F_k(\theta_{k+1}) = \frac{1}{N} \sum_{k=1}^N F_{k}(\theta_{k+1})
\]
with $N=\lceil T/h \rceil$ the number of time steps $t_k$ until final time $T$ and $K = \{1,\ldots,N\}$. We will see that maximal and averaged quantities are comparable for all our experiments. 

\subsection{Influence of the manifold dimension}
Increasing the number of Gaussian basis functions enlarges the approximation manifold and therefore reduces the residual contribution appearing in the error estimates of Theorems~\ref{prop:discr-opt} and~\ref{prop:opt-discr}. \Cref{fig:L-comparison} illustrates this effect for two of the test cases: the double-well with initial condition $\psi_0(x)=\mathcal N_0 e^{- (x-2)^2}$ and the hyperbolic cosine potential with $\psi_0(x)=\mathcal N_0 e^{-\frac12 (x-1)^2}$. While a single Gaussian already captures the dominant localization of the wave packet, a small number of additional components improves the approximation of its finer structure. 

\begin{figure}[htb!]
    \centering
    
    \begin{subfigure}{0.45\textwidth}
    \includegraphics[width=\linewidth]{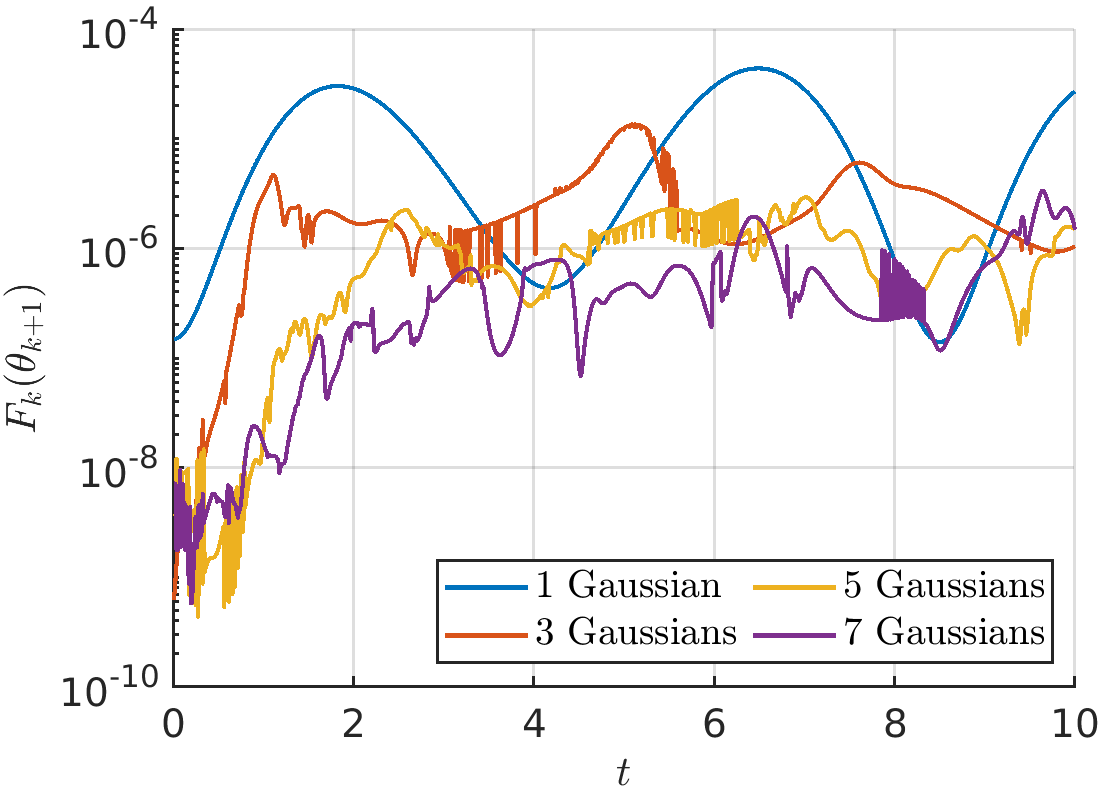}
    \caption{Double-well: cost function} \label{fig:L-comp-a}
    \end{subfigure}
\hfill
    \begin{subfigure}{0.45\textwidth}
    \includegraphics[width=\linewidth]{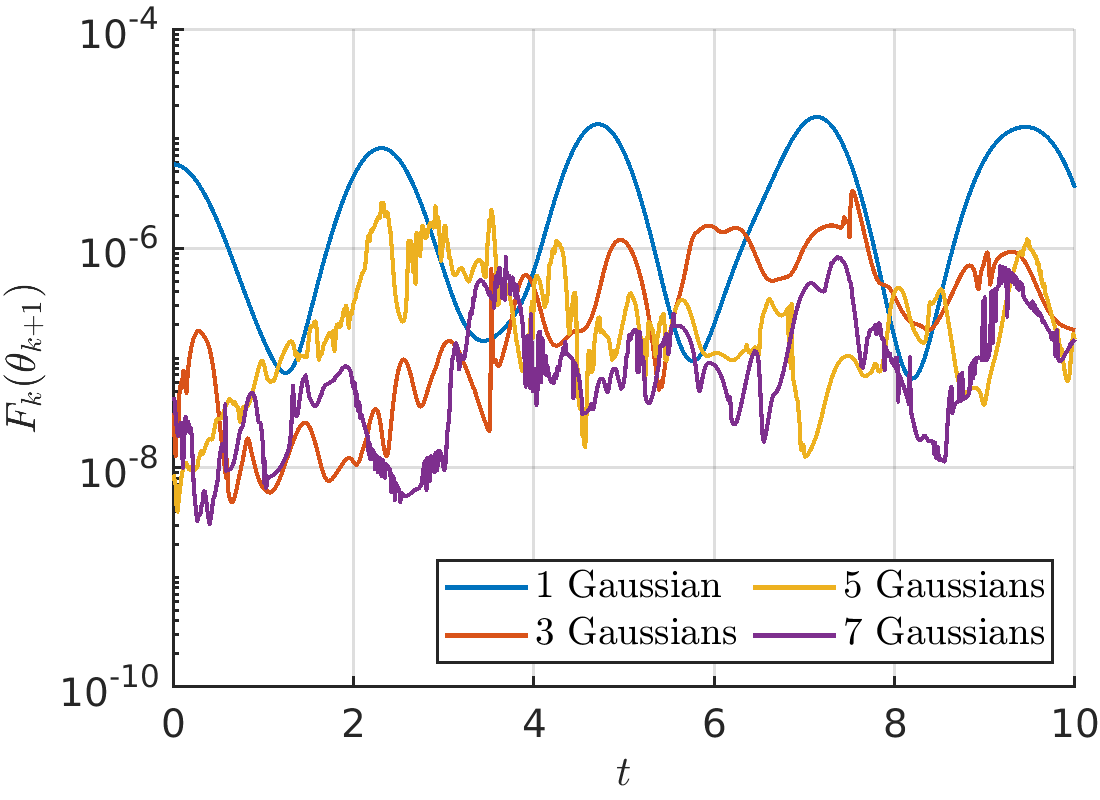 } 
    \caption{Cosh: cost function} \label{fig:L-comp-c}
    \end{subfigure}

    \vspace{0.5em}

    \begin{subfigure}{0.45\textwidth}
    \includegraphics[width=\linewidth]{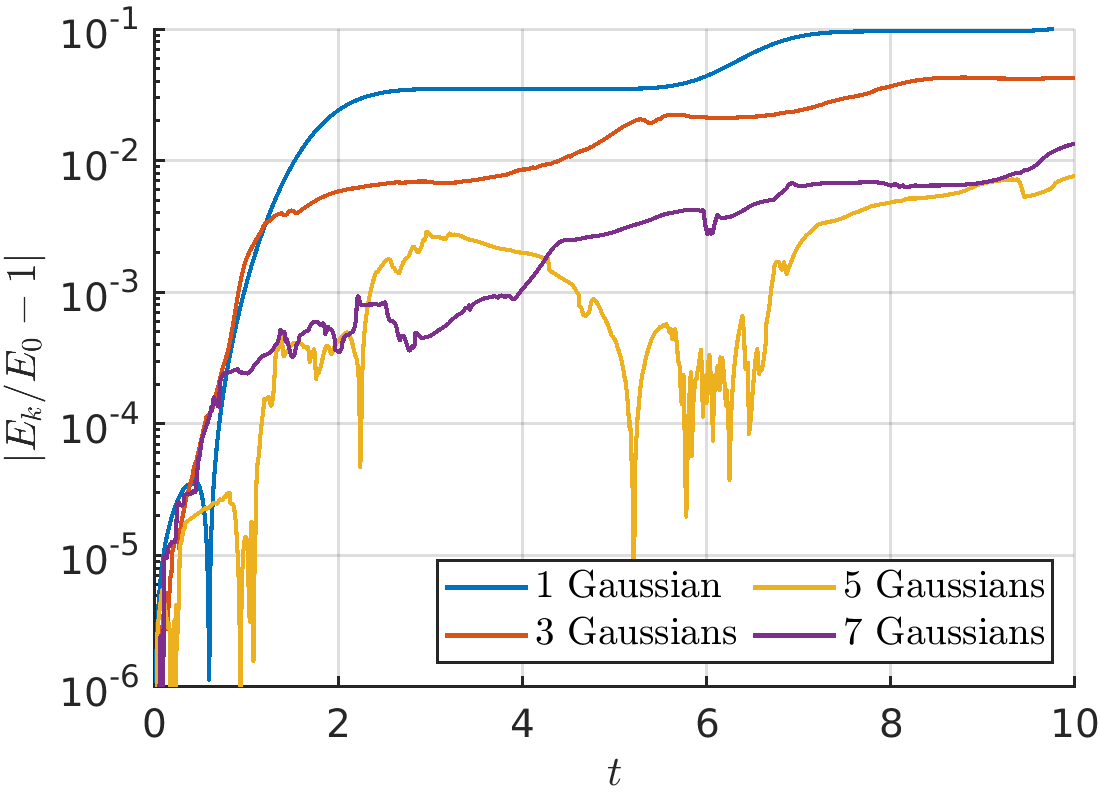}
    \caption{Double-well: energy deviation} \label{fig:L-comp-b}
    \end{subfigure} 
\hfill
    \begin{subfigure}{0.45\textwidth}
    \includegraphics[width=\linewidth]{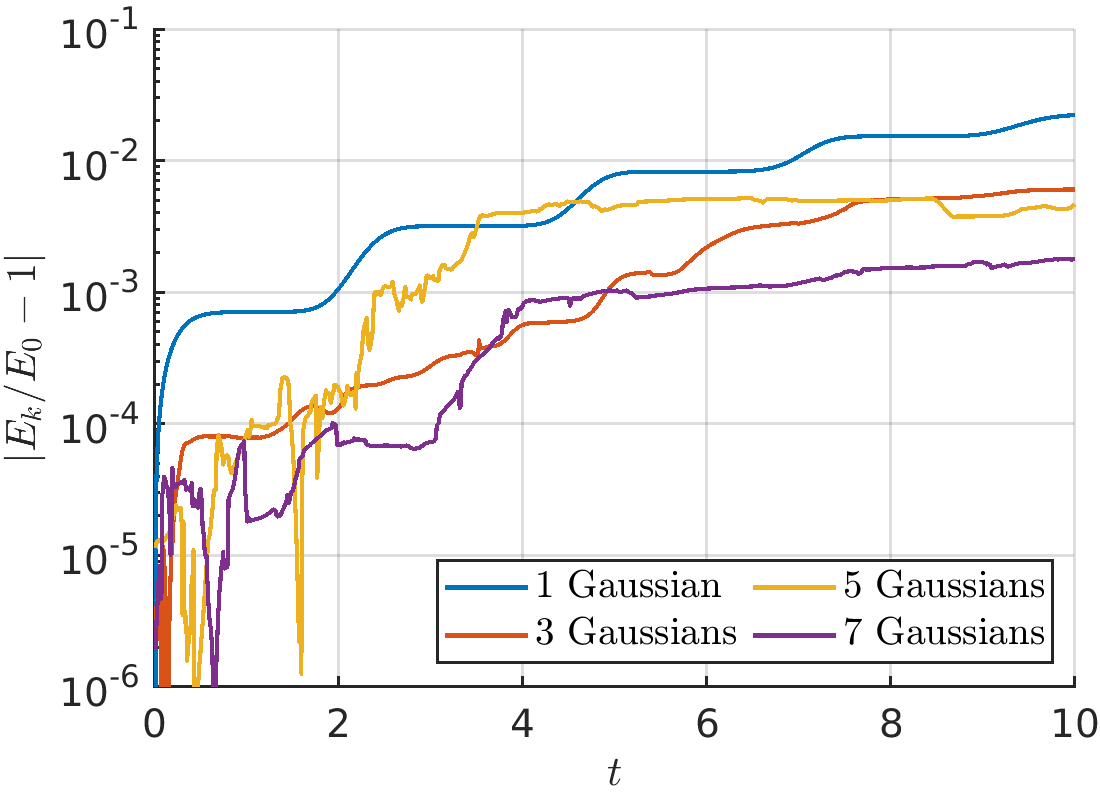} 
    \caption{Cosh: energy deviation} \label{fig:L-comp-d}
    \end{subfigure}
    
    \caption{Discretize--then approach. Comparison of the attained residual and energy deviation of the evolution of a single initial Gaussian condition in a double-well (left) and hyperbolic cosine (right) potential using $L=1,3,5,7$ Gaussian basis functions. 
    The plots in the top row show the squared residual norm $F_k(\theta_{k+1}) = \|r_k(\theta_{k+1},\cdot)\|^2$. Numerical parameters:
    $\zeta=\frac12$ and $h=10^{-2}$. }
    \label{fig:L-comparison}
\end{figure}

\subsection{Influence of the parameter $\zeta$}

The exact representability of the harmonic oscillator suggests that the midpoint choice $\zeta=\frac12$ is optimal whenever the residual vanishes. 
This situation occurs for the harmonic oscillator, where the Gaussian manifold is invariant under the exact dynamics. Numerical experiments readily confirm this prediction (results not shown). We now investigate whether the midpoint choice remains advantageous when the residual cannot be eliminated completely, which is the generic situation. For this, we consider the hyperbolic cosine and the double-well potential. For both experimental set-ups, the parameter $\zeta$ is sampled on a grid in $[0,1]$ with increased resolution near $\zeta=\frac12$. The final time is $T=2$ and the time step-size is fixed at $h=10^{-2}$.

\paragraph{Hyperbolic cosine potential.} The results for the hyperbolic cosine potential use the initial condition $\psi_0(x)=\mathcal{N}_0 e^{-\frac12 (x-1)^2}$ 
and work with a three-Gaussian approximation, $L=3$. They are shown in
\Cref{fig:zeta-cosh}. In contrast to the harmonic oscillator, the residual cannot be driven to zero; the (averaged) attained minima of the cost function are of order $10^{-7}$ for both residual formulations. However, we note that the cases $\zeta>0.5$ 
did not converge for the discretize-then formulation. There, a finer step-size than $h=10^{-2}$ is needed for convergence. 
Despite these comparable residual levels (for the cases that converge), the resulting norm and energy deviations exhibit a pronounced dependence on $\zeta$. 
In particular, the parametrize-then approach produces smaller norm and energy deviations for almost all values of $\zeta$ (except around $\zeta=\frac12$, where discretize-then is visibly better). For both the discretize-then and the parametrize-then formulations, the smallest deviations are attained for values of $\zeta$ close to $\frac12$. In the discretize-then formulation, the minimum is achieved at $\zeta=\frac12$, whereas for the parametrize-then formulation it occurs at a nearby value. 

\begin{figure}[htb!]
  \centering

  \begin{subfigure}{0.45\textwidth}
    \includegraphics[width=\linewidth]{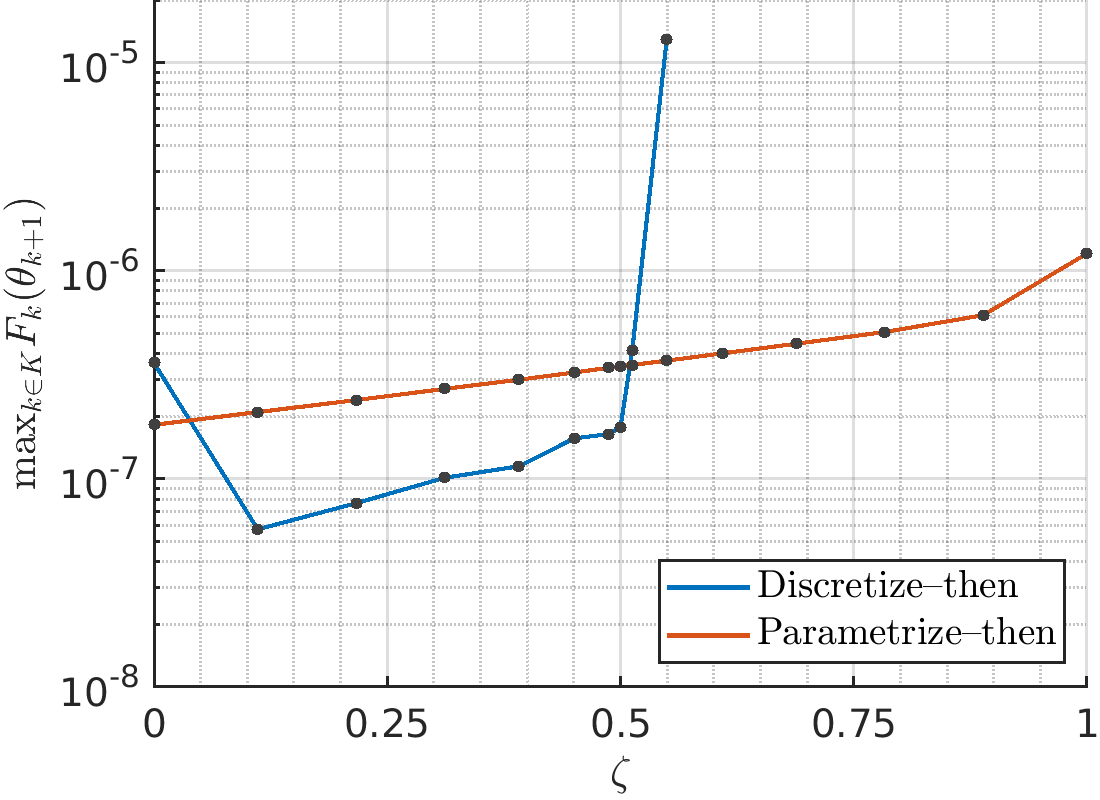} 
    \caption{Maximal cost function}
  \end{subfigure}
  \hfill
  \begin{subfigure}{0.45\textwidth}
    \includegraphics[width=\linewidth]{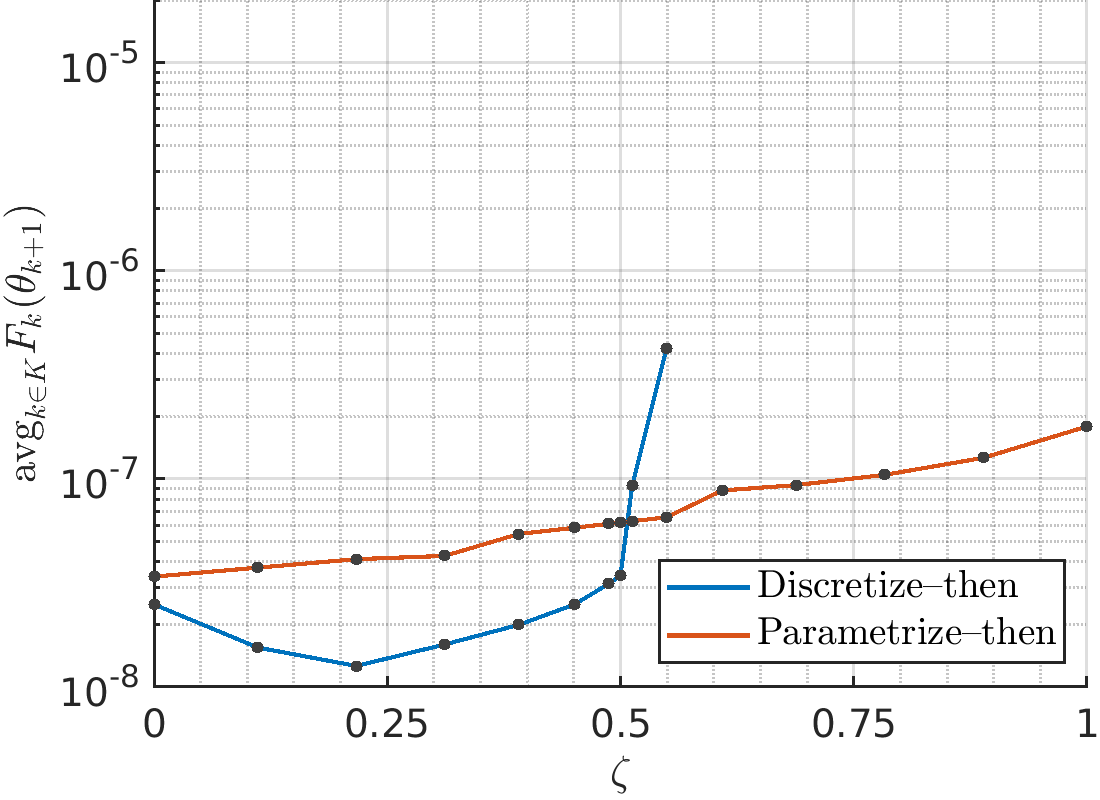}
    \caption{Averaged cost function}
  \end{subfigure}

  \vspace{0.5em}

  \begin{subfigure}{0.45\textwidth}
    \includegraphics[width=\linewidth]{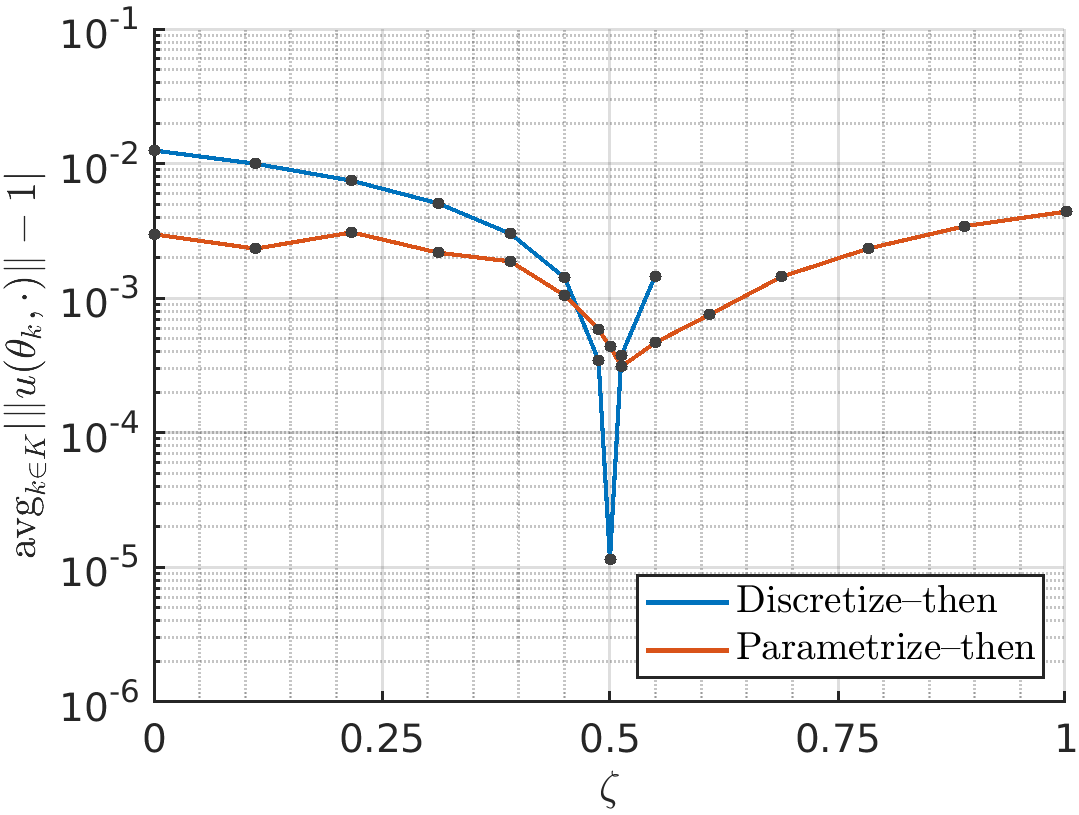}
    \caption{Average norm deviation}
  \end{subfigure}
  \hfill
  \begin{subfigure}{0.45\textwidth}
   \includegraphics[width=\linewidth]{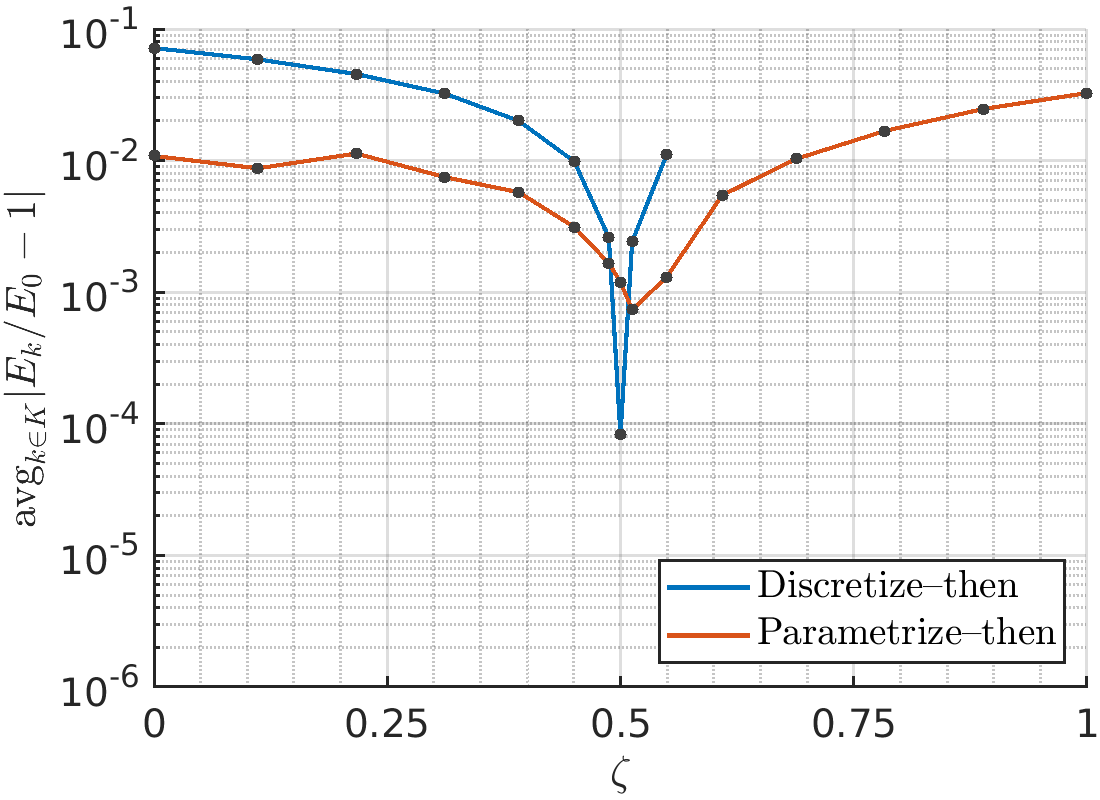}
    \caption{Average energy deviation}
  \end{subfigure}
  \caption{Hyperbolic cosine potential with one initial Gaussian. The approximate solution $u(\theta_k,\cdot)$ is built from  $L=3$ Gaussians, whose $\theta$--parameters are optimized for each time step $t_k$. The four plots show the outcome for the maximal and averaged cost function $F_k = \|r_k(\theta_{k+1},\cdot)\|^2$ as well as the average norm and energy deviation as a function of $\zeta\in[0,1]$.  }
  \label{fig:zeta-cosh}
\end{figure}

\paragraph{Double-well potential.}
In \Cref{fig:zeta-dw}, we repeat the same experiment for the double-well potential and the initial condition $\psi_0(x)=\mathcal{N}_0 e^{-(x-2)^2}$, which is centered at the bottom of the right well. Since a single Gaussian 
will not properly resolve the dynamics, we choose $L=3$ for the approximation. 
It can be observed that the attained values of the cost function are around $10^{-6}$ and $10^{-7}$.
As in the preceding example, the discretize-then approach will not converge in the vicinity of $\zeta=1$ for the chosen step-size $h=10^{-2}$. A smaller step-size is needed in such cases. 
We also observe that the optimal choice for $\zeta$ is again $\zeta\approx \frac12$. 
However, unlike the hyperbolic cosine example, the parametrize-then formulation does not perform better than discretize-then in all situations away from the midpoint. 
In this example, the discretize-then approach produces slightly smaller norm deviations for most values of $\zeta$ (except around $\zeta=\frac12$), but parametrize-then produces smaller energy deviations. 

\begin{figure}[htb!]
  \centering
  \begin{subfigure}{0.45\textwidth}
    \includegraphics[width=\linewidth]{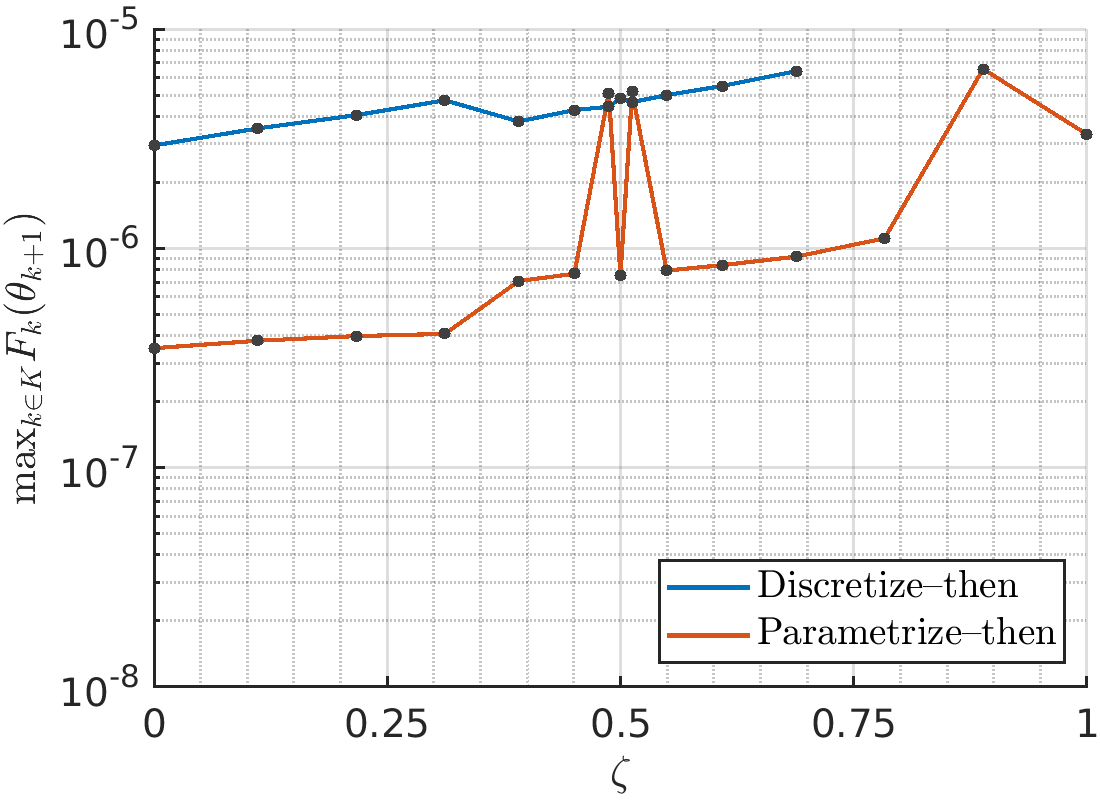} 
    \caption{Maximal cost function}
  \end{subfigure}
  \hfill
  \begin{subfigure}{0.45\textwidth}
    \includegraphics[width=\linewidth]{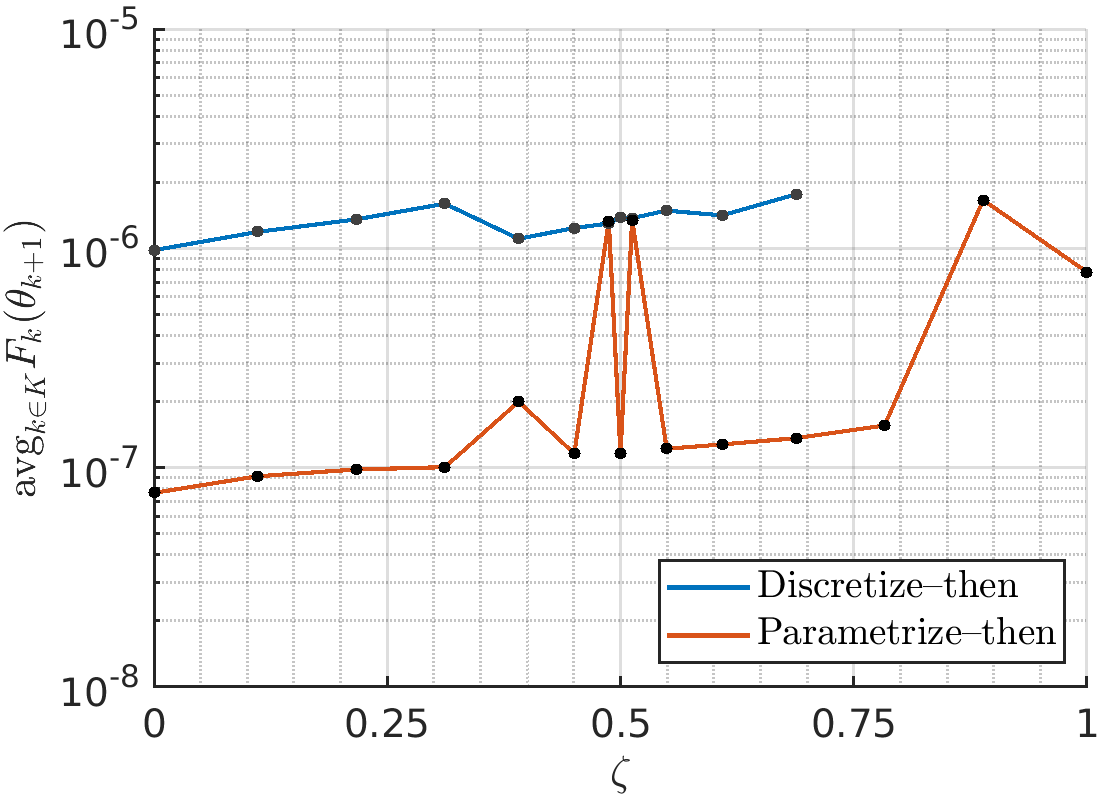}
    \caption{Average cost function}
  \end{subfigure}

  \vspace{0.5em}

  \begin{subfigure}{0.45\textwidth}
    \includegraphics[width=\linewidth]{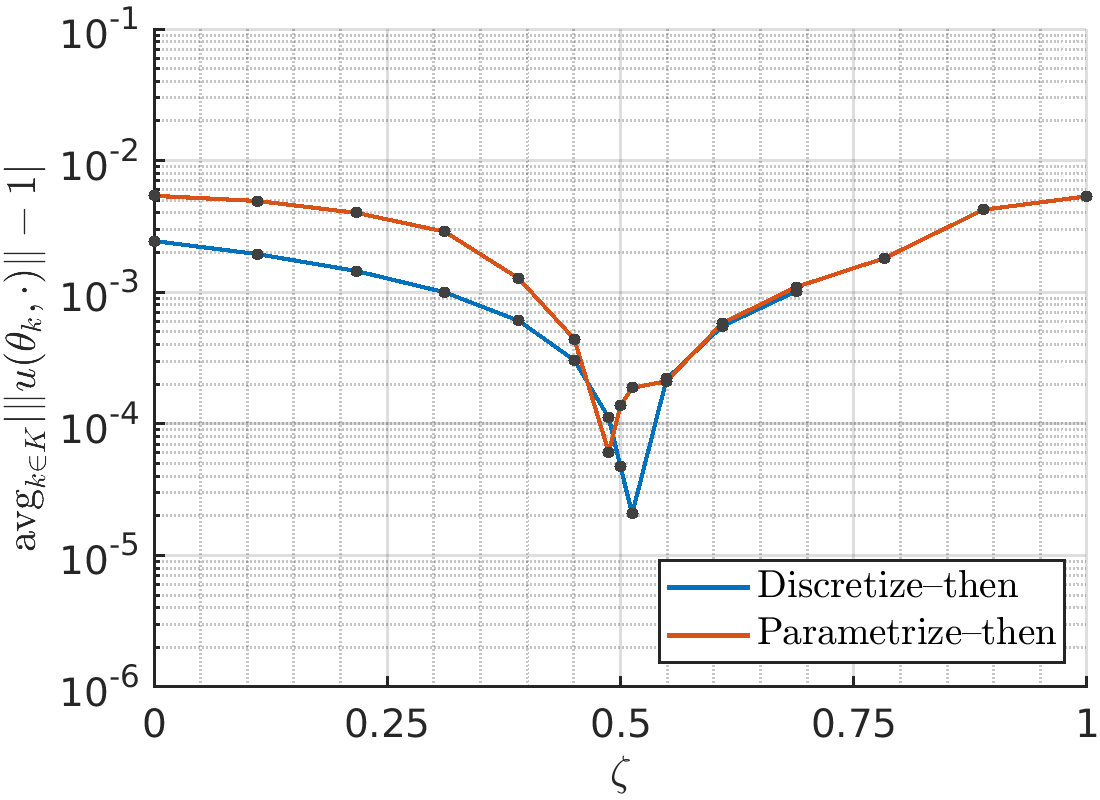}
    \caption{Average norm deviation}
  \end{subfigure}
  \hfill
  \begin{subfigure}{0.45\textwidth}
   \includegraphics[width=\linewidth]{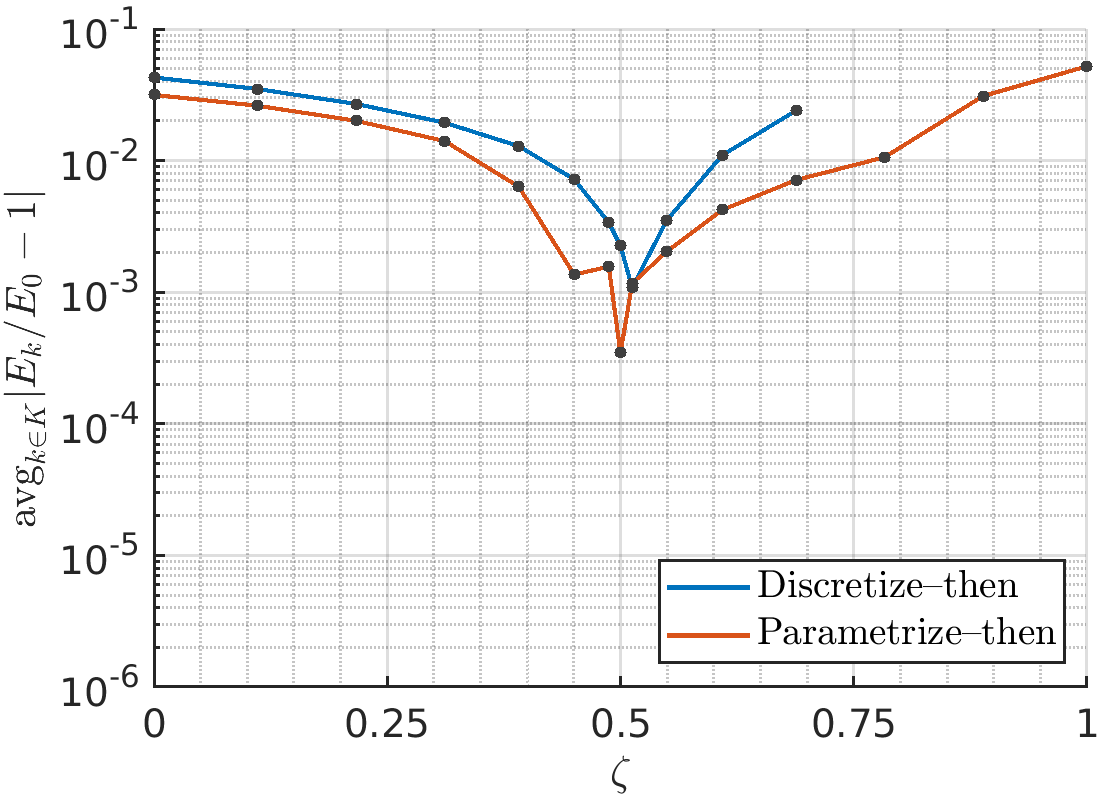}
    \caption{Average energy deviation}  
  \end{subfigure}

  \caption{Double-well potential with one initial Gaussian. The approximate solution $u(\theta_k,\cdot)$ is built from  $L=3$ Gaussians, whose $\theta$--parameters are optimized for each time step $t_k$. The plots show the time-averaged cost function $F_k = \|r_k(\theta_{k+1},\cdot)\|^2$ as well as the averaged norm and energy deviation as a function of $\zeta\in[0,1]$.  } \label{fig:zeta-dw}
\end{figure}

\paragraph{Conclusion.} In agreement with our theoretical findings, choosing $\zeta\approx \frac12$ is the preferred option for both schemes (discretize-then/parametrize-then), even in situations in which the cost function cannot be minimized to zero.
Moreover, the examples above show situations in which either of the schemes has a better performance. 
In particular, it seems that discretize-then is the optimal choice when $\zeta=\frac12$, whereas parametrize-then is the best choice when $\zeta=1$.

\subsection{Convergence with respect to the step-size}

We next investigate the dependence of the approximation error on the
time-step size $h$. According to Theorems~\ref{prop:discr-opt}
and~\ref{prop:opt-discr}, the global error consists of two
contributions: a time-discretization error determined by the chosen
$\zeta$-scheme and a residual contribution associated with the
restriction to the approximation manifold. The purpose of the present
experiment is to illustrate how these two contributions interact.

For the harmonic oscillator, the Gaussian manifold is invariant under
the exact dynamics. Consequently, the residual can be reduced to
machine precision, provided that the step-size is sufficiently small,
and the observed error is governed essentially by the time
discretization. In this residual-free regime, the expected orders are
visible: Euler-type choices of $\zeta$ lead to first-order convergence,
whereas the midpoint choice $\zeta=\frac12$ yields second-order
convergence (no numerical results presented here).

\begin{figure}[htb!]
    \begin{subfigure}{0.45\textwidth}
    \includegraphics[width=\linewidth]{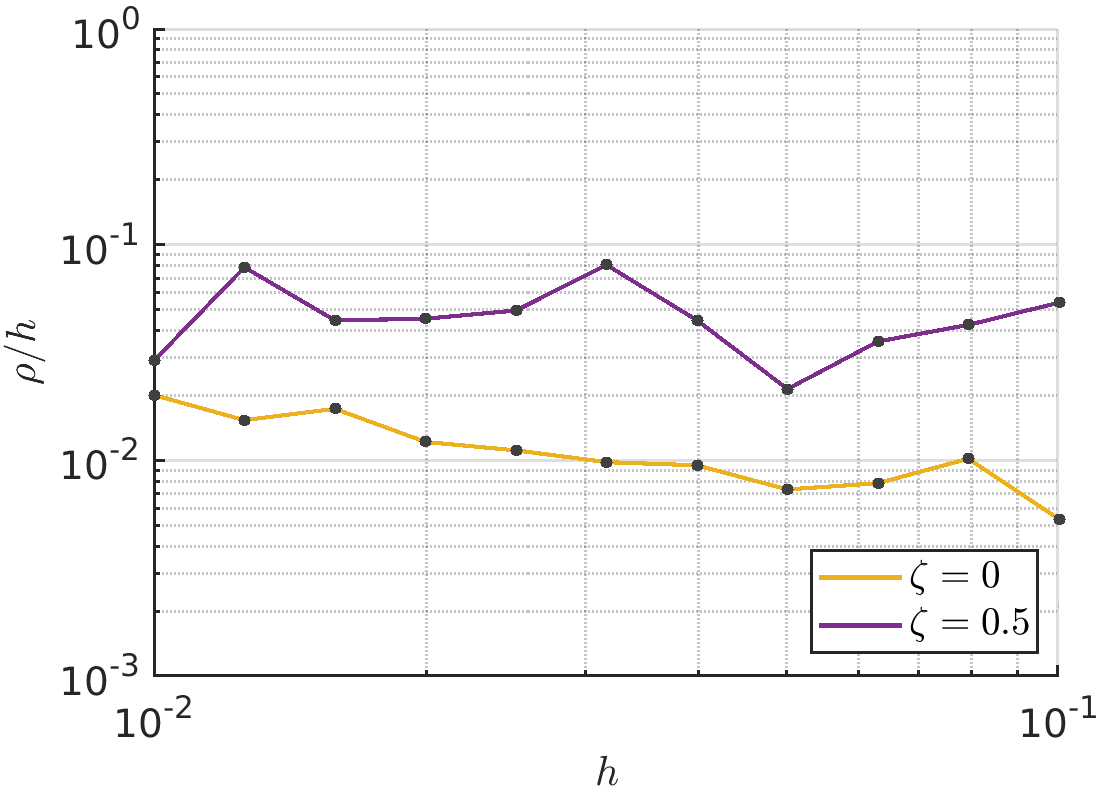}
    \caption{Discretize--then} 
    \end{subfigure}
\hfill
    \begin{subfigure}{0.45\textwidth}
    \includegraphics[width=\linewidth]{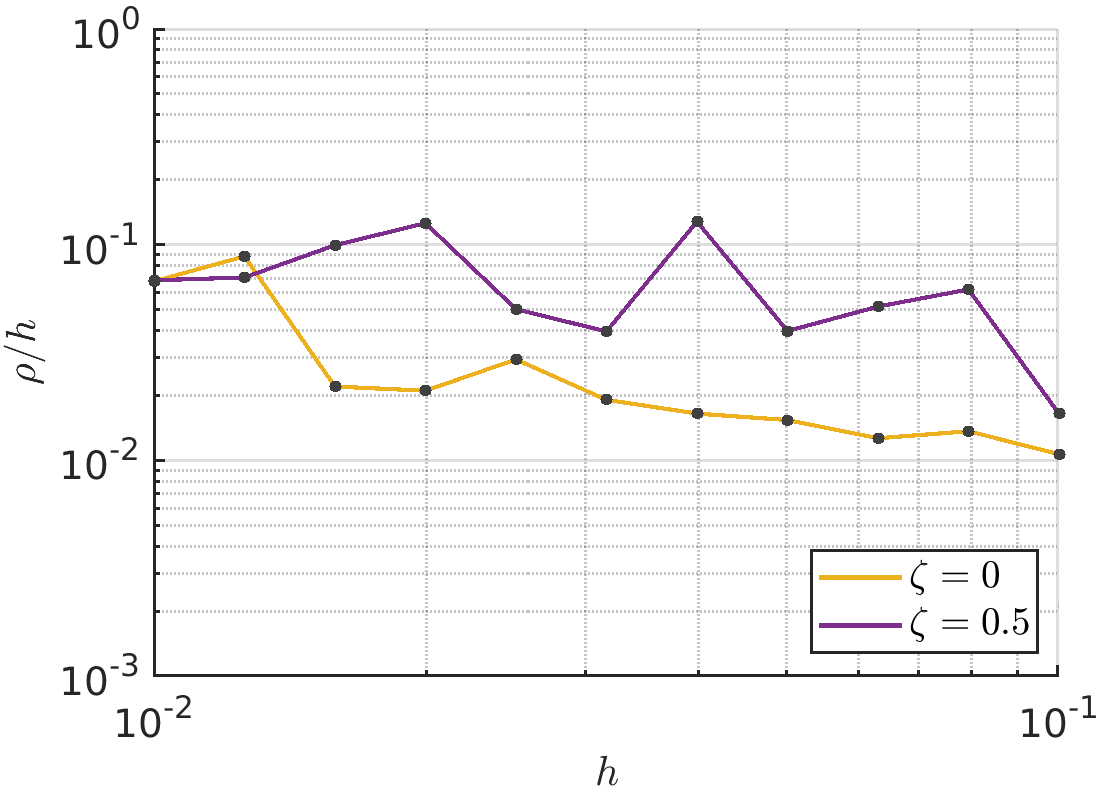}
    \caption{Parametrize--then}
    \end{subfigure}

    \vspace{0.5em}
    
    \begin{subfigure}{0.45\textwidth}
    \includegraphics[width=\linewidth]{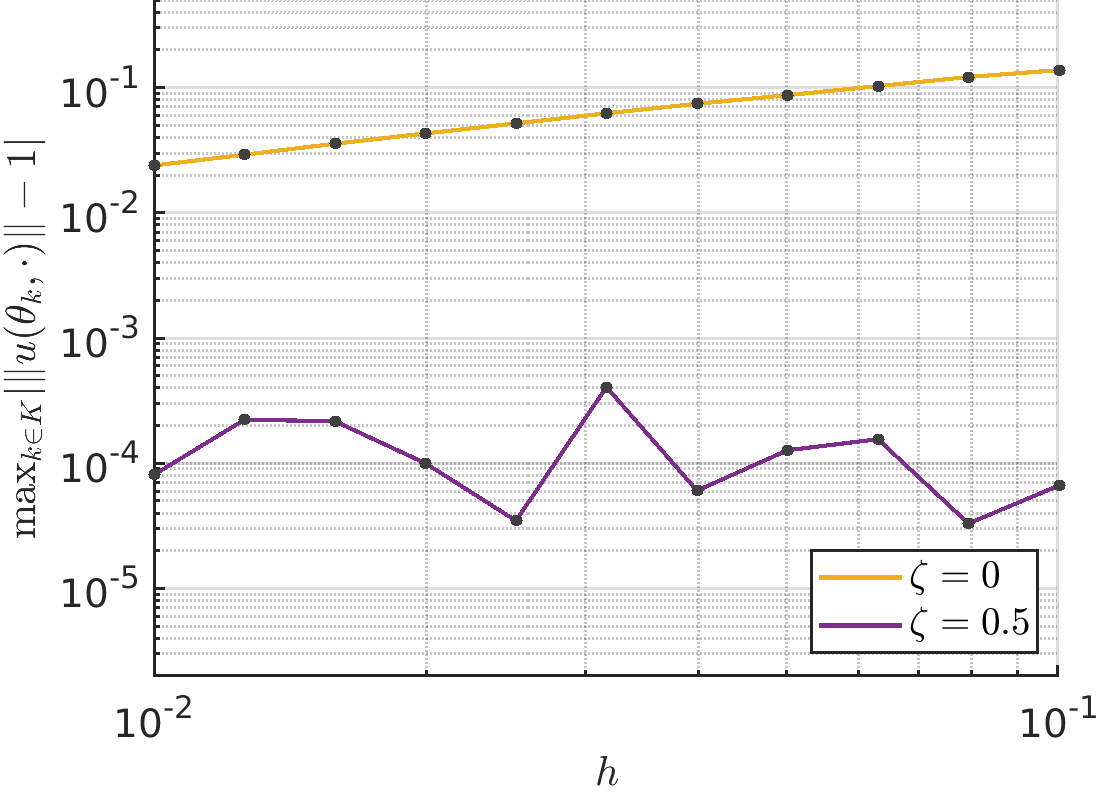}
    \caption{Discretize--then} 
    \end{subfigure}
\hfill
    \begin{subfigure}{0.45\textwidth}
    \includegraphics[width=\linewidth]{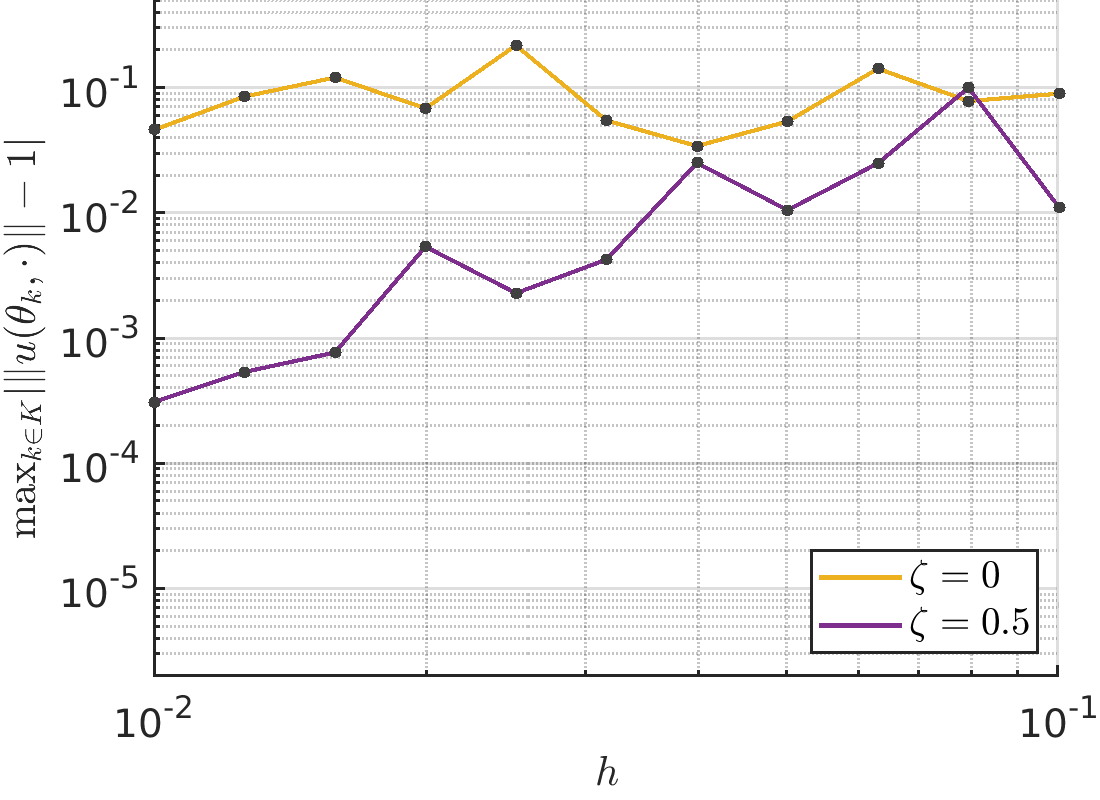}
    \caption{Parametrize--then}
    \end{subfigure}
    \caption{Hyperbolic cosine potential: scaled maximal residual and norm deviation (maximum) in dependence of the step-size $h$. Only the cases $\zeta=0$ and $\zeta= \frac12$ are shown. This example corresponds to a seven-Gaussian approximation ($L=7$).} \label{fig:convergence_wrt_h_cosh}
\end{figure}

The situation changes when the exact dynamics are no longer contained
in the Gaussian manifold. To illustrate this effect, we consider the 
hyperbolic cosine potential, for which the residual contribution in the
error estimates does not vanish. Figure~\ref{fig:convergence_wrt_h_cosh} shows 
the scaled residual $\rho/h = \max_{k\in K} \sqrt{ F_k(\theta_{k+1})}/h$ and
the maximal norm deviation
$\Delta u_{\max}:=
\max_{k\in K}
\bigl|\Vert u(\theta_k,\cdot)\Vert-1\bigr|$ 
as a function of the step-size $h$ for the discretize-then and
parametrize-then formulations. The computations use step-sizes
$h\in[10^{-2},10^{-1}]$ with equidistant values of $\log h$, final time
$T=2$, and Gaussian mixtures with $L=1,3,5,7$ components. Table~\ref{table:norm_wrt_h}
reports the corresponding regression parameters for the models
\[
\Delta u_{\max} \approx C_u h^{p_u}.
\]
The results show the expected first-order behavior for the implicit Euler 
case $\zeta=0$. The fitted exponents are close to one in most cases,
indicating that the contribution of the time discretization decreases
linearly with $h$. For $\zeta=\frac12$, the interpretation is more
delicate. The midpoint rule reduces the discretization error, but the
residual contribution does not vanish. Once this residual contribution
becomes comparable to, or larger than, the time-discretization error,
the second-order regime is partially masked. This explains why the
observed slopes do not always display a clean quadratic rate, even
though the midpoint rule still produces substantially smaller norm
deviations.

This behavior is particularly visible in the discretize-then
formulation, where for $\zeta=\frac12$ the norm deviation reaches its
minimal level already for relatively large step-sizes. Further
decreasing $h$ then produces little improvement, because the dominant
error is no longer the time-discretization error but the residual
contribution associated with the Gaussian manifold. In contrast, the
parametrize-then formulation shows a clearer decrease in some cases,
most notably for $L=7$, where the fitted exponent is close to two.
This indicates that, when the residual is sufficiently well controlled,
the second-order behavior of the midpoint rule can become visible.

\paragraph{Conclusion.} Overall, Figure~\ref{fig:convergence_wrt_h_cosh} and Table~\ref{table:norm_wrt_h}
are consistent with the error decomposition predicted by the theory.
A reduction of the step-size improves the approximation only until the residual
contribution becomes dominant. The midpoint choice $\zeta=\frac12$
remains advantageous, not necessarily because it always improves the
observed asymptotic order, but because it consistently reduces the
error constants and yields smaller norm deviations over the tested
range of step-sizes. Similar observations apply for the double-well dynamics (numerical experiments not presented here). 

\begin{table}[t]
\centering
\caption{Results for the hyperbolic cosine potential $V(x)=\cosh x -1$. Maximal residual
$\rho:=\max_{k\in K} \sqrt{ F_k(\theta_{k+1})}$ and linear regression for the maximal norm deviation
$\Delta u_{\rm max}:={\rm max}_{k\in K} |\Vert u(\theta_k,\cdot)\Vert-1|$.
The fitted models are
$\Delta u_{\rm max} \approx C_u h^{p_u}$.
The data are obtained from the log--log plots in \Cref{fig:convergence_wrt_h_cosh} for experiments with $L=1,3,5,7$ Gaussians.} \label{table:norm_wrt_h}
\renewcommand{\arraystretch}{1.15}
\begin{tabular}{ll|cc|cc}
\toprule
& & \multicolumn{2}{c|}{\textbf{Discretize--then}}
& \multicolumn{2}{c}{\textbf{Parametrize--then}} \\
\cline{3-6}
Problem & $\zeta$
& $\delta/h$
& $(C_u,p_u)$
& $\delta/h$
& $(C_u,p_u)$ \\
\midrule

\multirow{2}{*}{$L=1$}
& $0$
& $ 0.2248 $
& $( 0.8748 ,\; 0.7694)$
& $ 0.2379 $
& $( 0.6408 ,\; 0.8599)$
\\

& $\frac12$
& $ 0.2344 $
& $( 0.0122 ,\; 0.9291)$
& $ 0.2369 $
& $( 0.0667 ,\; 1.3296)$
\\
\midrule

\multirow{2}{*}{$L=3$}
& $0$
& $ 0.0124 $
& $( 0.8582 ,\; 0.7702 )$
& $ 0.0349 $
& $( 2.5724 ,\; 1.0896 )$
\\

& $\frac12$
& $ 0.0429 $
& $( 0.0033 ,\; 1.1364 )$
& $ 0.0483 $
& $( 1.4696 ,\; 1.6047 )$
\\
\midrule

\multirow{2}{*}{$L=5$}
& $0$
& $ 0.0109 $
& $( 0.8542 ,\; 0.7681 )$
& $ 0.0358 $
& $( 0.0918 ,\; 0.1535 )$
\\

& $\frac12$
& $ 0.0395 $
& $( 0.0008 ,\; 0.5993 )$
& $ 0.0542 $
& $( 1.0311 ,\; 1.4570 )$
\\
\midrule

\multirow{2}{*}{$L=7$}
& $0$
& $ 0.0107 $
& $( 0.8512 ,\; 0.7670 )$
& $ 0.0224 $
& $( 0.0917 ,\; 0.0459 )$
\\

& $\frac12$
& $ 0.0447 $
& $( 3\cdot 10^{-5} ,\; -0.3389 )$
& $ 0.0593 $
& $( 6.8412 ,\; 2.0920 )$
\\
\bottomrule
\end{tabular}
\end{table}

\subsection{Influence of quadrature accuracy}\label{sec:int_quad}
We compare the obtained results when the residual function is computed with closed-form formulas (only possible if the potential has a special form, for example, polynomial) and those obtained with Gauss--Hermite quadrature as described in \Cref{sec:GH}.
We choose the rule with $N=10$ nodes, which is exact for integrands that are a polynomial times a Gaussian for polynomial degree $\le 2N-1$. However, this assumes that the nodes and weights are exact, but in practice, they can only be obtained numerically. The residual evaluation involves repeated optimization steps and repeated quadrature evaluations. Small inaccuracies are therefore propagated through the nonlinear optimization process and accumulate over time.

In the harmonic oscillator example ( \Cref{fig:ex-vs-quad-a,fig:ex-vs-quad-b}), the quadrature and exact evaluation results are almost indistinguishable throughout the simulation interval.
In the double-well example (\Cref{fig:ex-vs-quad-e,fig:ex-vs-quad-f}), small discrepancies appear rather early and slightly accumulate over time. 
All double-well results 
reflect the greater sensitivity of the generic case, when the approximation manifold is not invariant under the dynamics.

These experiments indicate that, in one spatial dimension, a $10$-point 
Gauss--Hermite quadrature is sufficiently accurate to reproduce both the residual evolution and the conservation properties obtained with exact Gaussian formulas. Nevertheless, exact formulas remain preferable whenever available, since they eliminate quadrature error entirely and become increasingly advantageous in higher-dimensional settings.

\begin{figure}[htb!]
    \centering
    
    \begin{subfigure}{0.45\textwidth}
    \includegraphics[width=\linewidth]
    {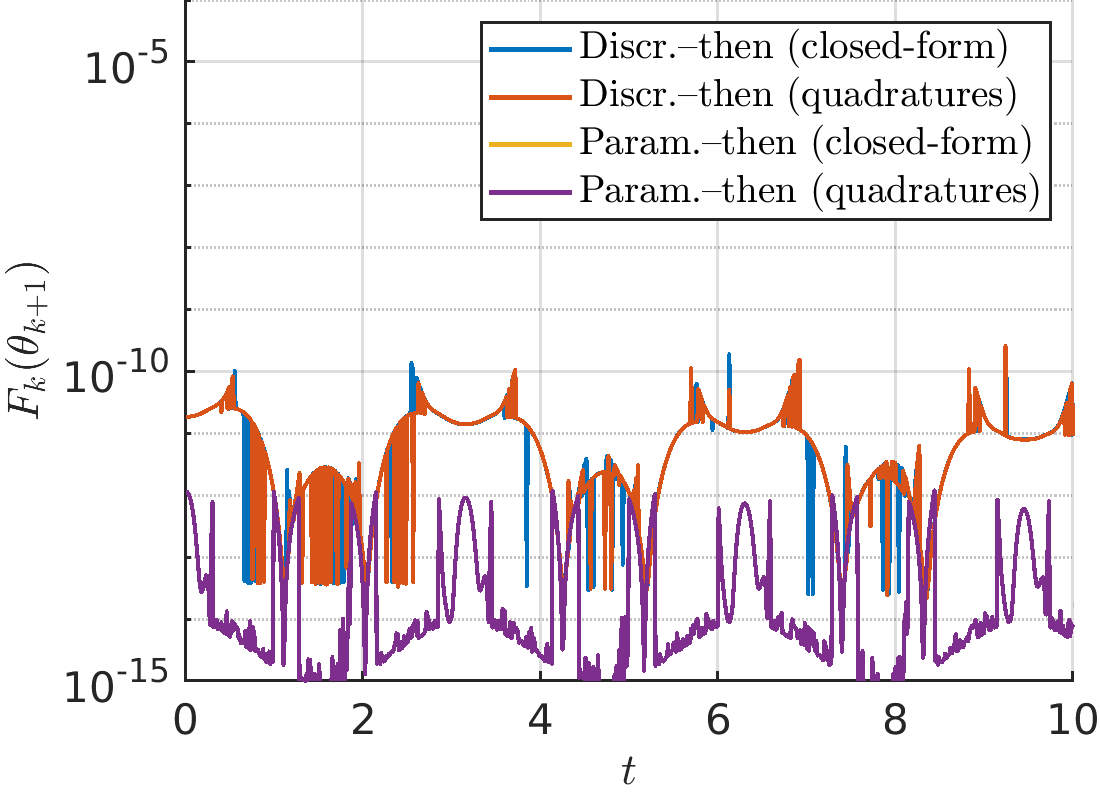}
    \caption{Harmonic oscillator: cost function} \label{fig:ex-vs-quad-a}
    \end{subfigure}
\hfill
    \begin{subfigure}{0.45\textwidth}
    \includegraphics[width=\linewidth]{{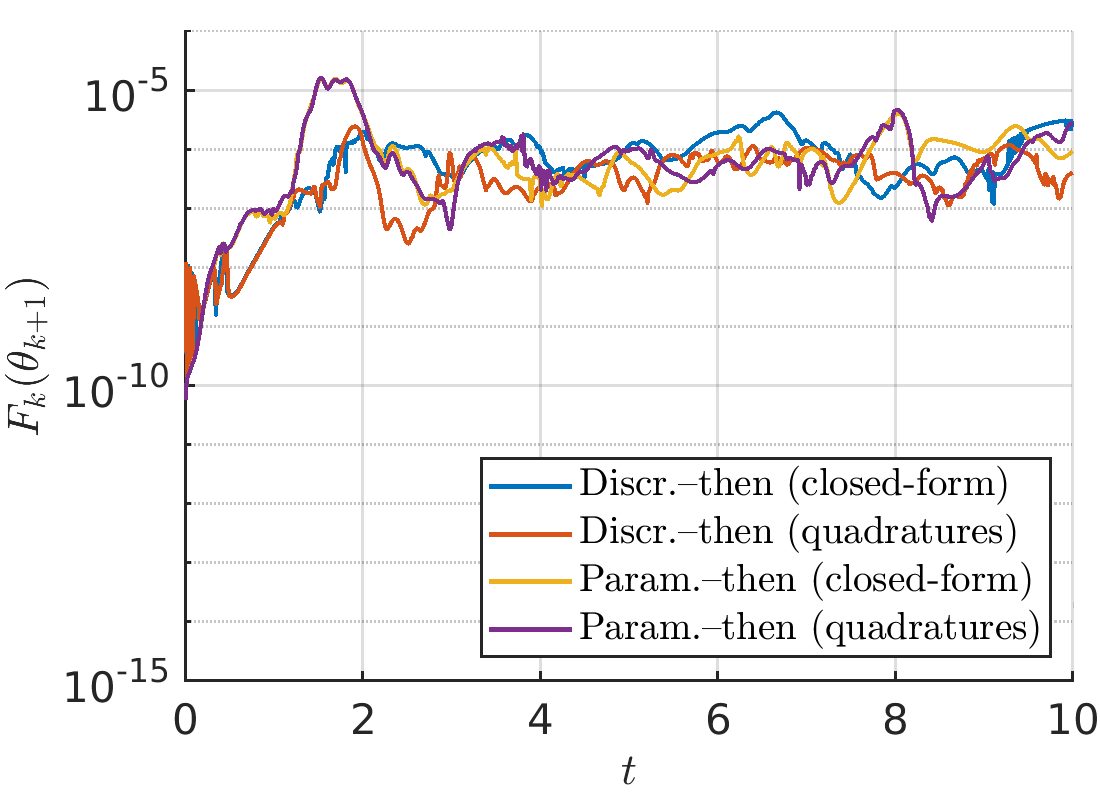} } 
    \caption{Double-well: cost function} \label{fig:ex-vs-quad-e}
    \end{subfigure}

    \vspace{0.5em}

    \begin{subfigure}{0.45\textwidth}
    \includegraphics[width=\linewidth]{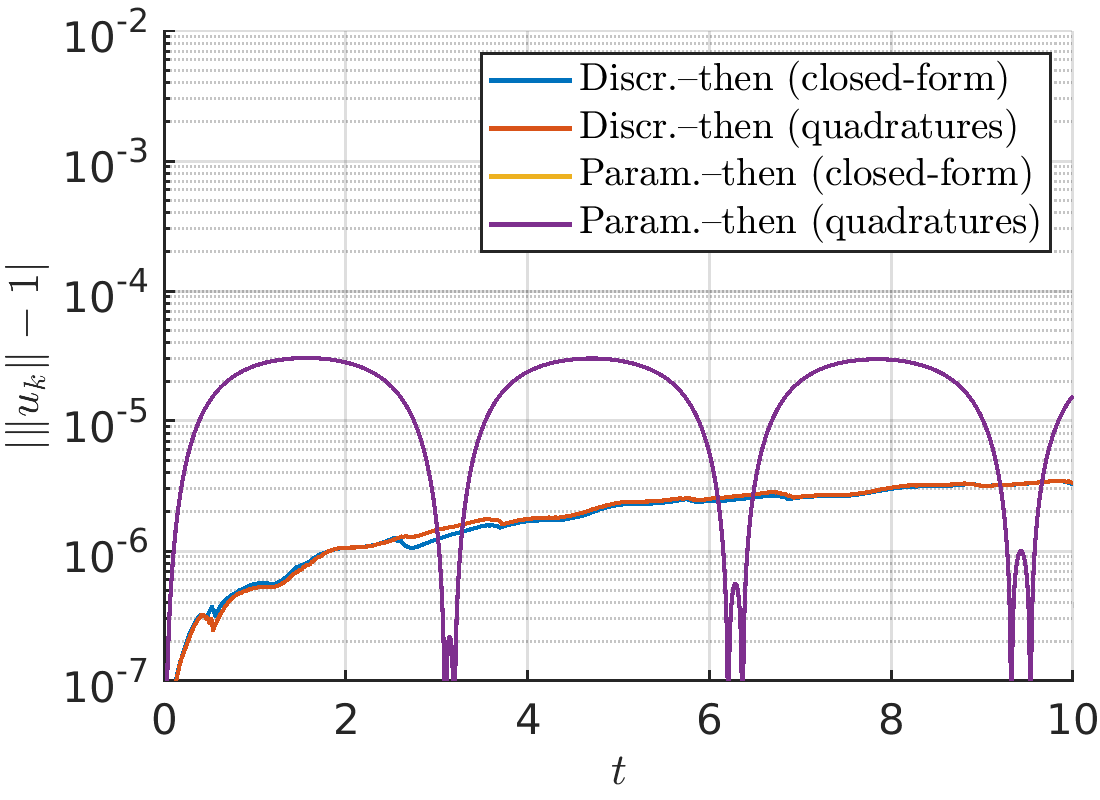}
    \caption{Harmonic oscillator: norm deviation} \label{fig:ex-vs-quad-b}
    \end{subfigure}
\hfill
    \begin{subfigure}{0.45\textwidth}
    \includegraphics[width=\linewidth]{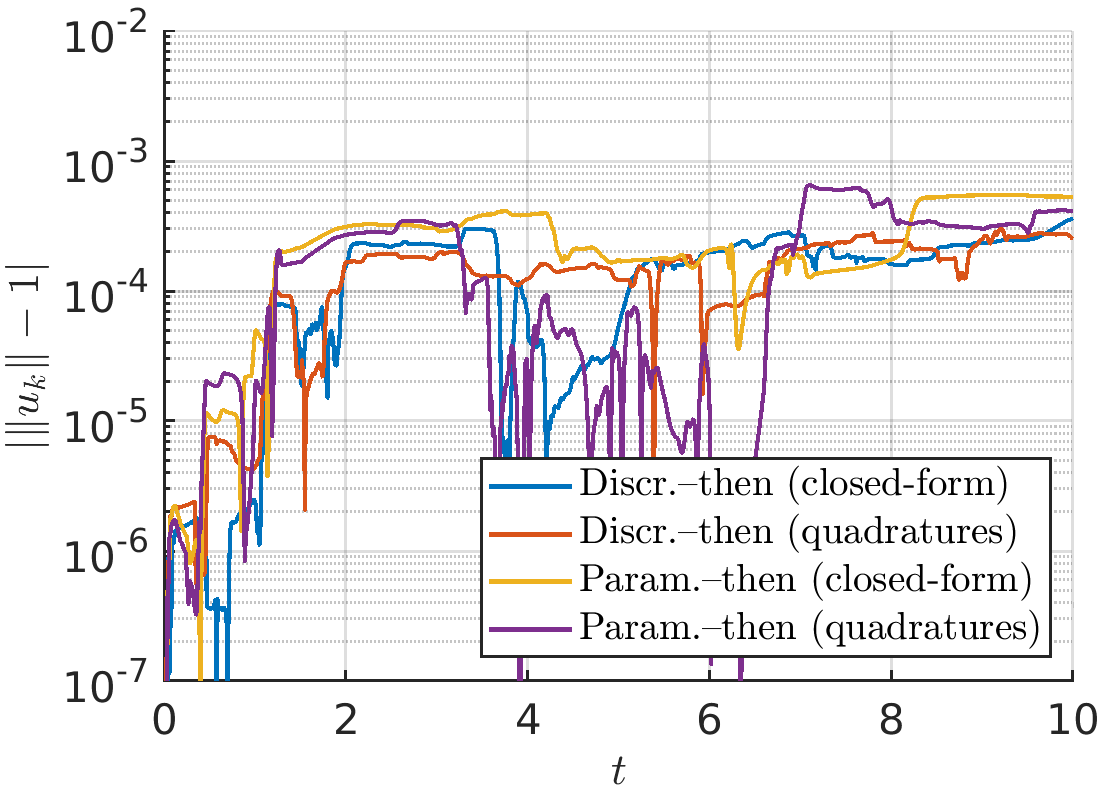} 
    \caption{Double-well: norm deviation} \label{fig:ex-vs-quad-f}
    \end{subfigure}
    
    \caption{Attained squared residuals $F_k(\theta_{k+1}) = \|r_k(\theta_{k+1},\cdot)\|^2$ and the norm deviation in dependence of time for: (i) shifted second excited state of the harmonic oscillator potential with a three-Gaussian approximation (left column), (ii) six-Gaussian approximation in a double-well potential (right column). Numerical parameters: $\zeta=\frac12$ and $h=10^{-2}$. } \label{fig:exact-vs-quadratures}
\end{figure}

\section{Conclusions}\label{sec:concl}

We have presented a unified analytical framework for time-discrete parametric approximations of evolution equations in Hilbert spaces based on residual minimization for nonlinear approximation manifolds. The formulation encompasses two natural paradigms: discretization followed by parametrization of the evolution equation (Rothe approach), and discretization of the variational dynamics induced by the Dirac--Frenkel principle. Both strategies are expressed in terms of residual norms whose minimization defines the time-stepping scheme in parameter space.

The main theoretical contribution is a unified error analysis for the full family of $\zeta$-methods, extending previous results in \cite{ZCVP} and complementing those for regularized parametric approximation \cite{FLL,LN}. The resulting bounds separate the effects of time discretization from those arising by residual minimization. This structure is independent of the particular form of the underlying evolution equation and relies only on general Lipschitz-type or dissipativity properties. For the parametrize-then formulation, additional conditioning requirements on the parametrization appear naturally and explain stability constraints encountered in nonlinear manifold approximations.
For the discretize-then approach with implicit Euler, the condition can be relaxed to a one-sided Lipschitz condition. Discretize-then with $\zeta<\frac12$ and a one-sided Lipschitz condition with $\ell=0$ allows proving analogous bounds.

While Gaussian approximation manifolds provide a concrete and computationally attractive realization of the abstract framework, the analysis itself is not tied to any particular functional form. The residual-based viewpoint applies equally to a broad range of nonlinear parametrizations, including reduced-basis spaces, adaptive dictionaries, and data-driven manifolds. The numerical experiments for time-dependent Schr\"odinger equations serve primarily as illustrative examples of the general theory. They confirm the predicted convergence orders and highlight the role of residual accuracy. The Fokker--Planck example presented in Appendix~~\ref{sec:FP_eq} further illustrates that the residual-based approximation is not tied to
Hamiltonian dynamics and can be applied to dissipative evolution equations as well.

The flexibility of the framework also suggests several directions for future research. In particular, adaptive approximation manifolds appear especially promising from a computational perspective. Rather than working with a fixed approximation space, one may allow the manifold to evolve together with the solution, enriching or simplifying the representation whenever required by the local approximation quality. In the Gaussian setting, this corresponds to dynamically adjusting the number of basis functions during the computation. More generally, the residual-based formulation developed here may provide a natural foundation for the analysis of adaptive parametrizations in a variety of reduced-order and data-driven approximation methods.

\section*{Acknowledgements}
Funded by the Deutsche Forschungsgemeinschaft (DFG, German Research Foundation) -- TRR~352 -- Project-ID 470903074.

\section*{Data Availability}
All code used for the qualitative and quantitative numerical experiments is openly available at \url{https://gitlab.lrz.de/00000000014C092A/discretize-then-and-parametrize-then}.

\appendix
\section{Local error estimate for $\zeta$-methods}

\begin{lemma}[Local error of $\zeta$-methods]\label{lem:local_error}
The local error
\[
 e_k = \psi(t_{k+1}) - \psi(t_k) - h \left( \zeta f(t_k,\cdot,\psi(t_k)) + \hat\zeta f(t_{k+1},\cdot,\psi(t_{k+1})) \right)
\]
satisfies
\[
\|e_{k}\| \le 
\left\{\begin{array}{ll}
h^2 \left( |\tfrac12-\zeta| + \tfrac14\right) \max_{t\in[0,T]}\|\psi''(t)\|, & \text{if}\ \psi\in C^2([0,T],\mathcal H),\\*[2ex]
\frac{1}{12}h^3 \max_{t\in[0,T]}\|\psi'''(t)\|, &
\text{if}\ \zeta = \frac12\ \text{and}\ \psi\in C^3([0,T],\mathcal H),
\end{array}
\right.
\]
for all $k$ such that $t_{k+1}=(k+1)h\le T$.
\end{lemma}

\begin{proof}
We use the linear interpolant 
$\ell(t) = \psi'(t_k) + \frac{t-t_k}{h}(\psi'(t_{k+1})-\psi'(t_k))$ 
of $\psi'(t)$ to write the local error as
\begin{align*}
    e_k &= \int_{t_k}^{t_{k+1}} \psi'(t) dt - h\left(\zeta \psi'(t_k) + \hat\zeta \psi'(t_{k+1})\right)\\
    &= \int_{t_k}^{t_{k+1}} \left(\psi'(t)-\ell(t)\right) dt + \frac{h}{2}\left(\psi'(t_{k+1})+\psi'(t_k)\right)
    - h\left(\zeta \psi'(t_k) + \hat\zeta \psi'(t_{k+1})\right)\\
    &= \int_{t_k}^{t_{k+1}} \left(\psi'(t)-\ell(t)\right) dt - h(\tfrac12-\zeta)\int_{t_k}^{t_{k+1}}\psi''(t) \, dt\\
    &=: I_1 + I_2.
\end{align*}
The first integral is the error of the single-step trapezoidal rule. If $\psi\in C^3(\R,\mathcal H)$, it  
satisfies the classical estimate
\[
\| I_1\| \le \frac{h^3}{12}\max_{t\in[0,T]}\|\psi'''(t)\|.
\]
Alternatively, see \cite{CruzNeugebauer02}, 
\[
I_1 = \int_{t_{k}}^{t_{k+1}} (t-m_k) \psi''(t) dt, 
\]
where $m_k = \frac12(t_k+t_{k+1})$ is the midpoint. This implies
$\|I_1\| \le \frac{h^2}{4} \max_{t\in[0,T]}\|\psi''(t)\|$.
For the second integral, we clearly have $\|I_2\| \le h^2|\tfrac12-\zeta|\max_{t\in[0,T]}\|\psi''(t)\|$.
\end{proof}

\section{Nonlinear least squares estimate}

\begin{lemma}[Nonlinear least squares problem]\label{lem:least}
Consider continuous $A:\C^p\to L(\C^p,\hil)$ and $b:\C^p\to\hil$ such that there exist $\sigma_0>0$ and 
$\lambda>0$ such that for all $h>0$ and $\eta,y\in\C^p$
\[
\sigma_{\min}(A(h\eta))\ge\sigma_0\quad \text{and}\quad \Vert b(\eta)- b(y)\Vert \le \lambda \Vert\eta-y\Vert.
\]
If $h\le \sigma_{0}/\lambda$, then the nonlinear least squares problem 
$\Vert A(h\eta)\eta - b(h\eta)\Vert = \min_{\eta\in\C^p} !$ has at least one solution $\theta(h)$, and each of 
these solutions satisfy
\[
\Vert \theta(h)\Vert \le \frac{2\Vert b(0)\Vert}{\sigma_{0}-\lambda h}. 
\]
\end{lemma}

\begin{proof}
We denote $G(h,\eta) = \Vert A(h\eta)\eta - b(h\eta)\Vert$. We use the lower bound $\Vert A(h\eta)\eta\Vert\ge\sigma_{0}\Vert \eta\Vert$ and the Lipschitz bound 
$\Vert b(h\eta)\Vert \le \Vert b(0)\Vert + \lambda h\Vert\eta\Vert$ to deduce
\begin{align*}
    G(h,\eta) &\ge \Vert A(h\eta)\eta\Vert - \Vert b(h\eta)\Vert\\
    &\ge \left(\sigma_{0} - \lambda h \right)\Vert\eta\Vert - \Vert b(0)\Vert.
\end{align*}
Since $h<\sigma_{0}/\lambda$, we have $G(h,\eta)\to\infty$ for $\Vert\eta\Vert\to\infty$, and coercivity guarantees the existence of at least one minimizer $\theta(h)$. Such a minimizer satisfies
\begin{align*}
    \left(\sigma_{0} - \lambda h \right)\Vert\theta(h)\Vert - \Vert b(0)\Vert &\le G(h,\theta(h)) 
    \le G(h,0) = \Vert b(0)\Vert,
\end{align*}
which implies the claimed bound on $\Vert\theta(h)\Vert$. 
\end{proof}

\section{Gaussian integrals}\label{app:Gauss}

\subsection{Monomial formula}
For proving Formula \ref{prop:n_gaussian_integral}, we compute
\begin{align*}
    G_n(\theta) &= \int_\R x^n e^{-a x^2 +\kappa x + \gamma } \dd x
     \ =\  \frac{\partial^n}{\partial \kappa^n} \int_\R e^{-a x^2 +\kappa x + \gamma } \dd x \\
     &\stackrel{\eqref{eq:G0}}{=} e^{\gamma } \sqrt{\frac{\pi}{a}} \frac{\partial^n}{\partial \kappa^n} e^{\frac{\kappa^2}{4a} } 
     \ = \  (2\sqrt{a})^{-n} \tilde{H}_n\left( \frac{\kappa}{2\sqrt{a}} \right)  \sqrt{\frac{\pi}{a}} e^{\frac{\kappa^2}{4a} +\gamma } ,
\end{align*}    
where we have used the Rodrigues formula for the Hermite polynomials.
For the gradient we simply observe
\begin{align*}
    \nabla_{\theta} G_{n}(\theta) = \left( \begin{array}{c}
       \partial_a G_n(\theta) \\
       \partial_\kappa G_n(\theta) \\
       \partial_\gamma G_n(\theta)
    \end{array} \right)^\trans  
    =  \left( \begin{array}{c}
       - \int_{\R} x^{n+2} g(\theta,x) \dd x\\
       \int_{\R} x^{n+1} g(\theta,x) \dd x \\
       \int_{\R} x^n g(\theta,x) \dd x
    \end{array} 
    \right)^\trans .
\end{align*}

\subsection{Gauss--Hermite quadrature}
The nodes $x_1,\ldots,x_n$ of the Gauss--Hermite quadrature rule \eqref{eq:GH_quadrature} are the roots of the $n$th Hermite polynomial $H_n(x)$ and the weights are
\begin{equation*}
    w_j = \frac{2^{n+1} n! \sqrt{\pi}}{[H_{n+1}(x_j)]^2},\quad j=1,\ldots,n. 
\end{equation*}
The classic error representation for a $2n$ times continuously differentiable integrand function $\varphi$ is
\begin{equation*}
    \int_{\R} \varphi(x) e^{-x^2}\dd x = \mathcal{Q}_n(\varphi) + \frac{n!\sqrt{\pi}}{2^n (2n)!} \varphi^{(2n)}(\xi),
\end{equation*}
for some $\xi\in\R$, see e.g. \cite[\S3.6]{DR}.


\section{Numerical experiments for Fokker--Planck equations} \label{sec:FP_eq}

The purpose of this appendix is merely to illustrate that the
residual formulations developed in the main
text are not restricted to Schr\"odinger equations. We thus consider the Fokker--Planck equation with time-dependent coefficients
\[
\partial_t \psi (t,x) = D(t) \Delta \psi(t,x) - \nabla (w(x,t) \psi(x,t)),
\]
where $D(t)$ is the diffusion coefficient and $w(x,t)=A(t)x+b(t)$ is the drift coefficient. 
It is clear that the operator $f(t,x,\cdot)=D(t) \Delta_x - \nabla_x (w(x,t) \cdot)$ acts on the Gaussian basis linearly, i.e., $f(t,x,b(\theta,x)) = \phi(\theta,t,x)b(\theta,x)$. For the univariate case $d=1$,
the function $\phi(\theta,t,x)$ is given by
\[
\phi(\theta,t,x) = D(t) [ -2a + (-2a x + \kappa)^2 ] - A(t) - w(x,t) (-2a x + \kappa).
\]

\Cref{fig:FP} shows the evolution of a single Gaussian wave packet $\psi_0(x) = \mathcal{N}_0 e^{-(x+1)^2}$ in four different regimes:
a) $A=0$, $b=0$, $D=0.02$; 
b) $A=0.02$, $b=0$, $D=0$; 
c) $A=0$, $b=4/10+\sin(2\pi t)$, $D=0$;
d) $A=0.01$, $b=4/10+\sin(2\pi t)$, $D=0.01$.
We recall that b) is the diffusion equation and c) is the transport equation. As in the examples in \Cref{sec:Num_Experiments}, all the numerical solutions were computed using the discretize-then formulation with the parameters $\zeta=\frac12$ and $h=10^{-2}$.

\begin{figure}[htb!] 
    \centering 
    \begin{subfigure}{0.45\textwidth}
    \includegraphics[width=\linewidth]{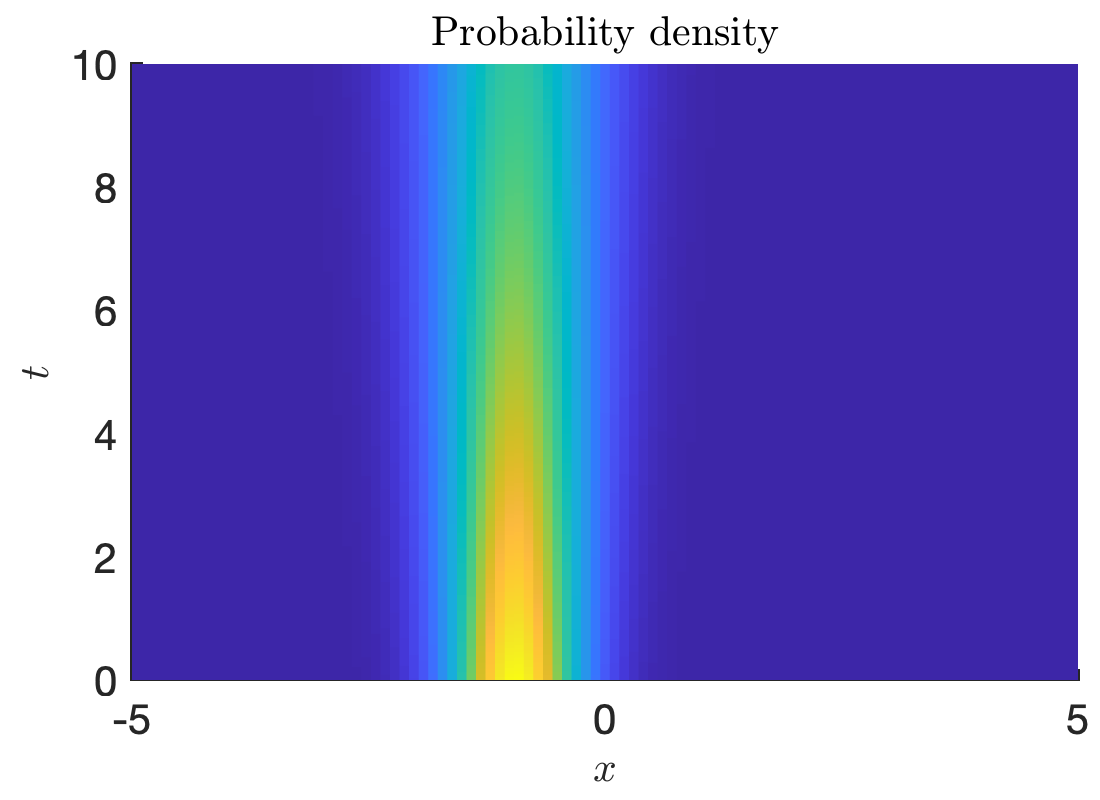}
    \caption{$A=0$, $b=0$, $D\neq 0$}
    \end{subfigure}
\hfill
    \begin{subfigure}{0.45\textwidth}
    \includegraphics[width=\linewidth]{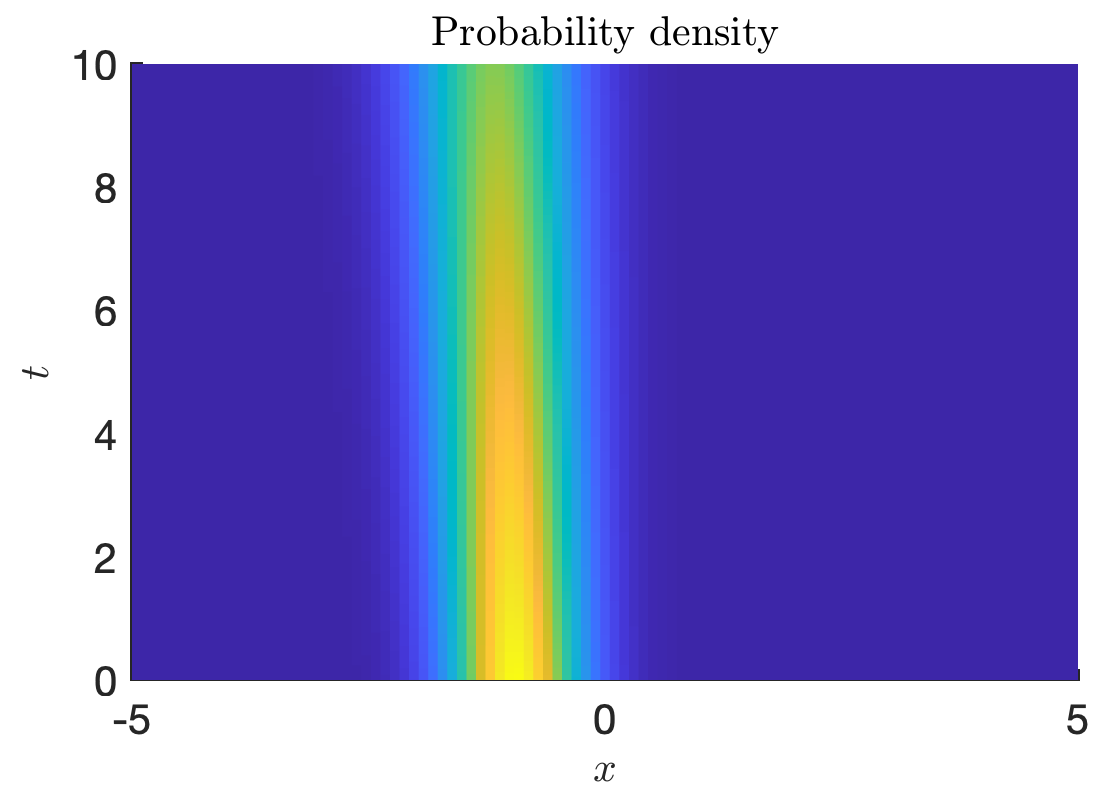}
    \caption{$A\neq 0$, $b=0$, $D=0$}
    \end{subfigure}

    \vspace{0.5em}

    \begin{subfigure}{0.45\textwidth}
    \includegraphics[width=\linewidth]{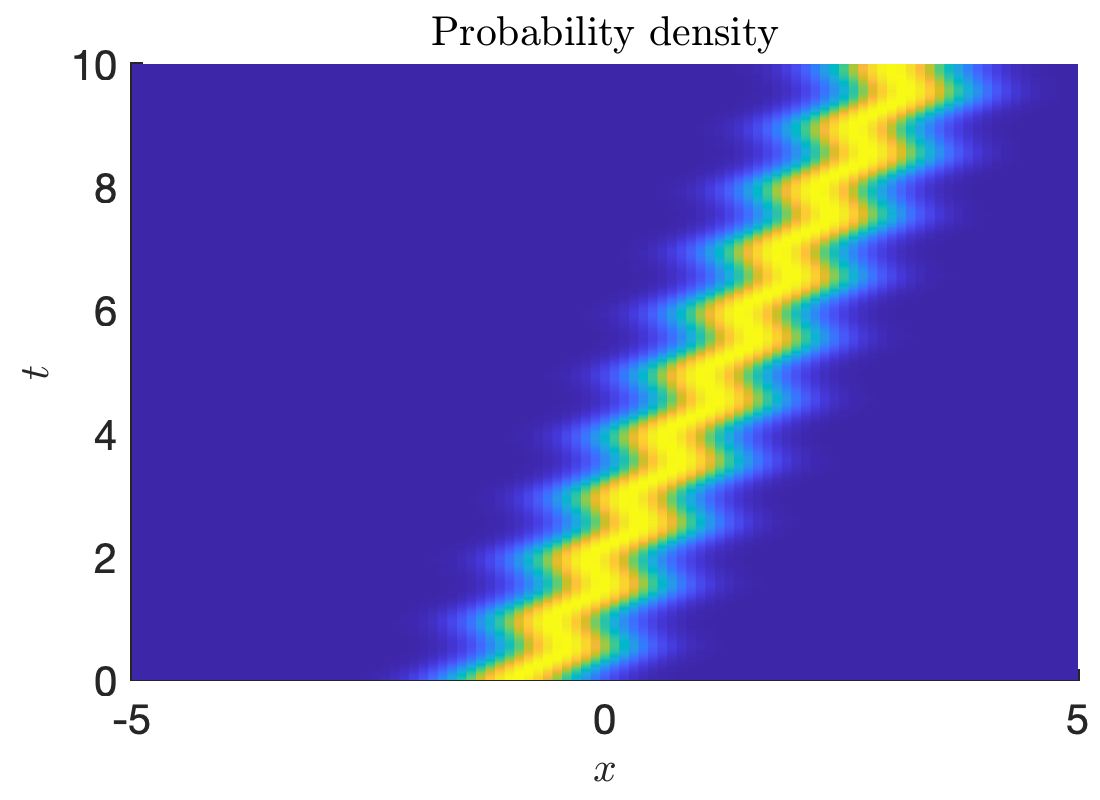}
    \caption{$A=0$, $b\neq 0$, $D=0$}
    \end{subfigure}
\hfill
    \begin{subfigure}{0.45\textwidth}
    \includegraphics[width=\linewidth]{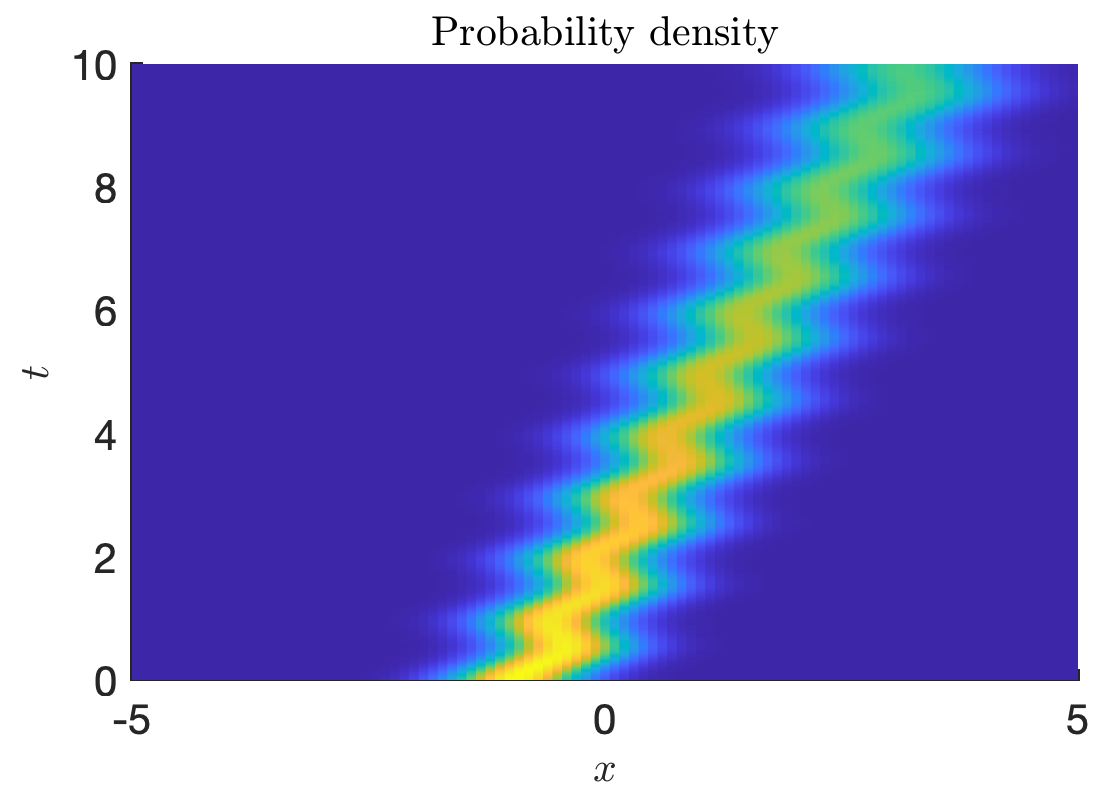}
    \caption{$A\neq 0$, $b\neq 0$, $D\neq 0$ \label{fig:FP4}}
    \end{subfigure}

    \caption{Evolution of a single Gaussian wave packet following the Fokker--Planck equation in four different regimes. The precise values of the non-zero coefficients are indicated in the text. } \label{fig:FP}
\end{figure}

In this case, the numerical approximation can be compared with the analytical solution. \Cref{fig:FP_difference} shows the difference $\Vert \psi(t_k,\cdot)-u(t_k,\cdot)\Vert^2_{L^2}$ in function of time, corresponding to the example in \Cref{fig:FP4}, which will decrease with the step-size until reaching the limit due to round-off errors.
For the analytical solution, recall that given a single Gaussian wave packet as the initial condition, the solution will remain a single Gaussian.
We write this solution in the form $\psi(t,x)=e^{-\frac{1}{S}(x-\nu)^2+\tilde{\gamma}}$ with coefficients $\tilde{\theta} =(S,\nu,\tilde{\gamma})$. Then, $\psi$ is described by the solutions of the ODE system
\begin{align*}
    S'(t) = 4D+2AS, \quad
    \nu'(t) = A\nu + b, \quad
    \tilde{\gamma}'(t) = -A-2D/S.
\end{align*}
Moreover, assume that the coefficient $A\geq 0$ is constant.
Then, the solution of the above ODE system with initial conditions $\tilde{\theta}_0=(S_0,\nu_0,\tilde{\gamma}_0)$ is
\begin{align*}
S(t) &= \left\{ \begin{array}{ll}
    S_0 + 4 \int_0^t D(t') dt', & \text{if } A=0,  \\
    e^{2At} \left[ S_0 + 4 \int_0^t D(t') e^{-2At'} dt' \right], & \text{if } A>0.
\end{array} \right. 
\\
\nu(t) &= \left\{ \begin{array}{ll}
    \nu_0 + \int_0^t b(t') dt', & \text{if } A=0,  \\
    e^{At} \left[ \nu_0 + \int_0^t b(t') e^{-At'}dt' \right] , & \text{if } A>0,
\end{array} \right. 
\\
\tilde{\gamma}(t) &= \tilde{\gamma}_0 -\frac{1}{2} \ln S(t) , \quad \text{if } A\geq 0,
\end{align*}
For instance, considering constant coefficients $A,D>0$ and the time-dependent function $b(t)=\alpha+\beta\sin(\omega t)$, the solutions would be
\[
S(t) = e^{2At} \left[ S_0 + \frac{2D}{A}\left( 1-e^{-2At}\right) \right], 
\]
\[
\nu(t) = e^{At} \left[ \nu_0 + \frac{\alpha}{A} (1-e^{-At}) + \frac{\beta\omega}{A^2+\omega^2} \right] - \frac{\beta}{A^2+\omega^2}\left[ \omega\cos(\omega t) + A\sin(\omega t) \right] .
\]

\begin{figure}[htb!]
    \centering
    \includegraphics[width=0.5\linewidth]{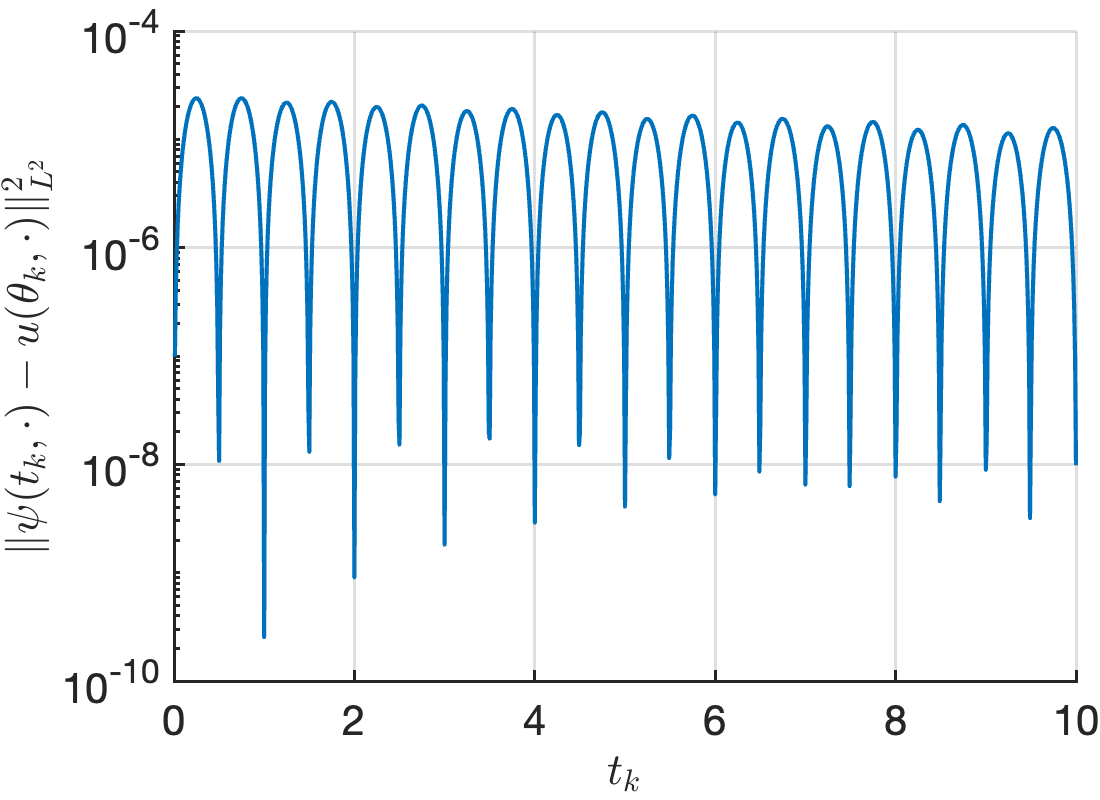}
    \caption{Squared $L^2$-norm of the difference between the exact solution $\psi(t,\cdot)$ and the numerical solution $u(t,\cdot)$ for $t\in [0,10]$, corresponding to the example with $A\neq 0$, $b\neq 0$, $D\neq 0$, see \Cref{fig:FP4}.} \label{fig:FP_difference}
\end{figure}

\printbibliography
\end{document}